\documentclass{amsart}

\usepackage{graphicx}
\usepackage{subfigure}
\usepackage{amsmath}
\usepackage{amsthm}
\usepackage{amsfonts}

\usepackage{hyperref}
\usepackage{todonotes}
\usepackage{verbatim} 

\newtheorem{theorem}{Theorem}[section]
\newtheorem{lemma}[theorem]{Lemma}

\theoremstyle{definition}
\newtheorem{definition}[theorem]{Definition}

\newtheorem{conjecture}[theorem]{Conjecture}

\theoremstyle{remark}

\numberwithin{equation}{section}

\begin{document}

\title[Laplacians on Snowflake Domains and Filled Julia Sets]{Spectral Properties of Laplacians on Snowflake Domains and Filled Julia Sets}

\author{Robert S. Strichartz}
\address{Department of Mathematics, Cornell University, Malott Hall, Ithaca, NY 14853, USA}
\email{str@math.cornell.edu}

\author{Samuel C. Wiese}
\address{Department of Mathematics, Universit\"at Leipzig, Augustusplatz 10, 04109 Leipzig, Germany}
\email{sw31hiqa@studserv.uni-leipzig.de}
\thanks{SCW was supported by the Foundation of German Business (SDW)}


\date{March 19, 2019.}


\begin{abstract}

We present eigenvalue data and pictures of eigenfunctions of the classic and quadratic snowflake fractal and of quadratic filled julia sets. Furthermore, we approximate the area and box-counting dimension of selected Julia sets to compare the eigenvalue counting function with the Weyl term.

\end{abstract}

\maketitle


\section{Introduction.}

We study eigenvalues and eigenfunctions of the Laplacian on Snowflake Domains and certain selected filled Julia sets. The Snowflake domains are examples of curves that are continuous everywhere but differentiable nowhere. They have finite area bounded by an infinitely long line. We will start by looking at the classic Snowflake and the quadratic snowflake with different parameters.

Then we analyze the spectrum of filled-in Julia sets from the main bulb and compare their counting functions with the Weyl term. For the area we use a simple approximation; for the dimension we use the box-counting dimension. We will describe three ways of approximating the spectrum of the Basilica and the Rabbit, as well as the junctions points from the main bulb to the Basilica bulb and the Rabbit bulb: by iteration, by "walking to these points" (only junctions) and by analyzing the quasicircles that they consist of (only basilica/rabbit).

The Laplacian on the surface is just the usual two dimensional Laplacian $\Delta = \frac{\partial^2}{\partial x^2} + \frac{\partial^2}{\partial y^2}$. By the spectrum we mean a study of both the eigenvalues $\lambda$ and eigenfunctions $u$ satisfying 
\begin{equation}
\label{eq:cond}
-\Delta u = \lambda u
\end{equation}
with either Dirichlet or Neumann boundary conditions. It is known that the eigenvalues form an increasing sequence $0 = \lambda_0 < \lambda_1 \leq \lambda_2 \leq \ldots$ tending to infinity, and satisfying the Weyl asymptotic law
\begin{equation*}
\label{eq:weyl}
N(t) = \#\{\lambda_j \leq t\} \sim \frac{A}{4\pi}t
\end{equation*}
where $A$ is the area of the surface. We will study the difference
\begin{equation*}
D_1(t) = N(t) - \frac{A}{4\pi}t
\end{equation*}
and
\begin{equation*}
D_2(t) = \frac{D_1(t)}{t^\beta}
\end{equation*}
where $\beta$ is the box-counting dimension divided by two.
We use the Finite Element Method (FEM) and linear splines for our computations. The website \url{http://pi.math.cornell.edu/\~sw972} \cite{W} contains the programs we used and more data on eigenvalues and eigenfunctions.

\begin{figure}[h]
    \centering
    \subfigure[Zoom into a mesh of the Classic Snowflake]
    {
        \includegraphics[height=1.7in]{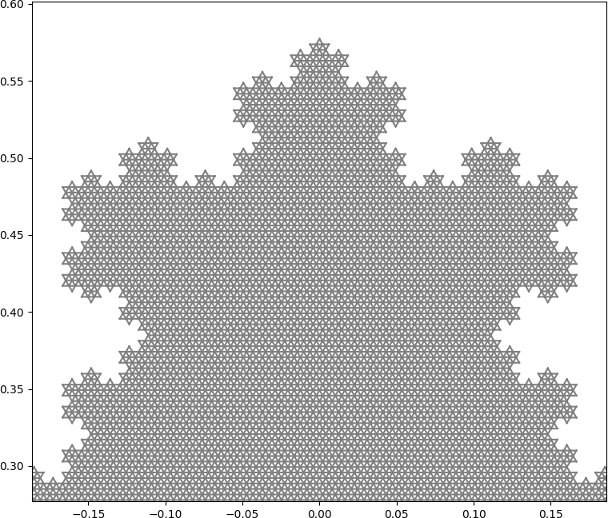}
    }
    \hfill
    \subfigure[Zoom into a mesh of the Basilica]
    {
        \includegraphics[height=1.7in]{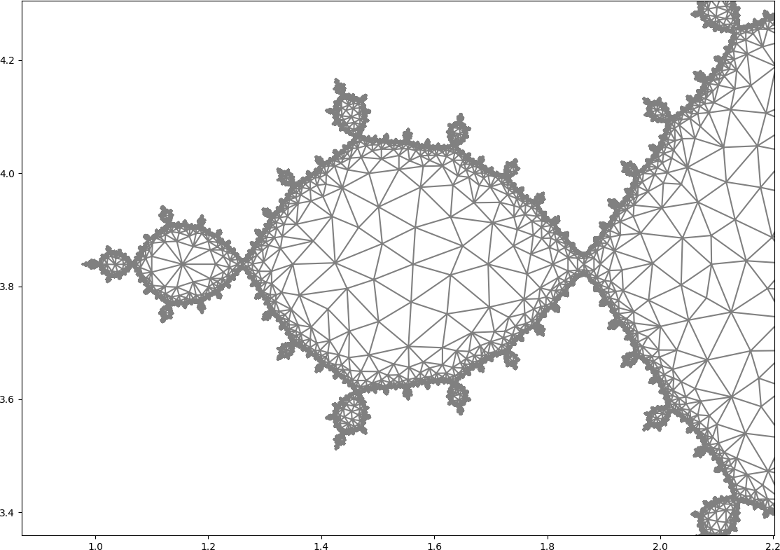}
    }
    \caption{Examples of meshes}
    \label{fig:sfbcd}
\end{figure}

The organization of the paper is as follows. In section 2 we describe the geometry of the snowflake domains. In section 3 we discuss the eigenvalues for the snowflake domains, and in section 4 the eigenfunctions. The reader should consult \cite{K}, \cite{L1}, \cite{L2}, \cite{LNRG}, \cite{LP}, \cite{NSS}, \cite{P} for earlier work on these questions. In section 5 we describe the geometry of the filled Julia sets. In section 6 and 7 we discuss the eigenvalues and eigenfunctions. For earlier references see \cite{S}, \cite{Y}. In section 8 we give a discussion of the significance of our results, suggestions for future research and some interesting conjectures.

\clearpage

\section{Geometry of Snowflakes.}
We will look at the classic Koch snowflake and three chosen quadratic snowflakes. The symmetries of the classic Snowflake is the group $D_6$, the quadratic snowflakes have the $D_8$ symmetry group. The dihedral-$6$ group of symmetries of the hexagon is generated by $6$ reflections.

\subsection{Classic Snowflake.}
We choose the side length of the level 1 triangle to be $1$. We construct the Clasic Snowflake domain by gluing together three Koch curves, each generated by the IFS
\begin{align*}
f_1(x)&=\left(
\begin{matrix}
\frac{1}{3} & 0 \\
0 & \frac{1}{3}
\end{matrix}\right)
x\\
f_2(x)&=\left(
\begin{matrix}
\frac{1}{6} & -\frac{\sqrt{3}}{6} \\
\frac{\sqrt{3}}{6} & \frac{1}{6}
\end{matrix}\right)
x+\left(
\begin{matrix}
\frac{1}{3} \\ 0
\end{matrix}\right)\\
f_3(x)&=\left(
\begin{matrix}
\frac{1}{6} & \frac{\sqrt{3}}{6} \\
-\frac{\sqrt{3}}{6} & \frac{1}{6}
\end{matrix}\right)
x+\left(
\begin{matrix}
\frac{1}{2} \\ \frac{\sqrt{3}}{6}
\end{matrix}\right)\\
f_4(x)&=\left(
\begin{matrix}
\frac{1}{3} & 0 \\
0 & \frac{1}{3}
\end{matrix}\right)
x+\left(
\begin{matrix}
\frac{2}{3} \\ 0
\end{matrix}\right)
\end{align*}
and shown in Fig. \ref{fig:csnowflakes}.

\begin{figure}[h!]
    \centering
    \subfigure[Level 1]
    {
        \includegraphics[width=0.8in]{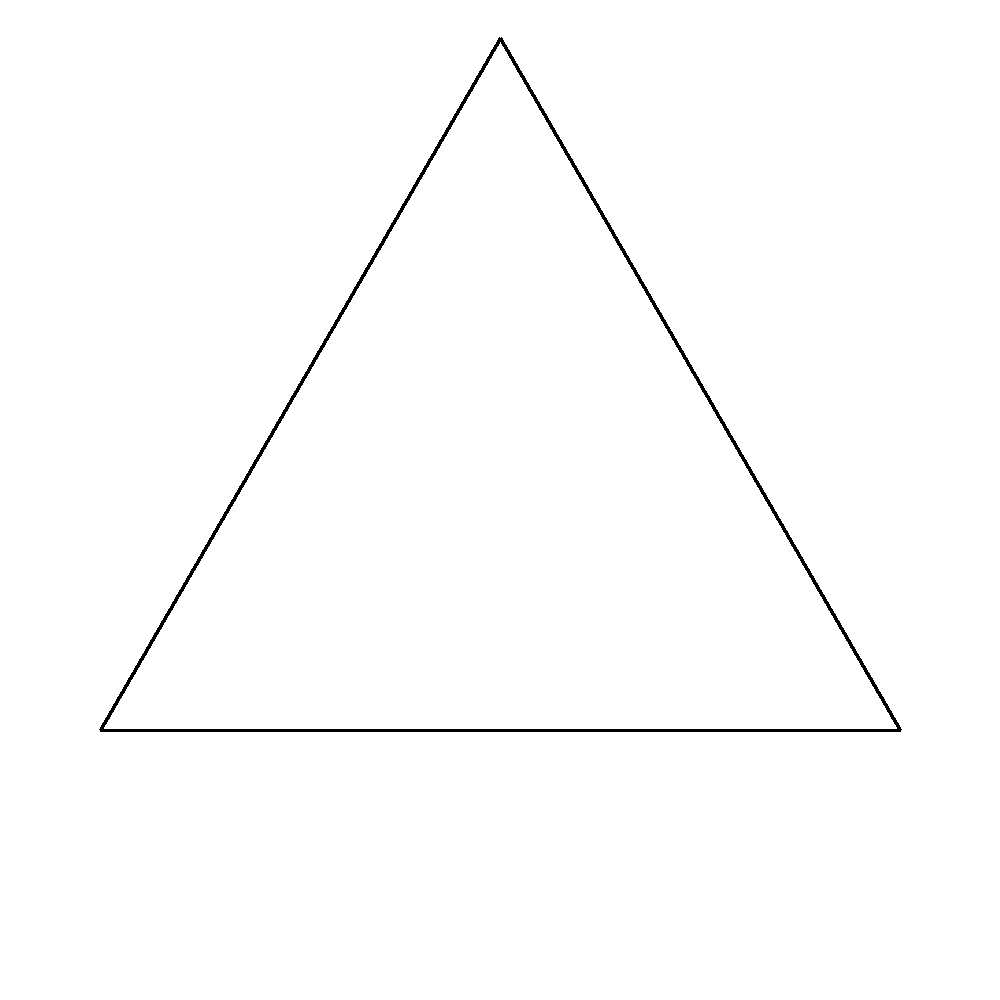}
    }
    \hfill
    \subfigure[Level 2]
    {
        \includegraphics[width=0.8in]{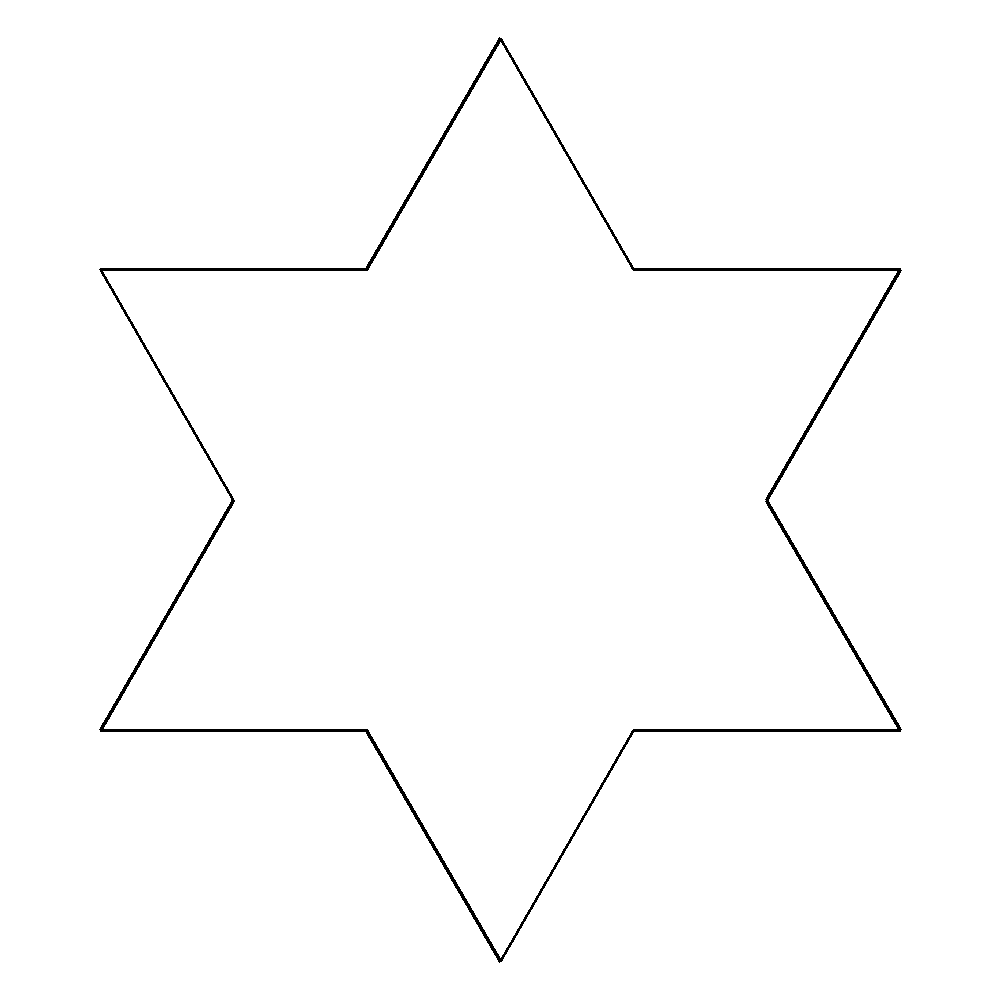}
    }
    \hfill
    \subfigure[Level 3]
    {
        \includegraphics[width=0.8in]{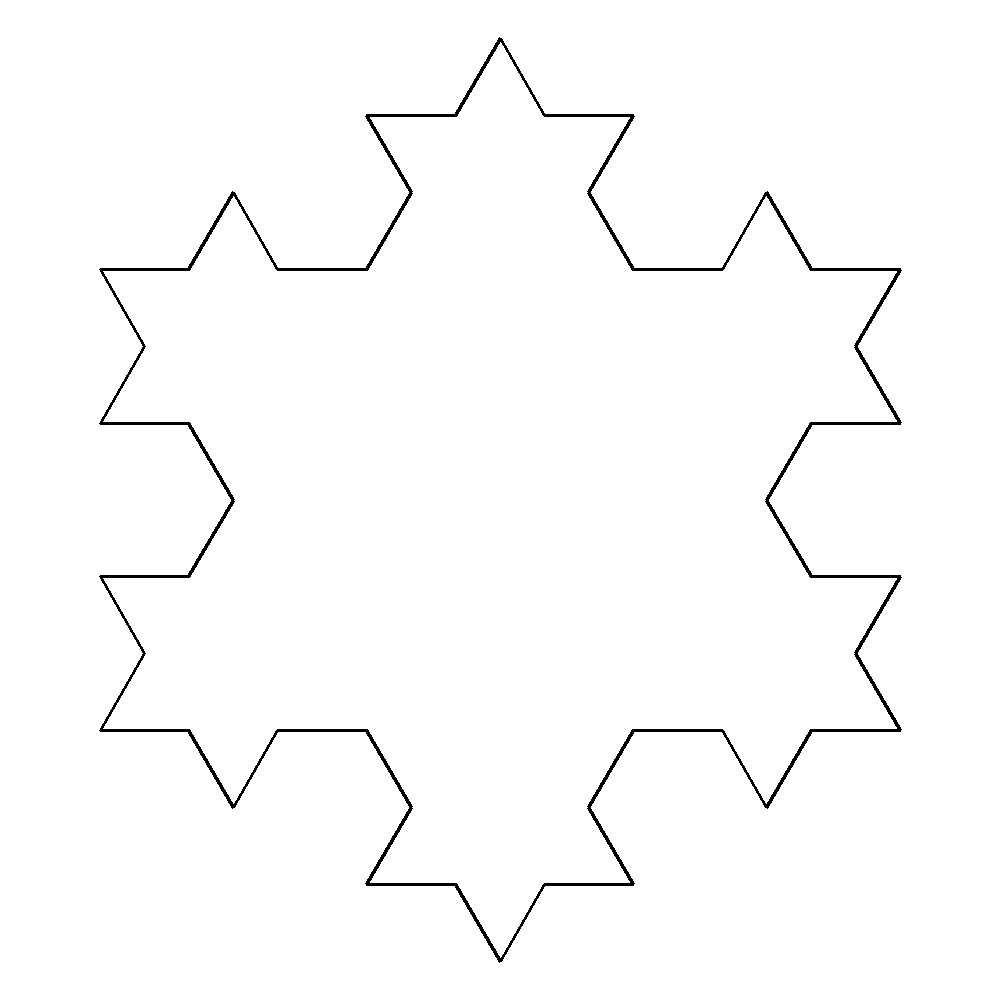}
    }
    \hfill
    \subfigure[Level 4]
    {
        \includegraphics[width=0.8in]{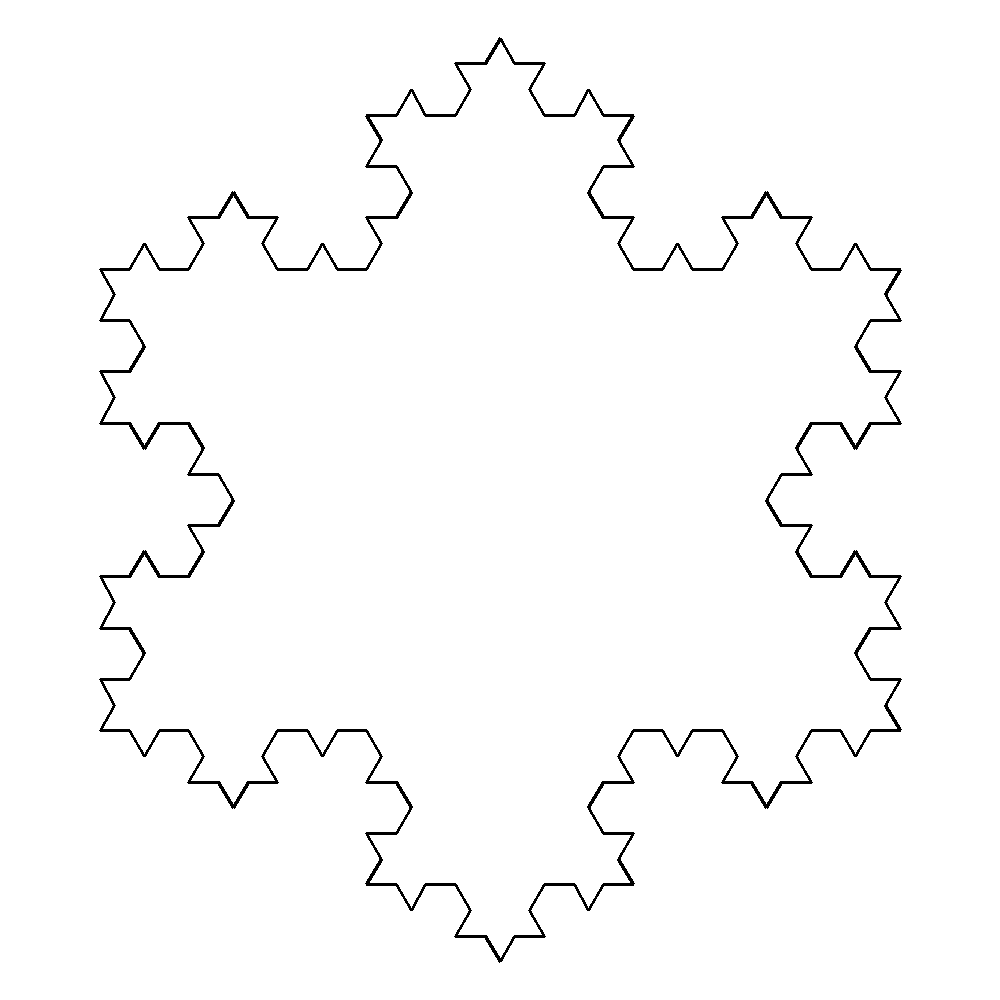}
    }
    \hfill
    \caption{The Classic Snowflake}
    \label{fig:csnowflakes}
\end{figure}

\begin{lemma}
The limit of the area of the classic snowflake with side length $1$ is 
\begin{equation}
\label{eq:sfarea}
A=\frac{2\sqrt{3}}{5}
\end{equation}
\end{lemma}

\begin{proof}
We start with a single triangle with side length $1$ at level 0, so $A_0=\frac{\sqrt{3}}{4}$. The number of triangles added in level $m$ is $3\cdot 4^{m-1}$. The area of each triangle added in level $m$ is one ninth of the area of each triangle added in the level $m-1$, so the area of a single triangle added in level $m$ is $\frac{A_0}{9^m}$. The total area added when going from level $m$ to $m+1$ is then $\frac{3}{4}\cdot\left(\frac{4}{9}\right)^m\cdot A_0$. The total area of the snowflake at level $m$ is then
\begin{equation*}
A_m=A_0\left(1+\frac{3}{4}\sum_{k=1}^m\left(\frac{4}{9}\right)^m\right)=A_0\left(1+\frac{1}{3}\sum_{k=1}^{m-1}\left(\frac{4}{9}\right)^m\right)
\end{equation*}
which is
\begin{equation}
A_m=A_0\left(1+\frac{3}{5}\left(1-\left(\frac{4}{9}\right)^m\right)\right)=\frac{A_0}{5}\left(8-3\left(\frac{4}{9}\right)^m\right)\text{.}
\end{equation}
Taking the limit gives (\ref{eq:sfarea}).
\end{proof}

\clearpage

\begin{lemma}
The box-counting dimension of the classic snowflake is $\frac{\log{4}}{\log{3}}$.
\end{lemma}

\begin{proof}
We just consider the Koch curve, so one third of the boundary of the snowflake with side length $1$. If we choose boxes with side length $\frac{1}{3}$, we need $3$ boxes to cover the curve. If we choose boxes with side length $\frac{1}{9}$, then we need $4\cdot 3=12$, since we have to cover 4 identical copies shrank to $\frac{1}{3}$, and so on (Figure \ref{fig:sfbcd}). Thus the number of boxes with side length $r_n=(\frac{1}{3})^n$ necessary to cover is equal to $N\left(r_n\right)=3\cdot 4^{n-1}$. From this we can compute the box-counting dimension exactly:
\begin{align*}
d_b&=\lim_{n\rightarrow \infty}\frac{\log{N(r_n)}}{\log\frac{1}{r_n}}\\
&=\lim_{n\rightarrow \infty}\frac{\log{N\left(\left(\frac{1}{3}\right)^n\right)}}{\log\frac{1}{\left(\frac{1}{3}\right)^n}}\\
&=\lim_{n\rightarrow \infty}\frac{\log\left({3\cdot 4^{n-1}}\right)}{\log 3^n}\\
&=\lim_{n\rightarrow \infty}\frac{(n-1)\log{4}+\log{3}}{n\cdot\log{3}}\\
&=\lim_{n\rightarrow \infty}\frac{n\log{4}-\log{4}+\log{3}}{n\cdot\log{3}}\\
&=\lim_{n\rightarrow \infty}\left(\frac{n\log{4}}{n\log{3}}+\frac{log{3}-\log{4}}{n\log{3}}\right)\\
&=\frac{\log{4}}{\log{3}}
\end{align*}
\end{proof}

\begin{figure}[h]
    \centering
    \subfigure[Covering with boxes of length $\frac{1}{3}$]
    {
        \includegraphics[width=2.2in]{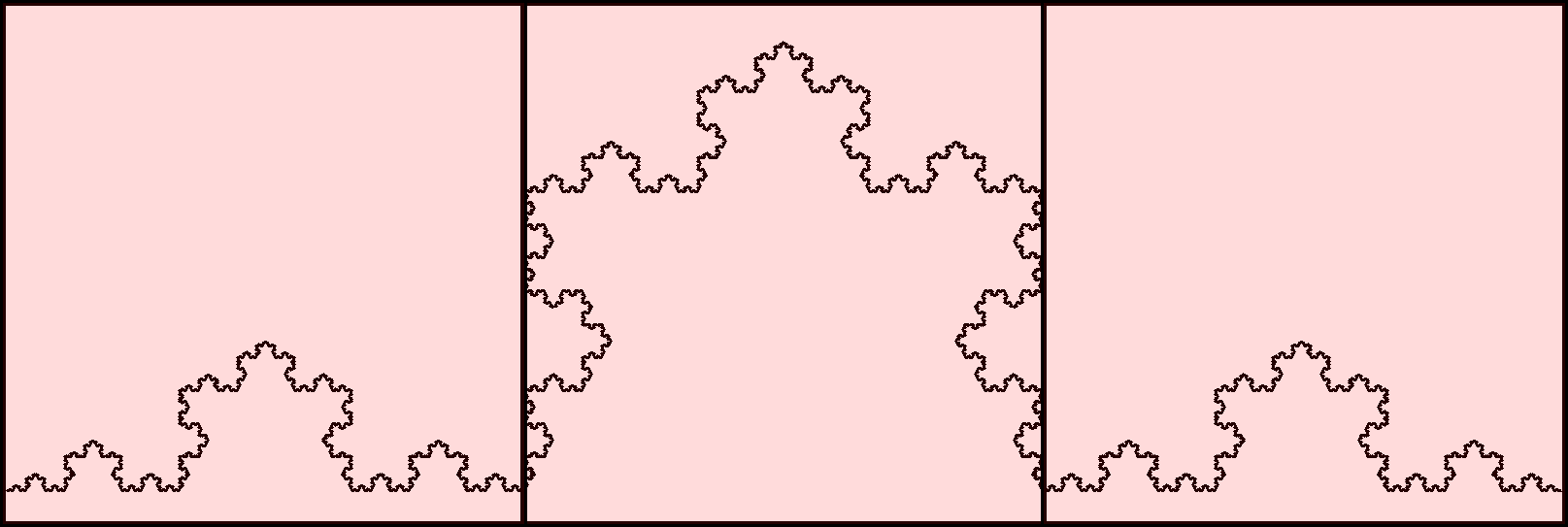}
    }
    \hfill
    \subfigure[Covering with boxes of length $\frac{1}{9}$]
    {
        \includegraphics[width=2.2in]{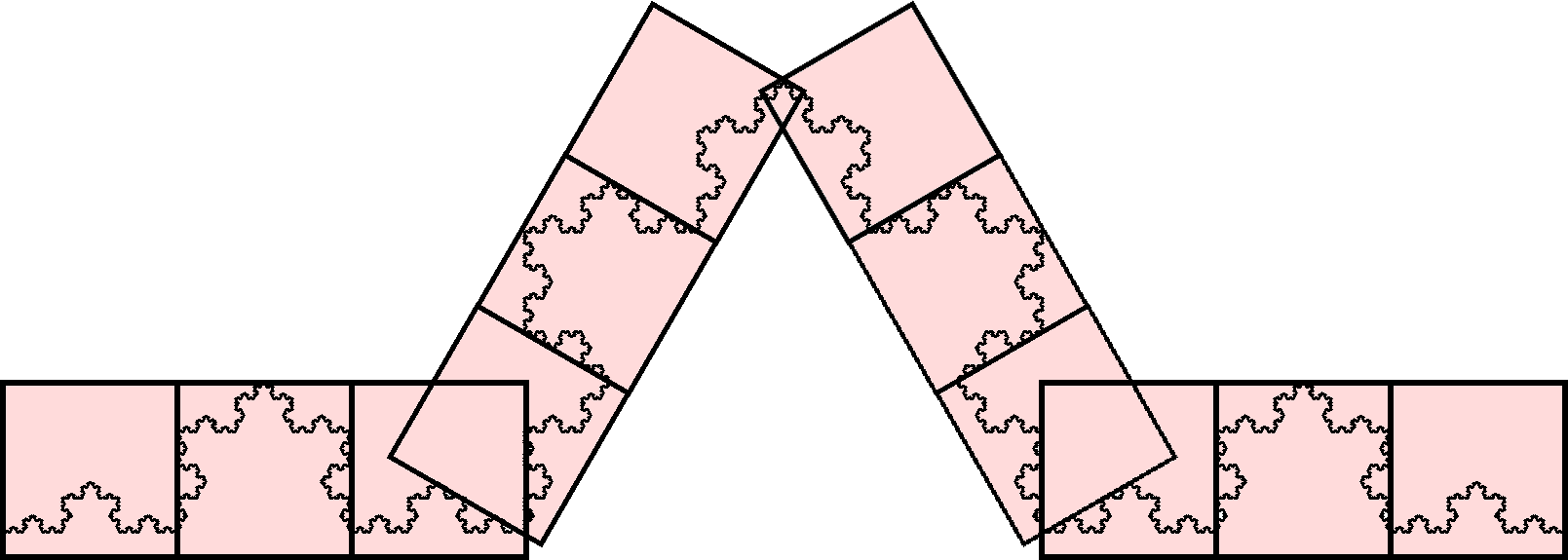}
    }
    \caption{Covering the classic snowflake with boxes}
    \label{fig:sfbcd}
\end{figure}

For these iterated function system we were able to give the exact dimension. This will not be possible for the Julia sets, where we will compute an approximation by counting boxes of different sizes from a sequence $r_n$ that tends geometrically to zero. The idea is that the relationship between the number of boxes necessary to cover the curve and the size of the boxes should be maintained over a broad range of values. We fit a line into a log-log plot of the number of boxes to the size of each box and take the slope. We check the method on the Classic Snowflake on Level $6$ and the Quadratic Snowflake ($b=0.2$, Level $8$) and computed dimensions $d\approx1.231$ and $d\approx1.280$ respectively, shown in Fig. \ref{fig:loglog1}).

Keep in mind that for Julia sets, the computation of the box counting dimension is more problematic and that the limit does not necessarly converge. For a more elaborate approach see \cite{S}, but for our purposes a simple counting method suffices.

\begin{figure}[h]
    \centering
    \subfigure[For the Classic Snowflake on Level 6: $d\approx1.231$]
    {
        \includegraphics[width=2.3in]{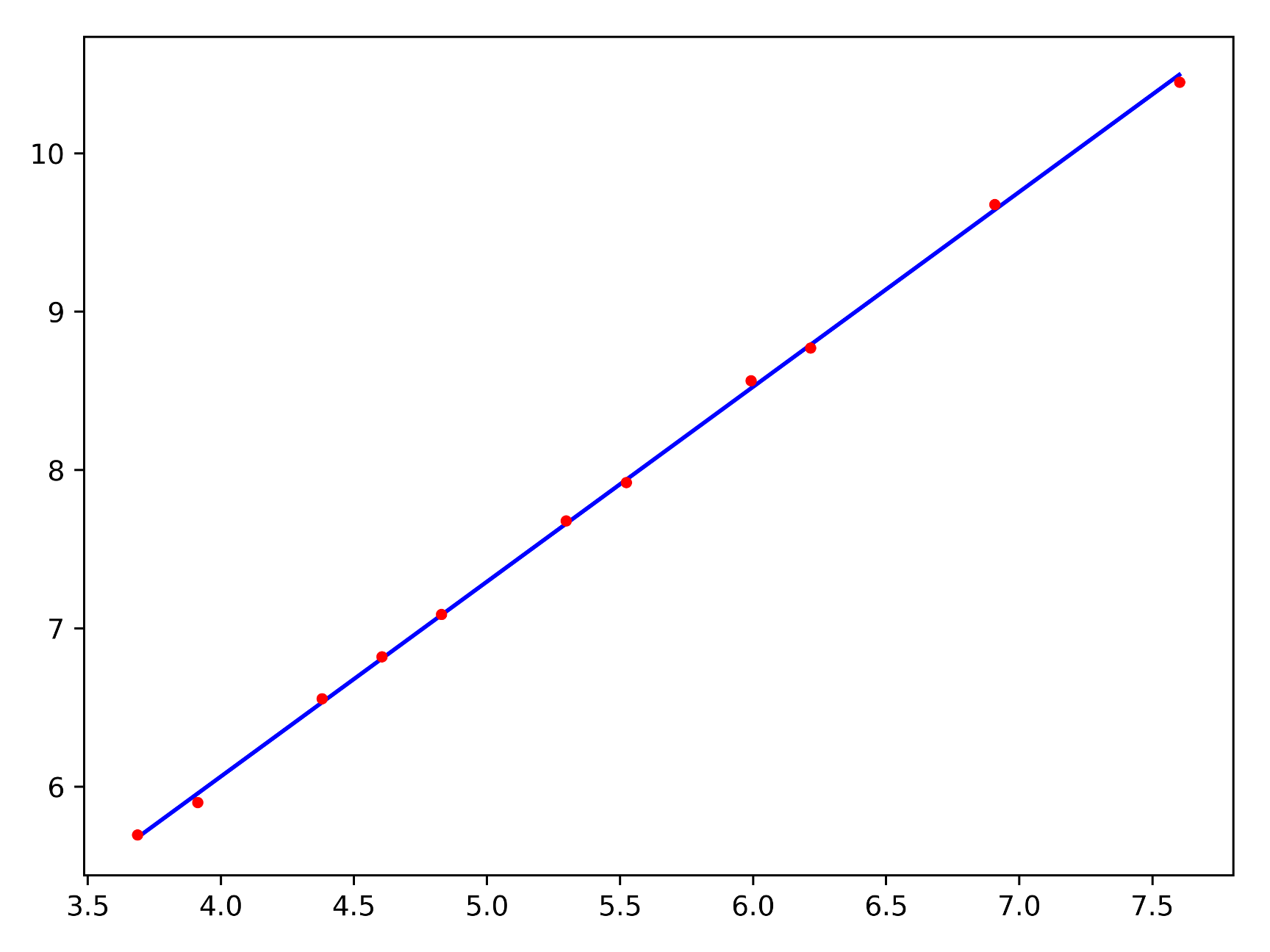}
    }
    \hfill
    \subfigure[For the Quadratic Snowflake ($b=0.2$) on Level 6: $d\approx1.280$]
    {
        \includegraphics[width=2.3in]{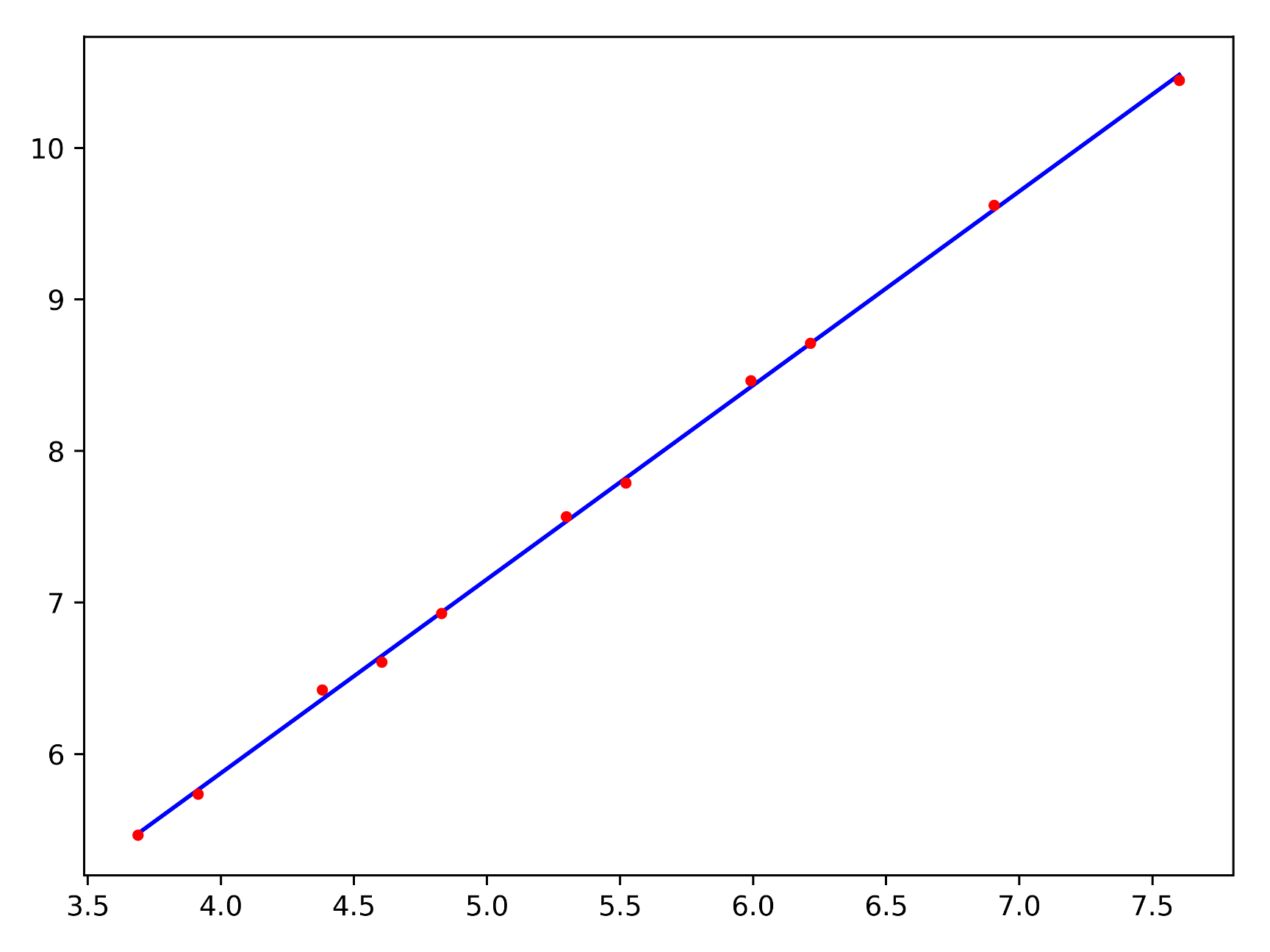}
    }
    \caption{Log-Log plot for the box-counting dimension}
    \label{fig:loglog1}
\end{figure}

\subsection{Quadratic Snowflake}
We choose the side length of the level 1 square to be $1$. We construct the domain by gluing together three curves, each generated by the IFS
\begin{align*}
f_1(x)&=\left(
\begin{matrix}
a & 0 \\
0 & a
\end{matrix}\right)
x\\
f_2(x)&=\left(
\begin{matrix}
0 & -b \\
b & 0
\end{matrix}\right)
x+\left(
\begin{matrix}
a \\ 0
\end{matrix}\right)\\
f_3(x)&=\left(
\begin{matrix}
b & 0 \\
0 & b
\end{matrix}\right)
x+\left(
\begin{matrix}
a \\ b
\end{matrix}\right)\\
f_4(x)&=\left(
\begin{matrix}
0 & b \\
-b & 0
\end{matrix}\right)
x+\left(
\begin{matrix}
a+b \\ b
\end{matrix}\right)\\
f_5(x)&=\left(
\begin{matrix}
a & 0 \\
0 & a
\end{matrix}\right)
x+\left(
\begin{matrix}
a+b \\ 0
\end{matrix}\right)
\end{align*}

and shown in Fig. \ref{fig:qsnowflakes}.

\begin{figure}[h]
    \centering
    \subfigure[Level 1]
    {
        \includegraphics[width=1in]{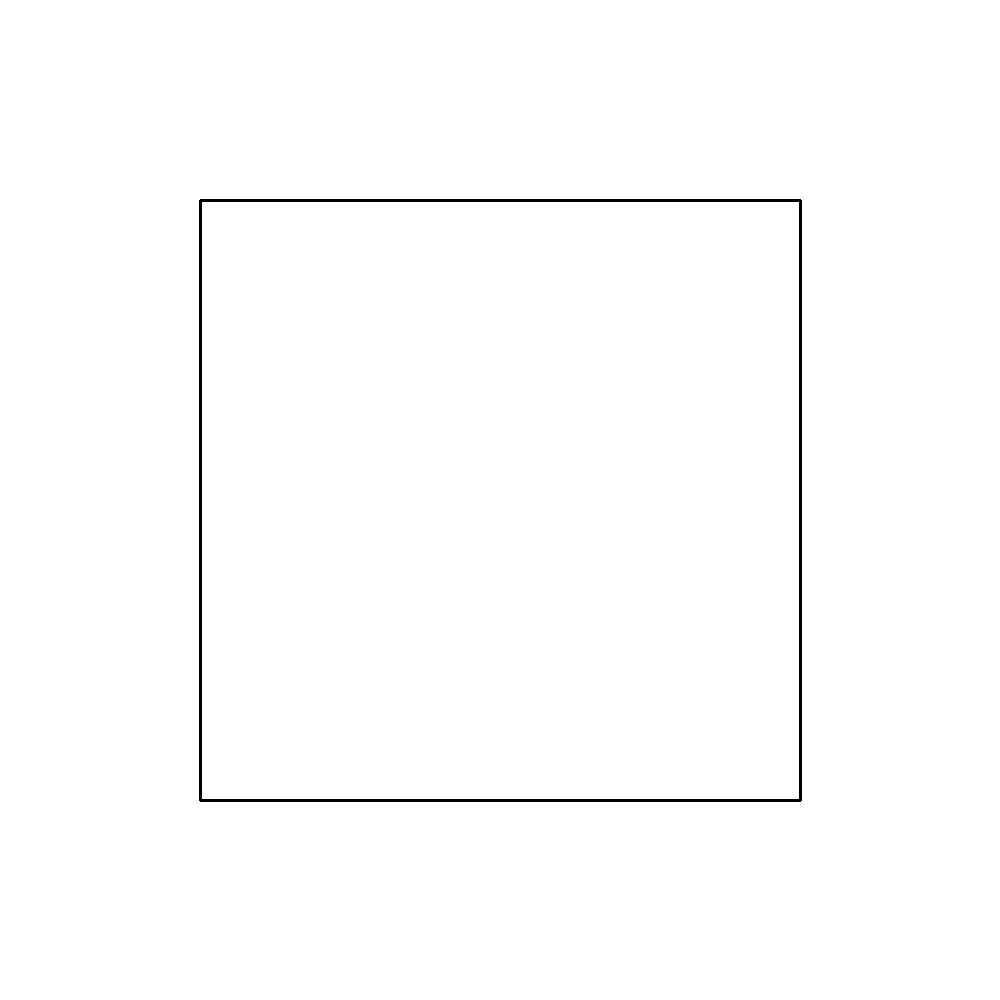}
    }
    \hfill
    \subfigure[Level 2]
    {
        \includegraphics[width=1in]{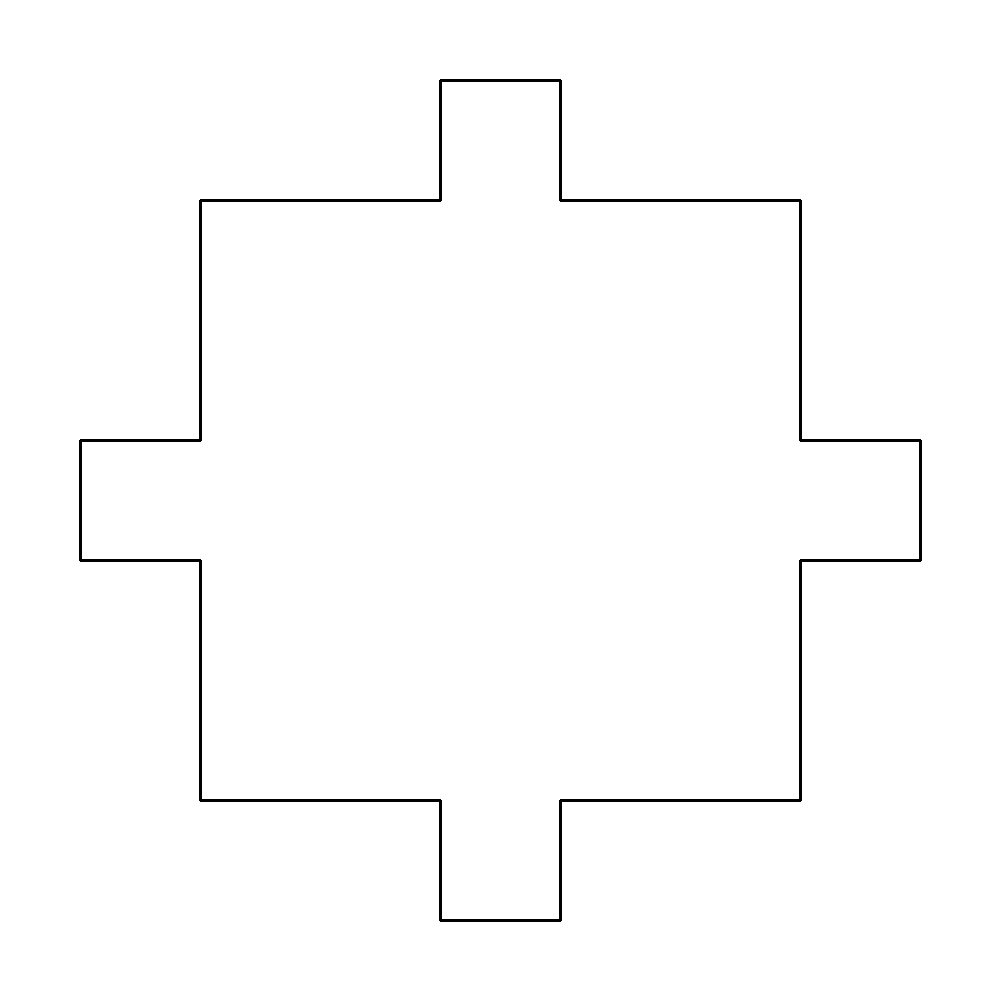}
    }
    \hfill
    \subfigure[Level 3]
    {
        \includegraphics[width=1in]{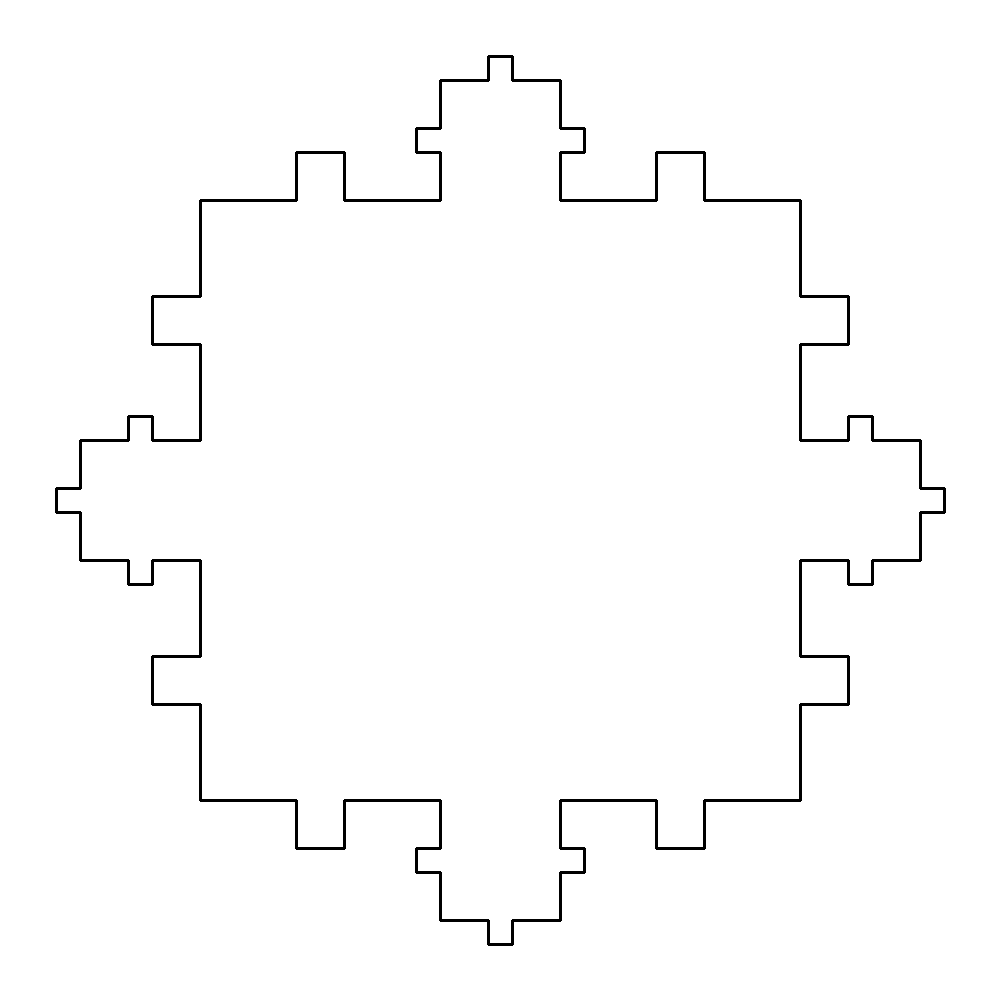}
    }
    \hfill
    \subfigure[Level 4]
    {
        \includegraphics[width=1in]{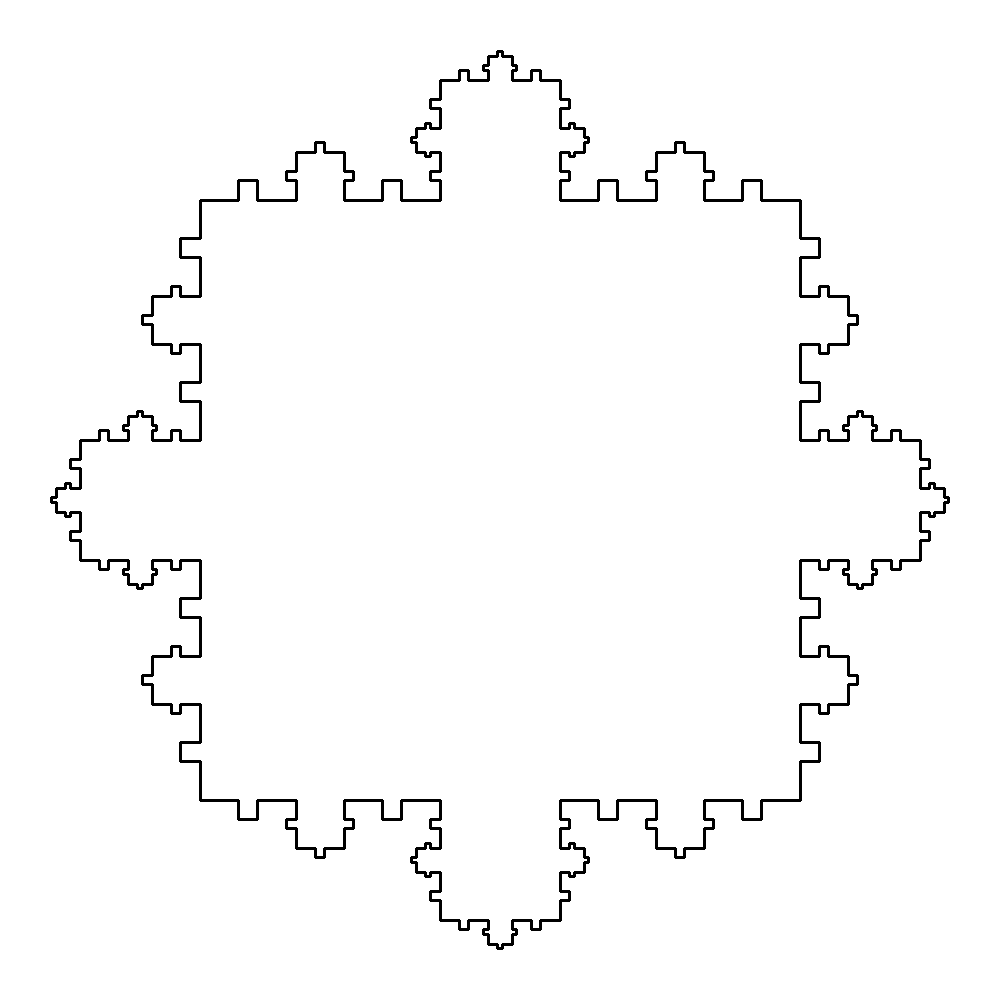}
    }
    \hfill
    \caption{The Quadratic Snowflake ($b=0.2$)}
    \label{fig:qsnowflakes}
\end{figure}

\begin{lemma}
The limit of the area of the quadratic snowflake is
\begin{equation}
\label{eq:qsfarea}
A=(2a+b)^2+\frac{4b^2}{1-2a^2-3b^2}\text{.}
\end{equation}
This gives $A\approx1.07080$ for $b=0.1$, $A\approx1.28571$ for $b=0.2$, and $A\approx1.20711$ for the lattice case where $b=3-2\sqrt{2}$ (then $a^2=b)$.
\end{lemma}

\begin{proof}
We start at a square with side length $2a+b$. By inspection we see that the area added when going from level $m$ to $m+1$ is
\begin{equation*}
\sum_{j=1}^m \sum_{k=0}^{j-1}\binom{j-1}{k} 2^k 3^{(j-1)-k}\left(a^k b^{j-k}\right)^2
\end{equation*}
so the area at level $m$ is
\begin{align*}
A_m&=(2a+b)^2+4\sum_{j=1}^m \sum_{k=0}^{j-1}\binom{j-1}{k} 2^k 3^{(j-1)-k}\left(a^k b^{j-k}\right)^2\\
&=(2a+b)^2+4b^2\sum_{j=1}^m \sum_{k=0}^{j-1}\binom{j-1}{k} \left(2a^2\right)^k\left(3b^2\right)^{(j-1)-k}\\
&=(2a+b)^2+4b^2\sum_{j=0}^{m-1} \left(2a^2+3b^2\right)^j\text{.}
\end{align*}
Taking the limit gives (\ref{eq:qsfarea}).
\end{proof}

\begin{lemma}
The box-counting dimension $d$ of the quadratic snowflake satisfies
\begin{equation}
\label{eq:qbcd}
2\left(\frac{1-b}{2}\right)^{d}+3b^{d}=1\text{.}
\end{equation}
This gives $d\approx1.15965$ for $b=0.1$, $d\approx1.2811$ for $b=0.2$, and $d\approx1.2465$ for the lattice case where $b=3-2\sqrt{2}$.
\end{lemma}

\begin{proof}
The box-counting dimension $d_b$ of the quadratic snowflake satisfies:
\begin{equation}
2a^{d}+3b^{d}=1
\end{equation}
since if we just look at the upper fourth of the curve, there are two copies each shrunk by $a$ next to the middle square with area $b^2$. This gives (\ref{eq:qbcd}).
\end{proof}

\section{Snowflake spectrum.}
In Figures \ref{fig:cd} -- \ref{fig:q3n} we show the graphs of the Dirichlet and Neumann counting function $N(t)$,
\begin{equation}
D_1(t) = N(t) - \frac{A}{4\pi}\text{,}
\end{equation}
and
\begin{equation}
D_2(t) = \frac{D_1(t)}{t^\beta}
\end{equation}
where $\beta$ is the box-counting dimension divided by two, for the classic snowflake and three different quadratic snowflakes with $a=0.45$, $a=0.4$ and $a=\sqrt{2}-1$ (lattice case).

\clearpage

\begin{figure}[h!]
    \centering
    \subfigure[$N(t)$ (blue) and $\frac{A}{4\pi}t$ (orange)]
    {
        \includegraphics[width=1.81in]{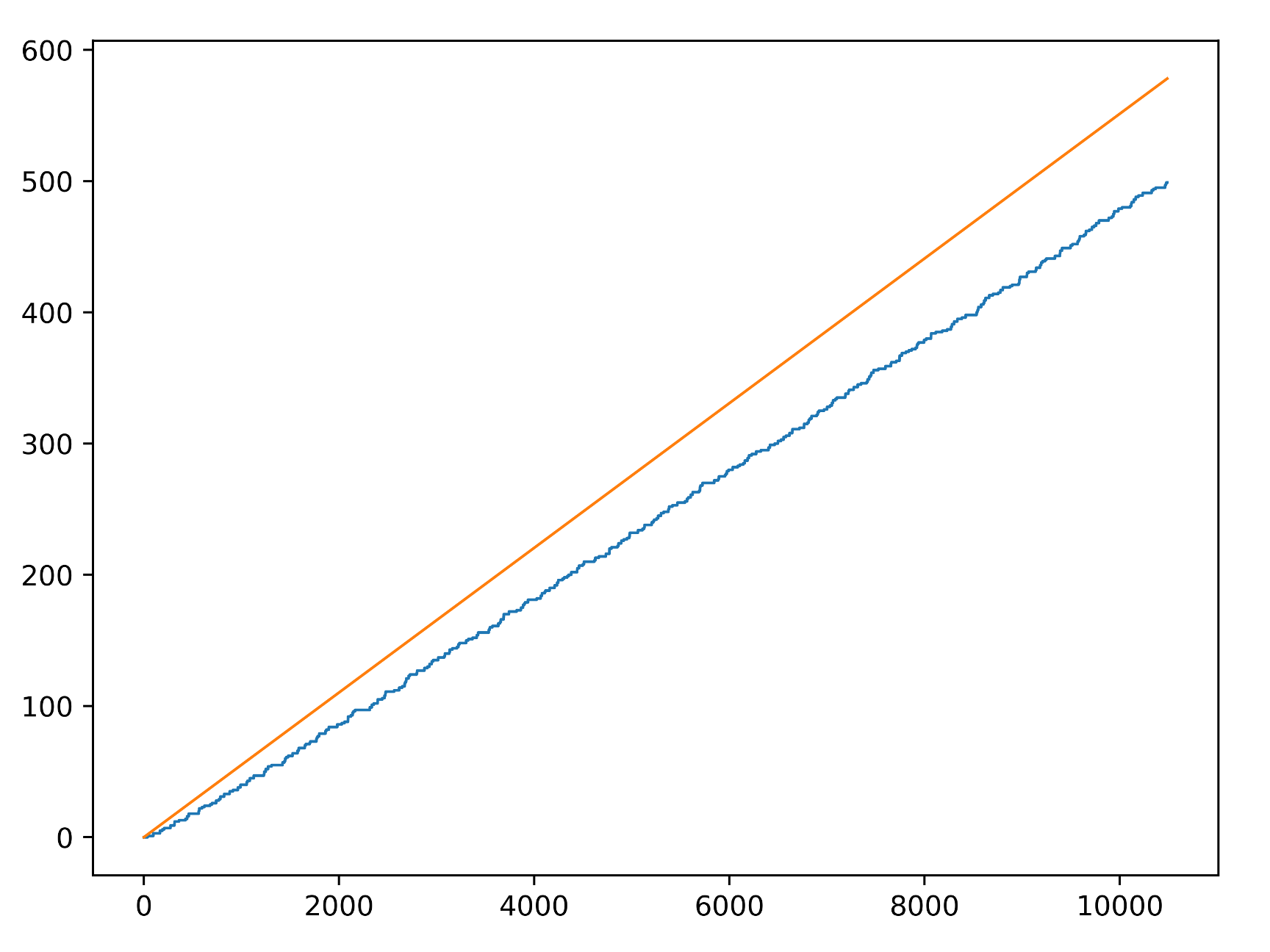}
    }
    \hfill
    \subfigure[$D_1(t)$]
    {
        \includegraphics[width=1.81in]{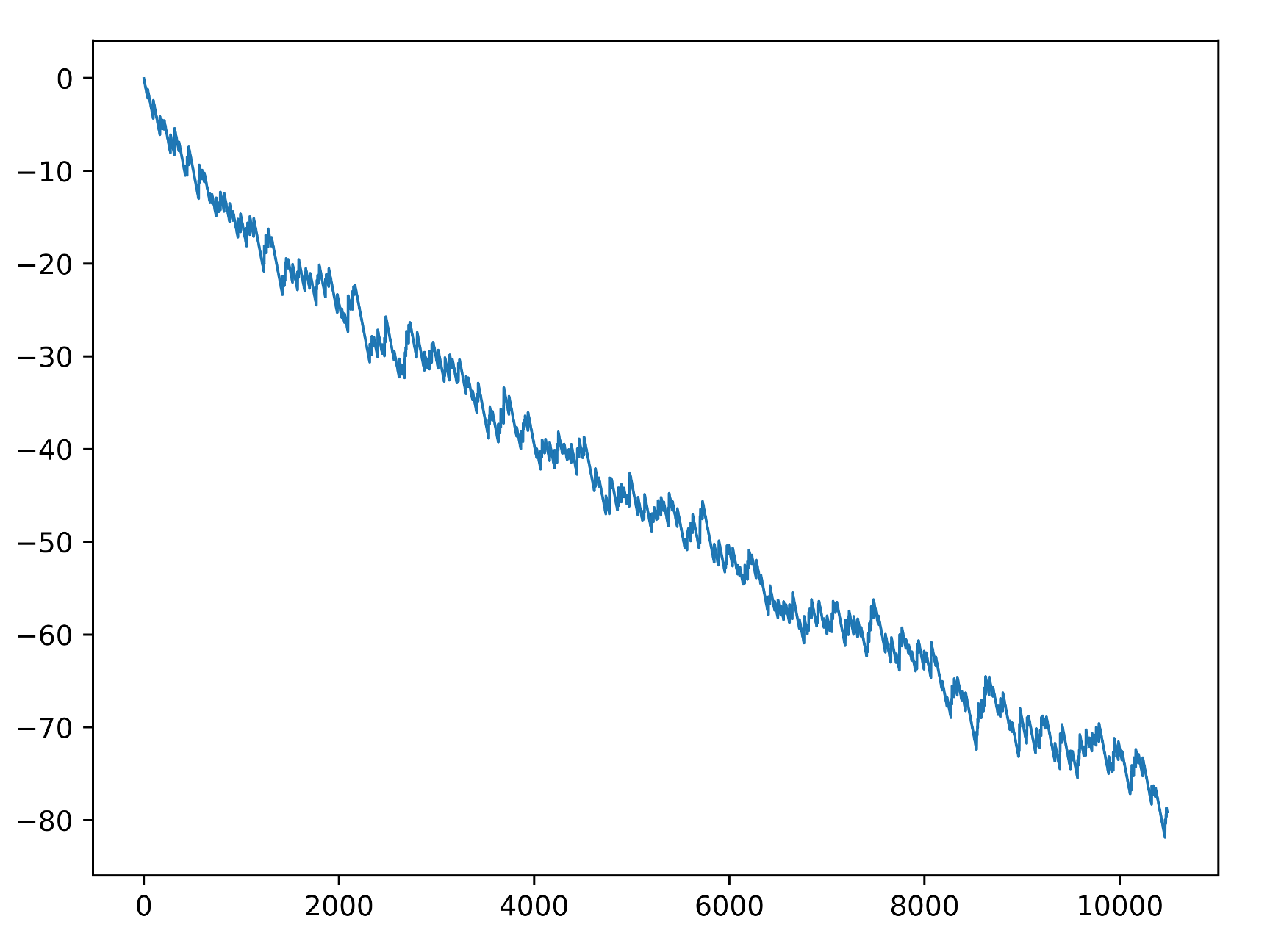}
    }
    \subfigure[$D_2(t)$]
    {
        \includegraphics[width=1.81in]{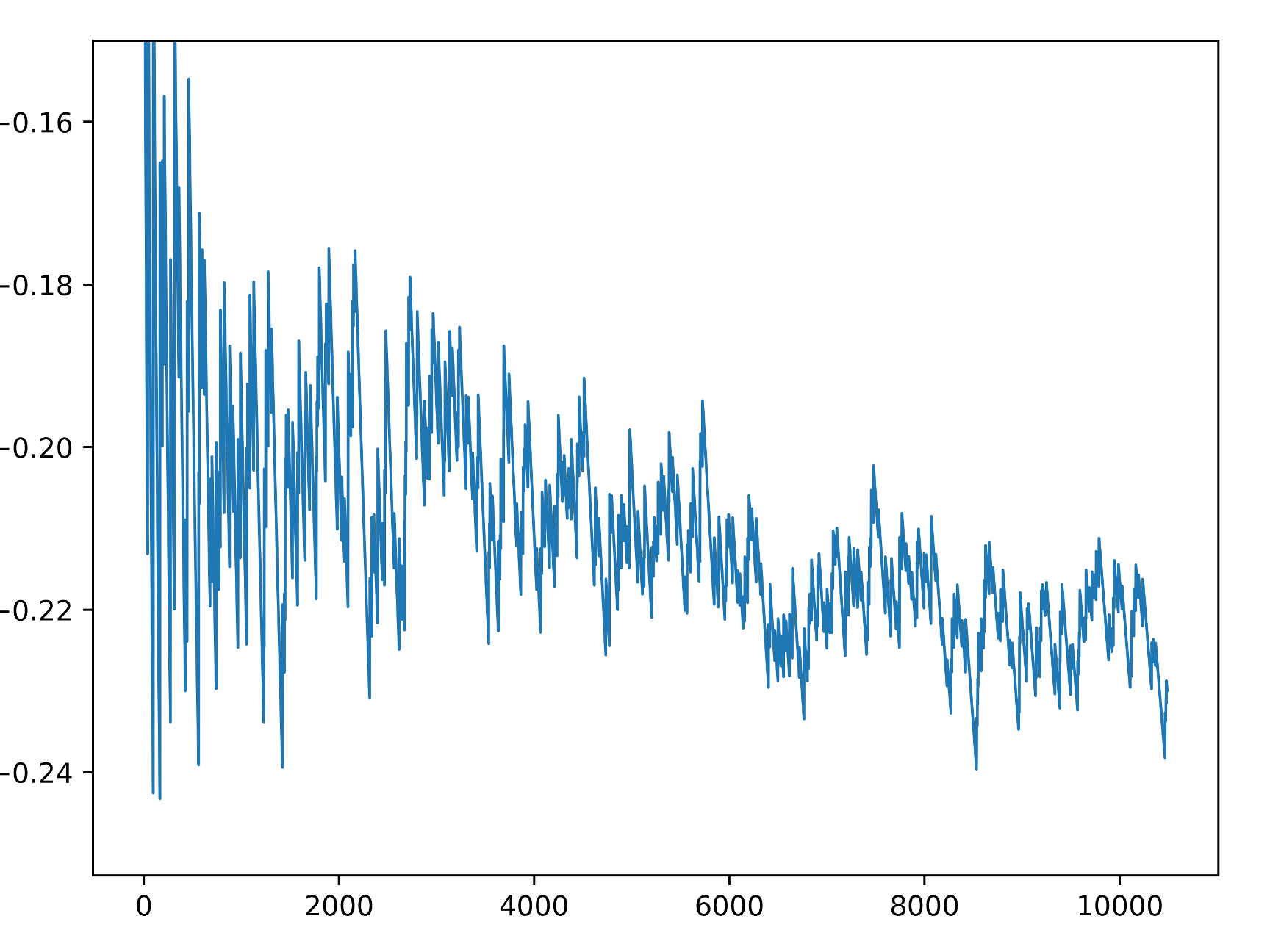}
    }
    \hfill
    \subfigure[$D_2(t)$ (zoom)]
    {
        \includegraphics[width=1.81in]{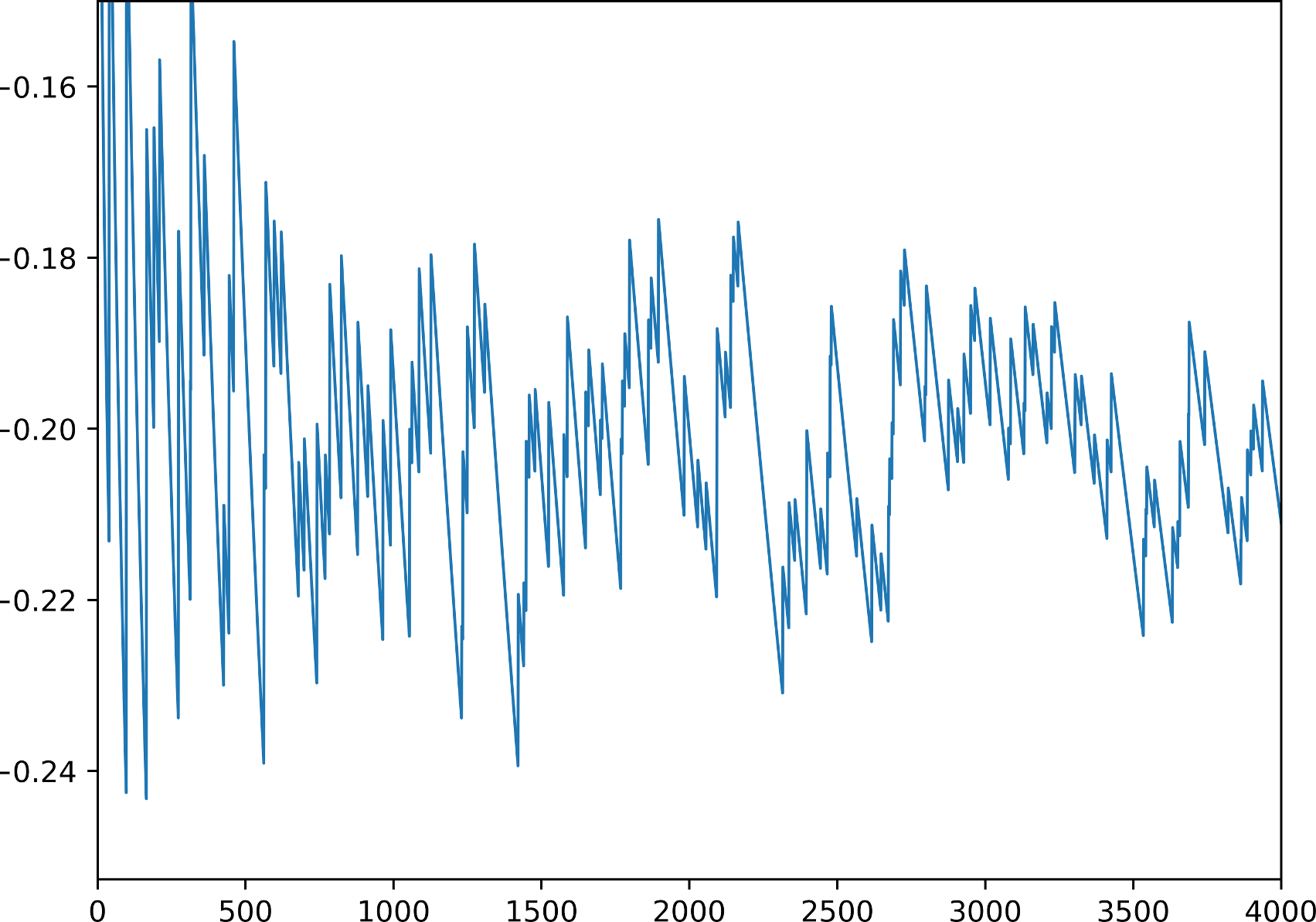}
    }
    \caption{Classic snowflake (Level 6) with Dirichlet BC}
    \label{fig:cd}
\end{figure}

\begin{figure}[h!]
    \centering
    \subfigure[$N(t)$ (blue) and $\frac{A}{4\pi}t$ (orange)]
    {
        \includegraphics[width=1.81in]{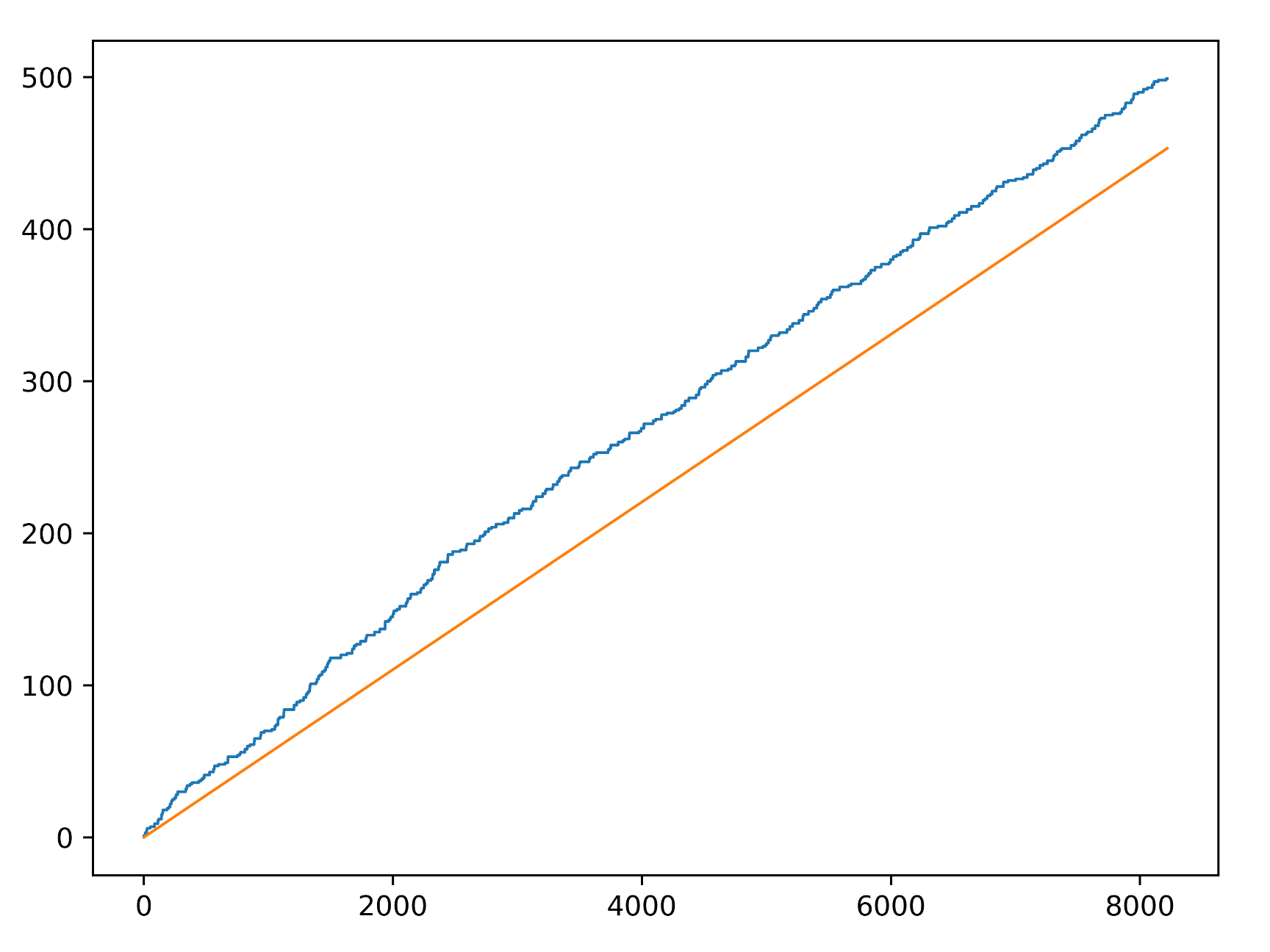}
    }
    \hfill
    \subfigure[$D_1(t)$]
    {
        \includegraphics[width=1.81in]{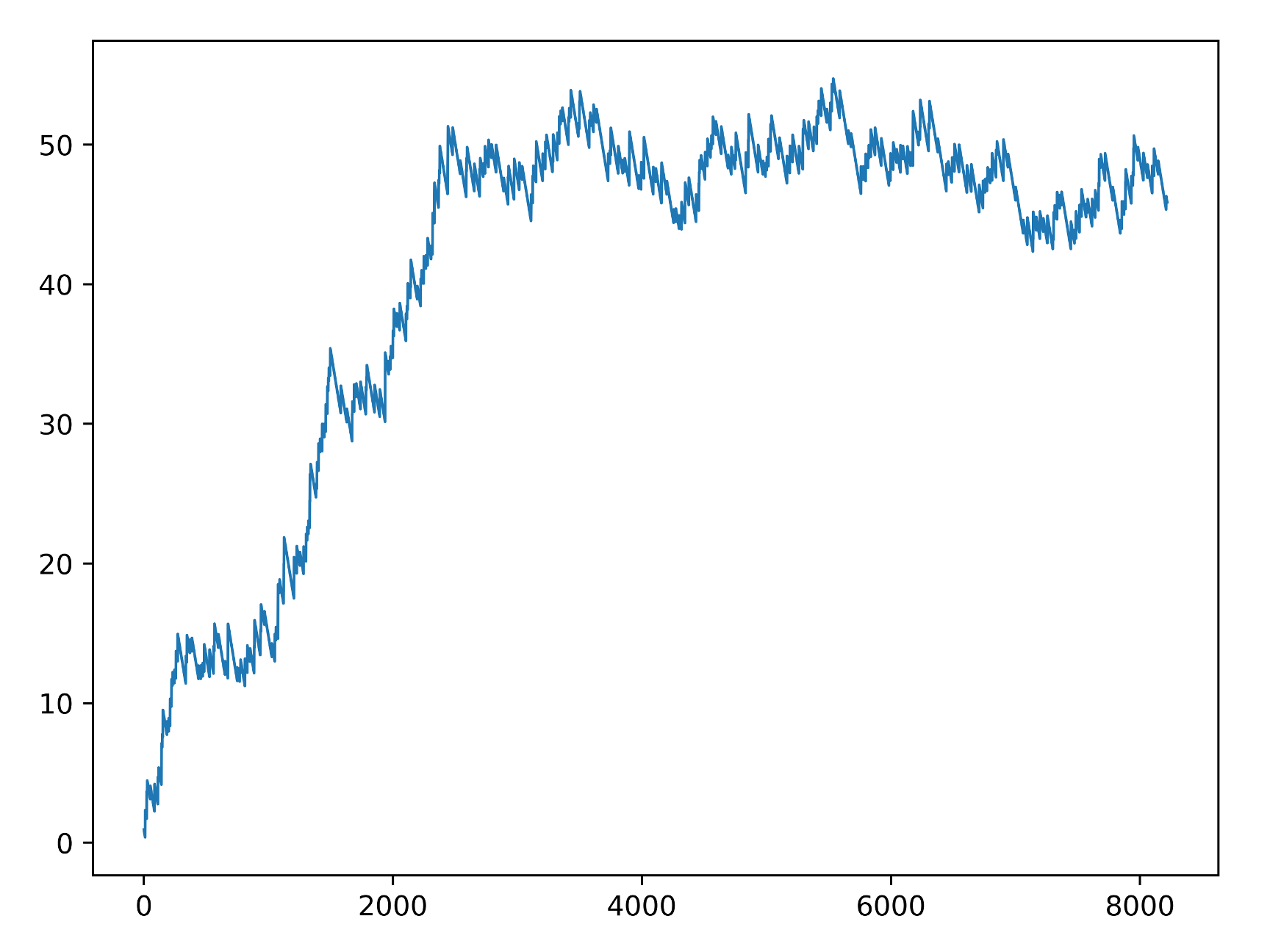}
    }
    \subfigure[$D_2(t)$]
    {
        \includegraphics[width=1.81in]{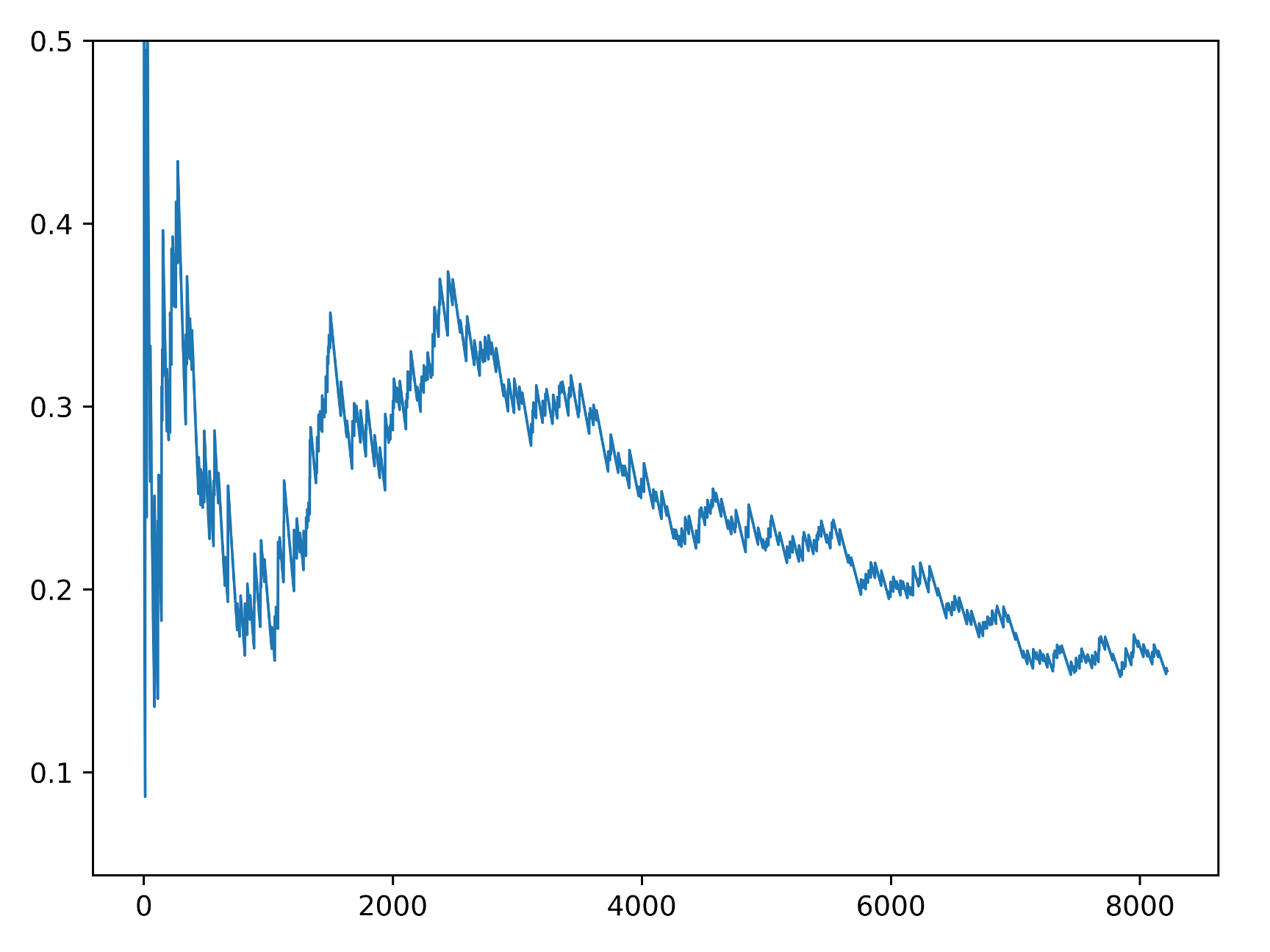}
    }
    \hfill
    \subfigure[$D_2(t)$ (zoom)]
    {
        \includegraphics[width=1.81in]{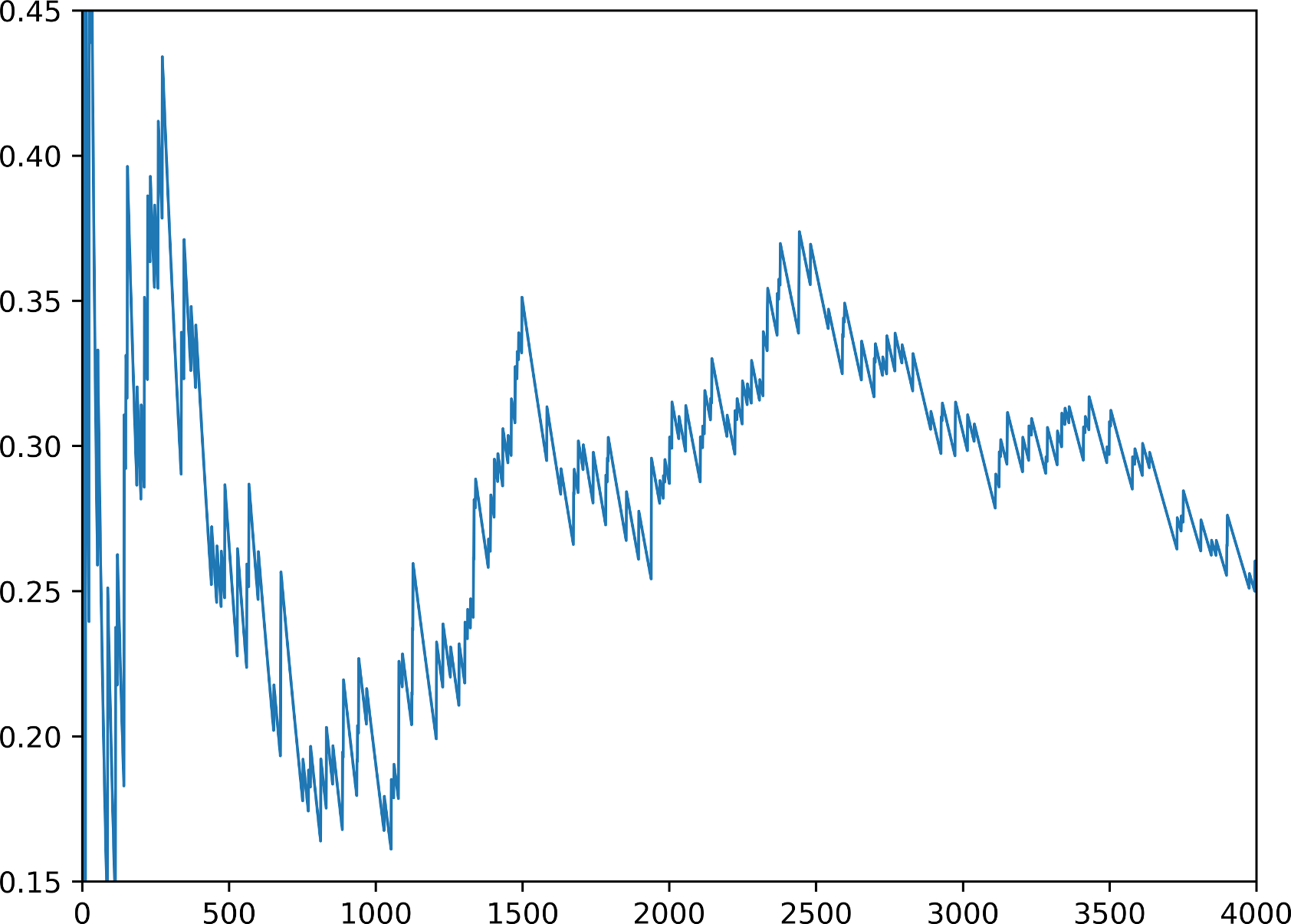}
    }
    \caption{Classic snowflake (Level 6) with Neumann BC}
\end{figure}

\clearpage

\begin{figure}[h!]
    \centering
    \subfigure[$N(t)$ (blue) and $\frac{A}{4\pi}t$ (orange)]
    {
        \includegraphics[width=1.81in]{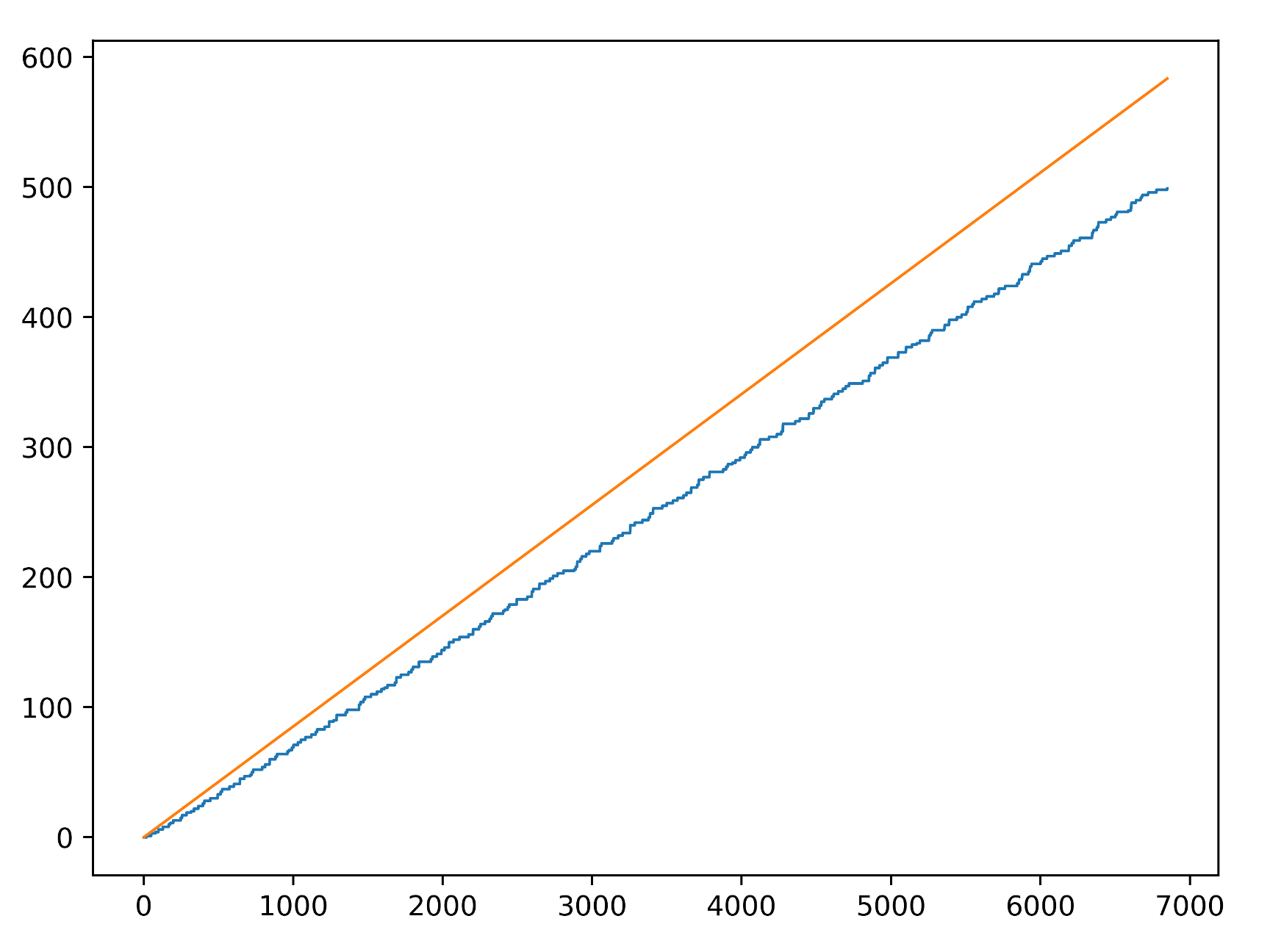}
    }
    \hfill
    \subfigure[$D_1(t)$]
    {
        \includegraphics[width=1.81in]{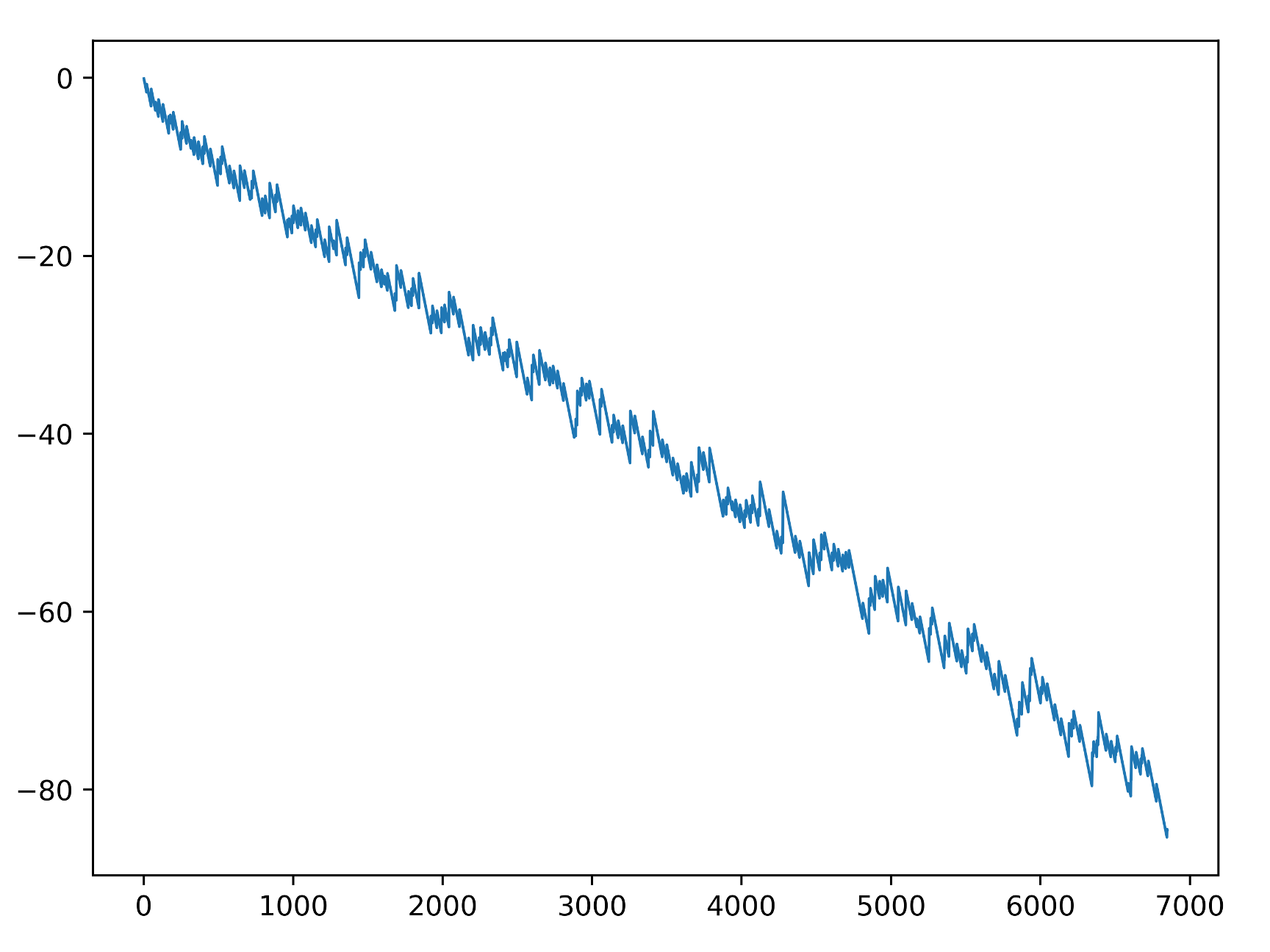}
    }
    \subfigure[$D_2(t)$]
    {
        \includegraphics[width=1.81in]{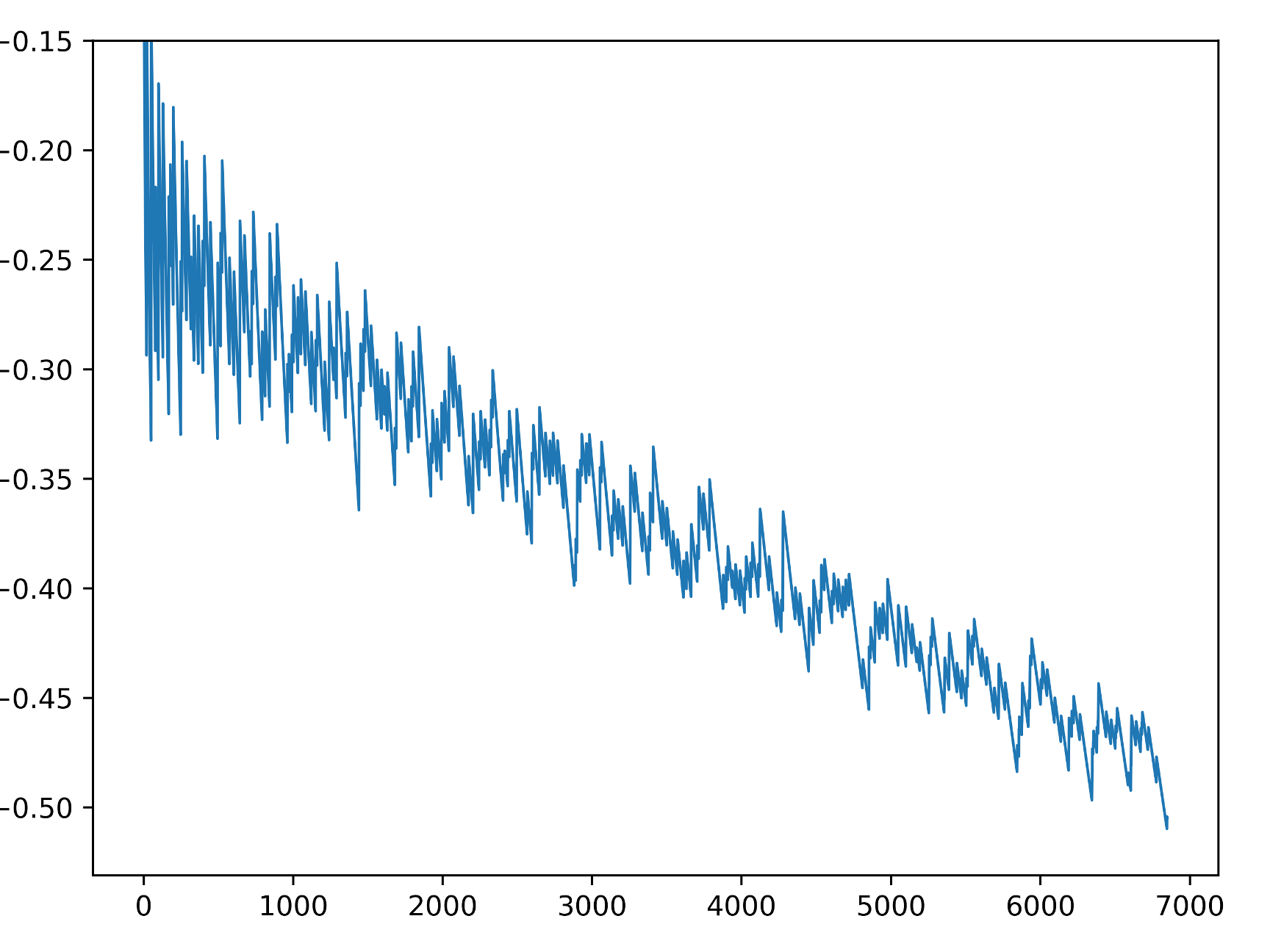}
    }
    \hfill
    \subfigure[$D_2(t)$ (zoom)]
    {
        \includegraphics[width=1.81in]{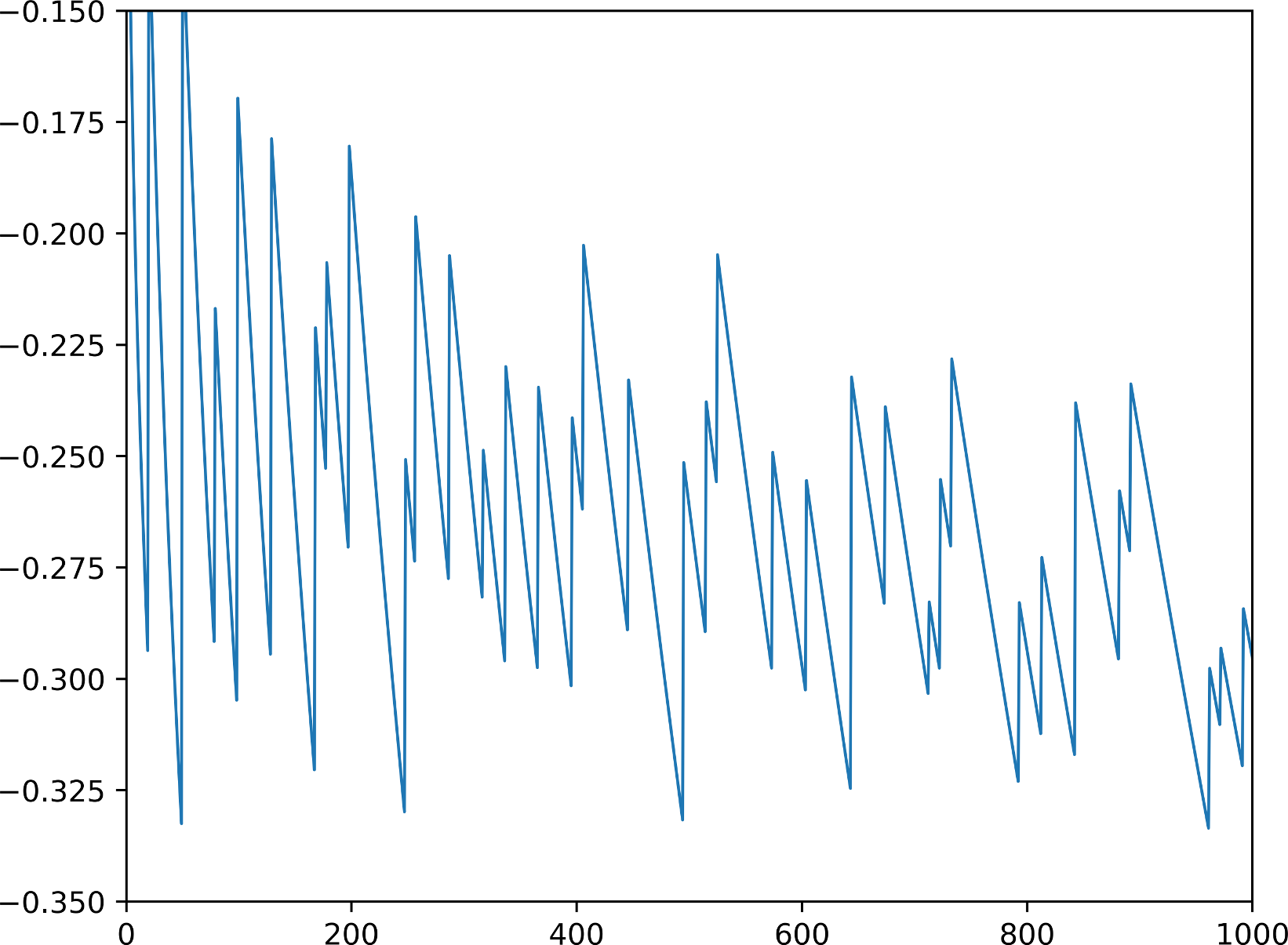}
    }
    \caption{Quadratic snowflake (Level 4, $a=0.45$) with Dirichlet BC}
\end{figure}

\begin{figure}[h!]
    \centering
    \subfigure[$N(t)$ (blue) and $\frac{A}{4\pi}t$ (orange)]
    {
        \includegraphics[width=1.81in]{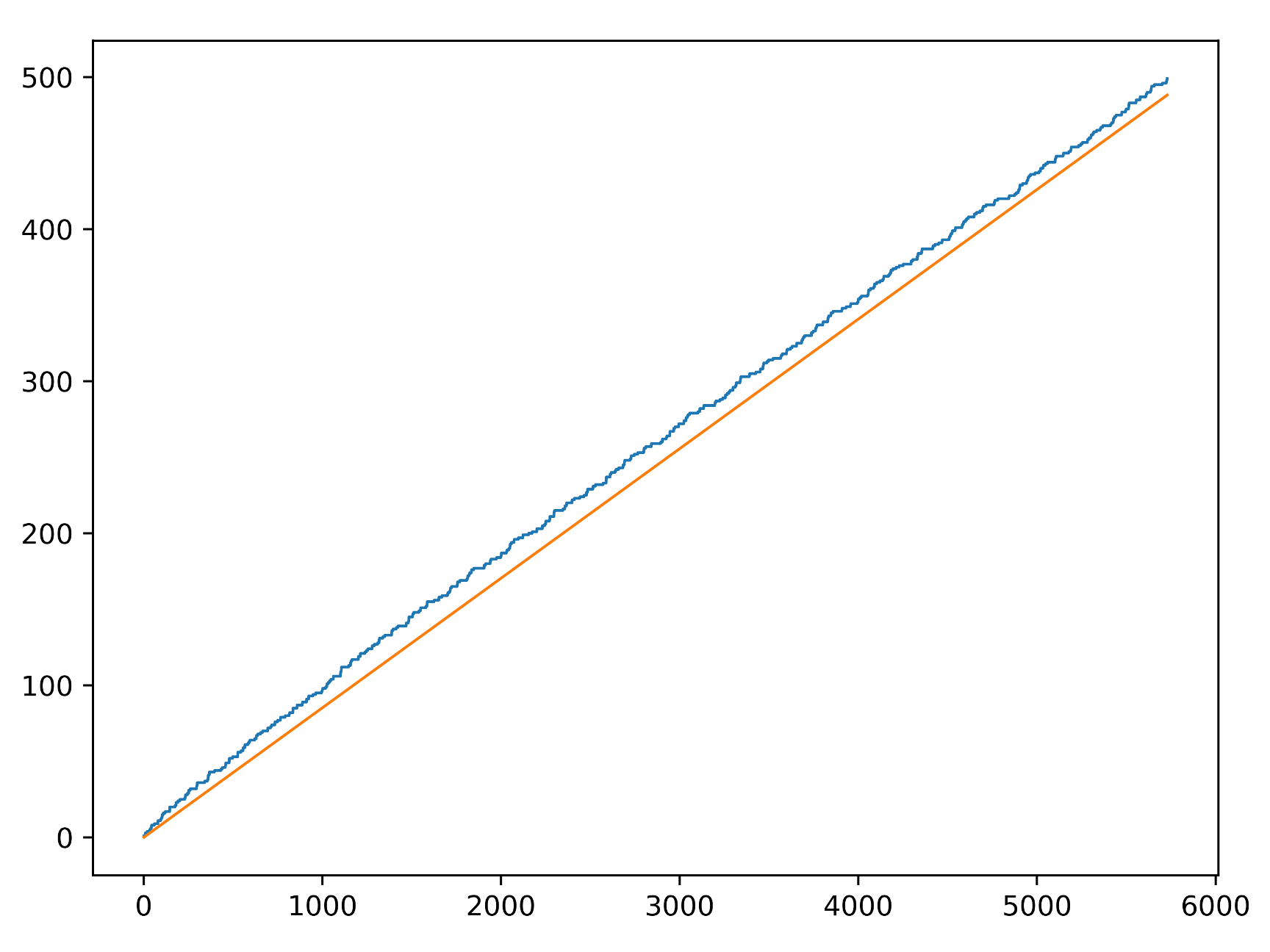}
    }
    \hfill
    \subfigure[$D_1(t)$]
    {
        \includegraphics[width=1.81in]{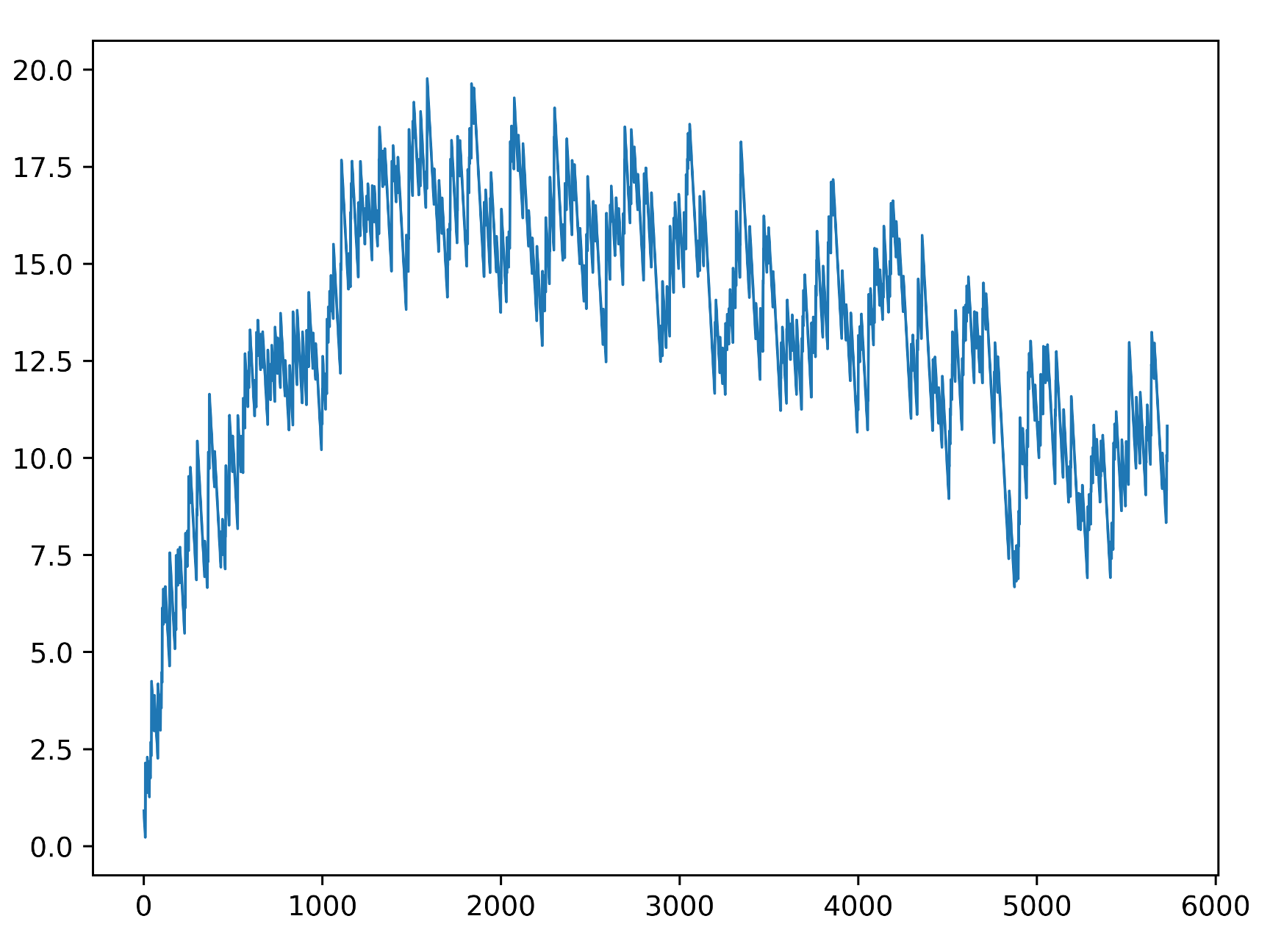}
    }
    \subfigure[$D_2(t)$]
    {
        \includegraphics[width=1.81in]{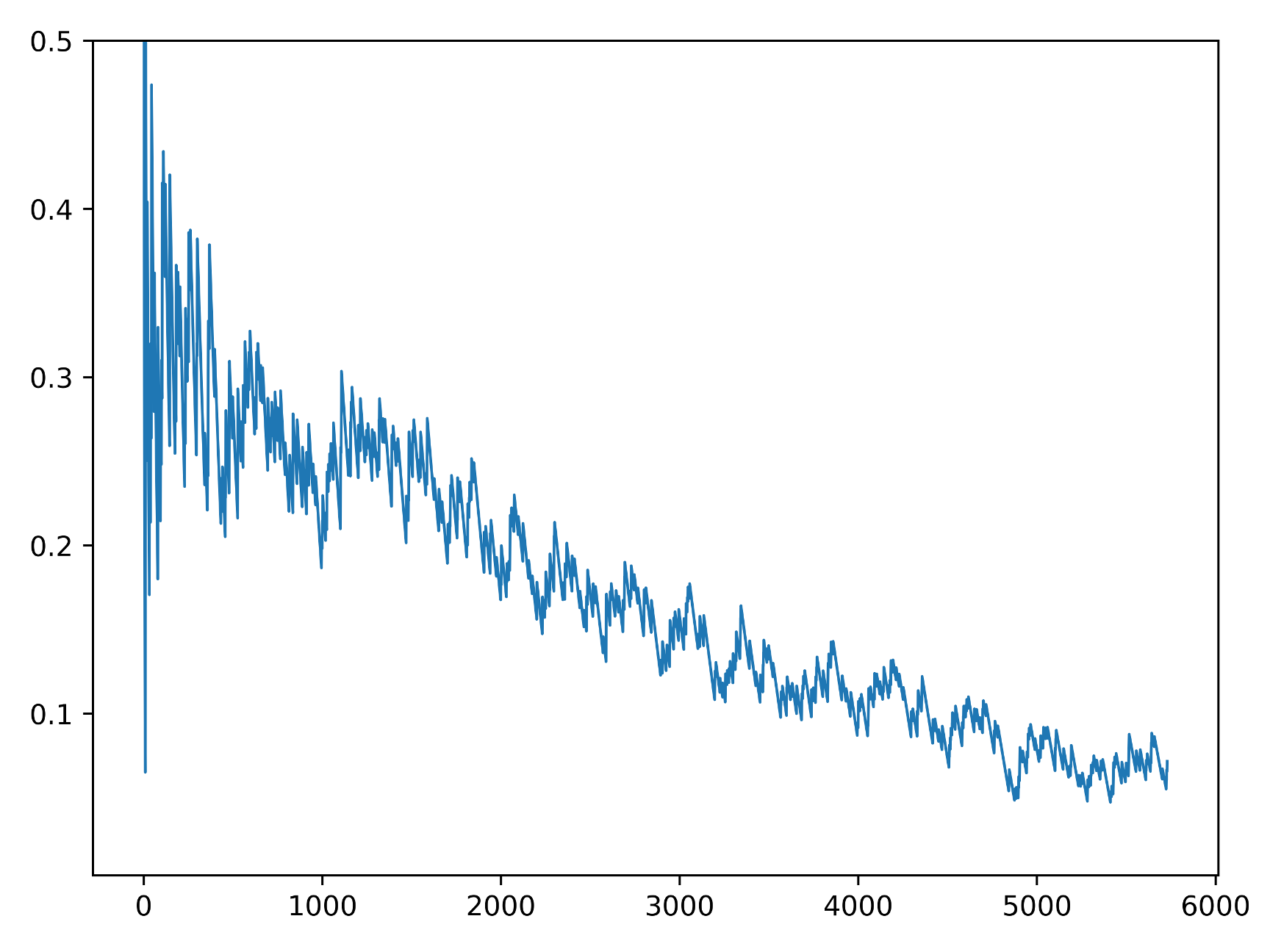}
    }
    \hfill
    \subfigure[$D_2(t)$ (zoom)]
    {
        \includegraphics[width=1.81in]{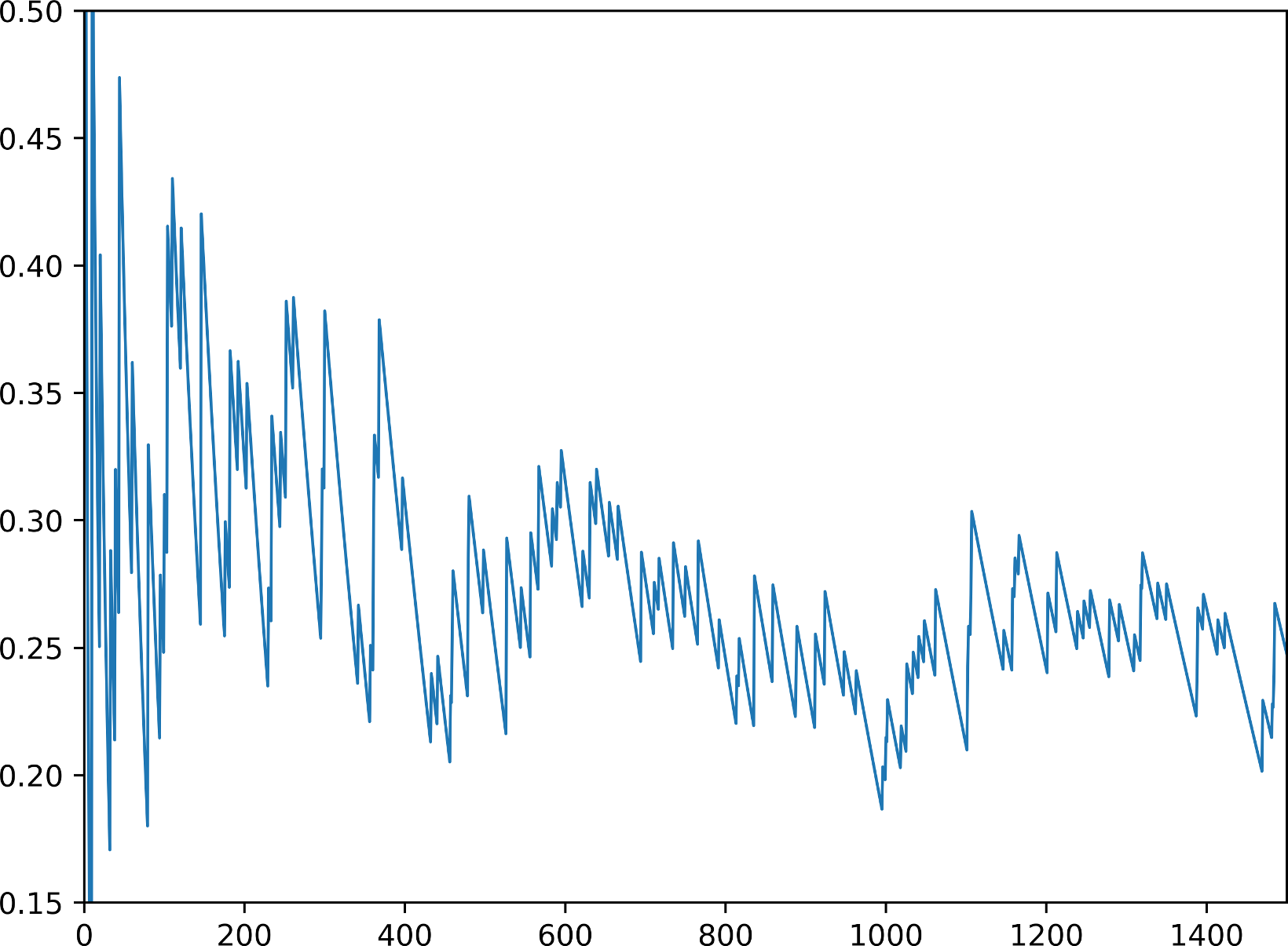}
    }
    \caption{Quadratic snowflake (Level 4, $a=0.45$) with Neumann BC}
\end{figure}

\clearpage

\begin{figure}[h!]
    \centering
    \subfigure[$N(t)$ (blue) and $\frac{A}{4\pi}t$ (orange)]
    {
        \includegraphics[width=1.81in]{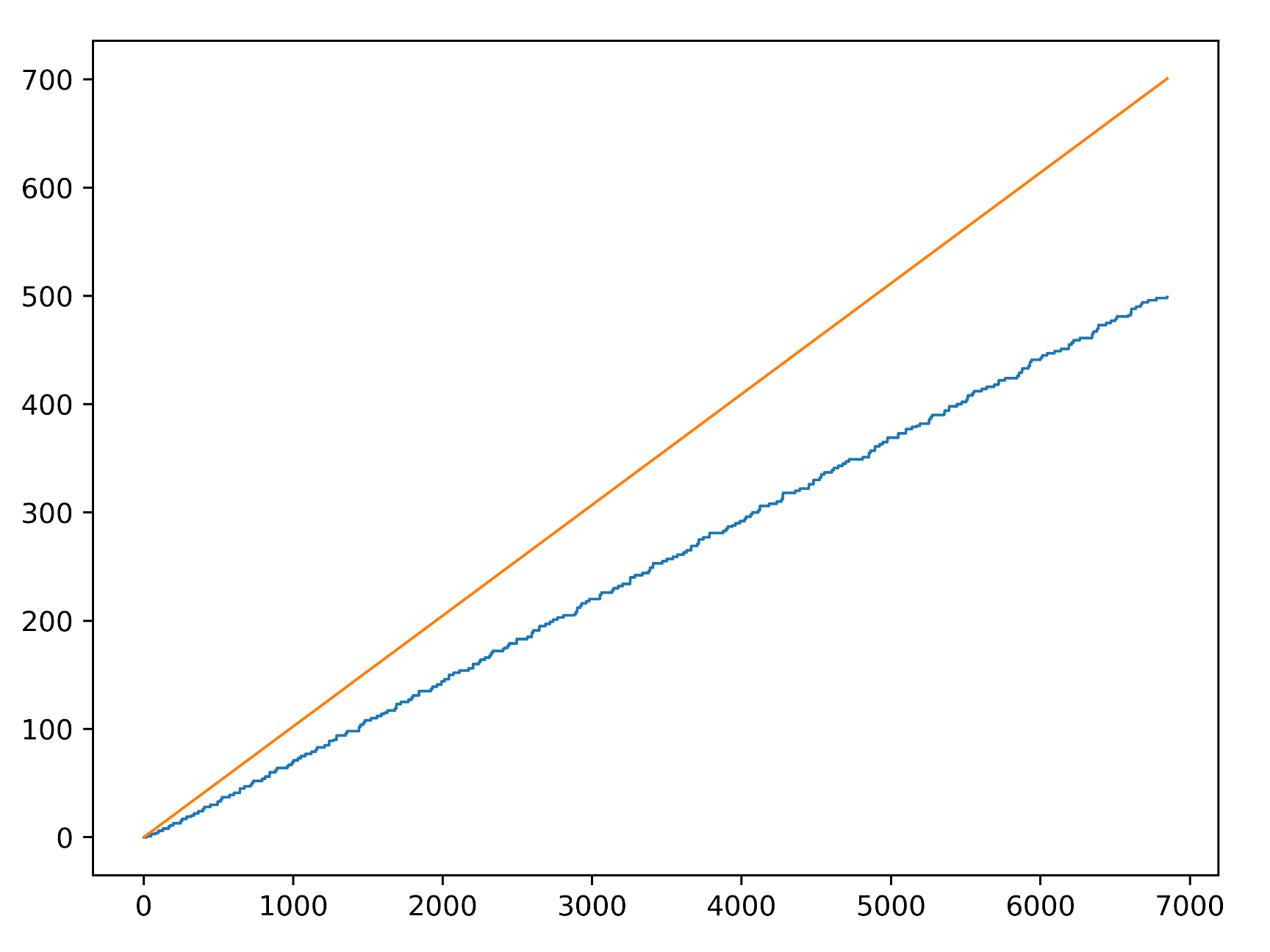}
    }
    \hfill
    \subfigure[$D_1(t)$]
    {
        \includegraphics[width=1.81in]{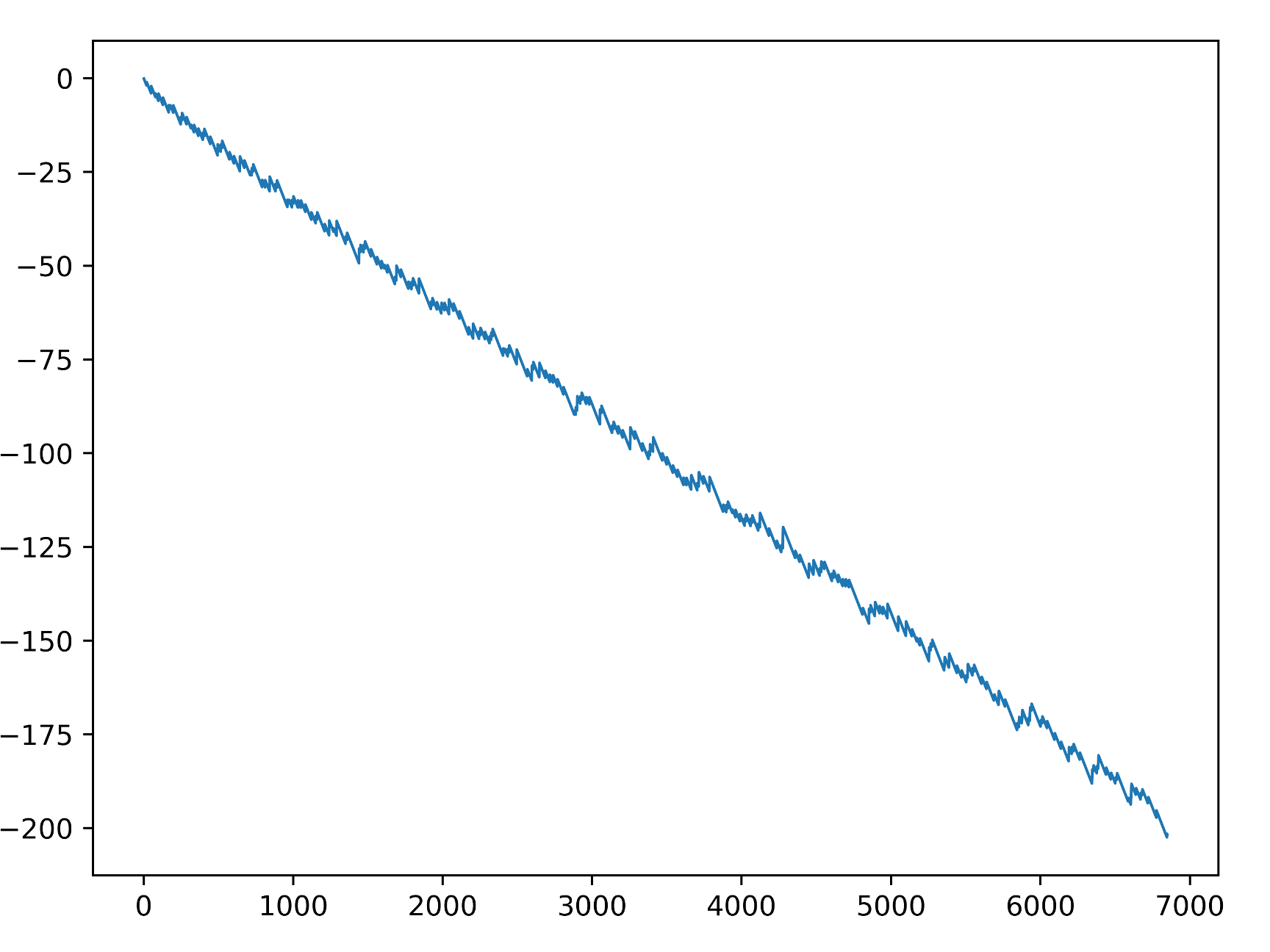}
    }
    \subfigure[$D_2(t)$]
    {
        \includegraphics[width=1.81in]{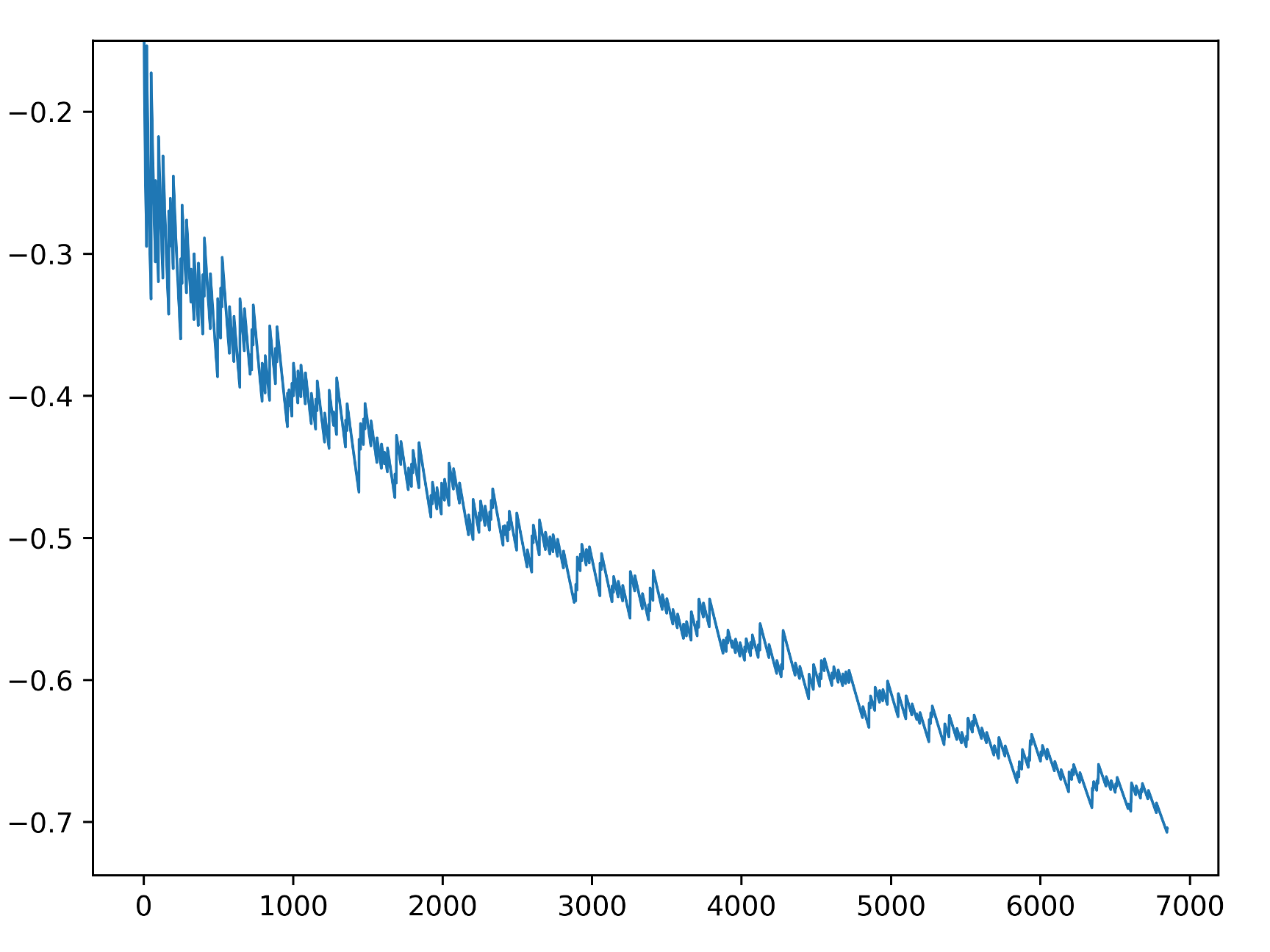}
    }
    \hfill
    \subfigure[$D_2(t)$ (zoom)]
    {
        \includegraphics[width=1.81in]{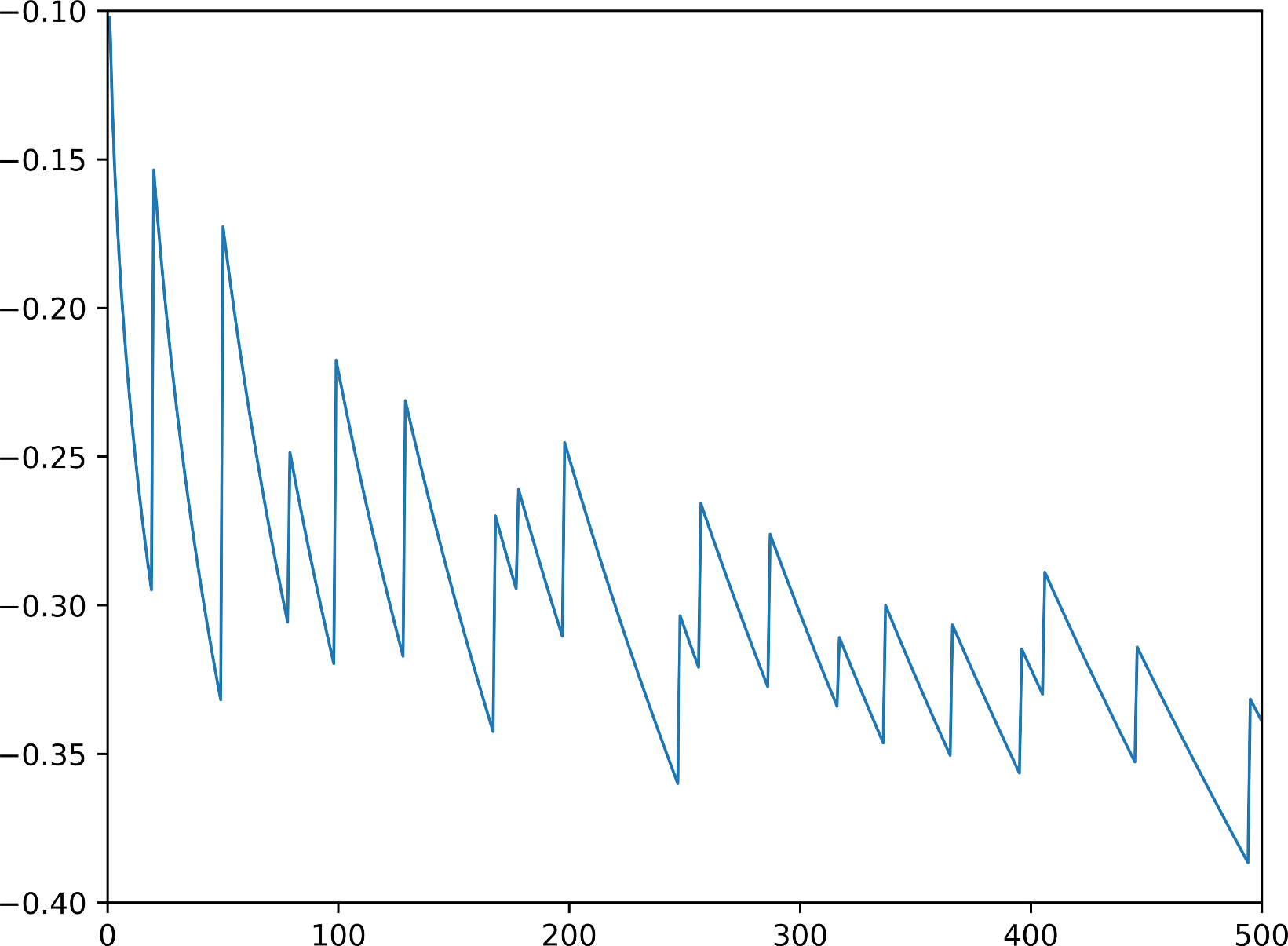}
    }
    \caption{Quadratic snowflake (Level 4, $a=0.4$) with Dirichlet BC}
\end{figure}

\begin{figure}[h!]
    \centering
    \subfigure[$N(t)$ (blue) and $\frac{A}{4\pi}t$ (orange)]
    {
        \includegraphics[width=1.81in]{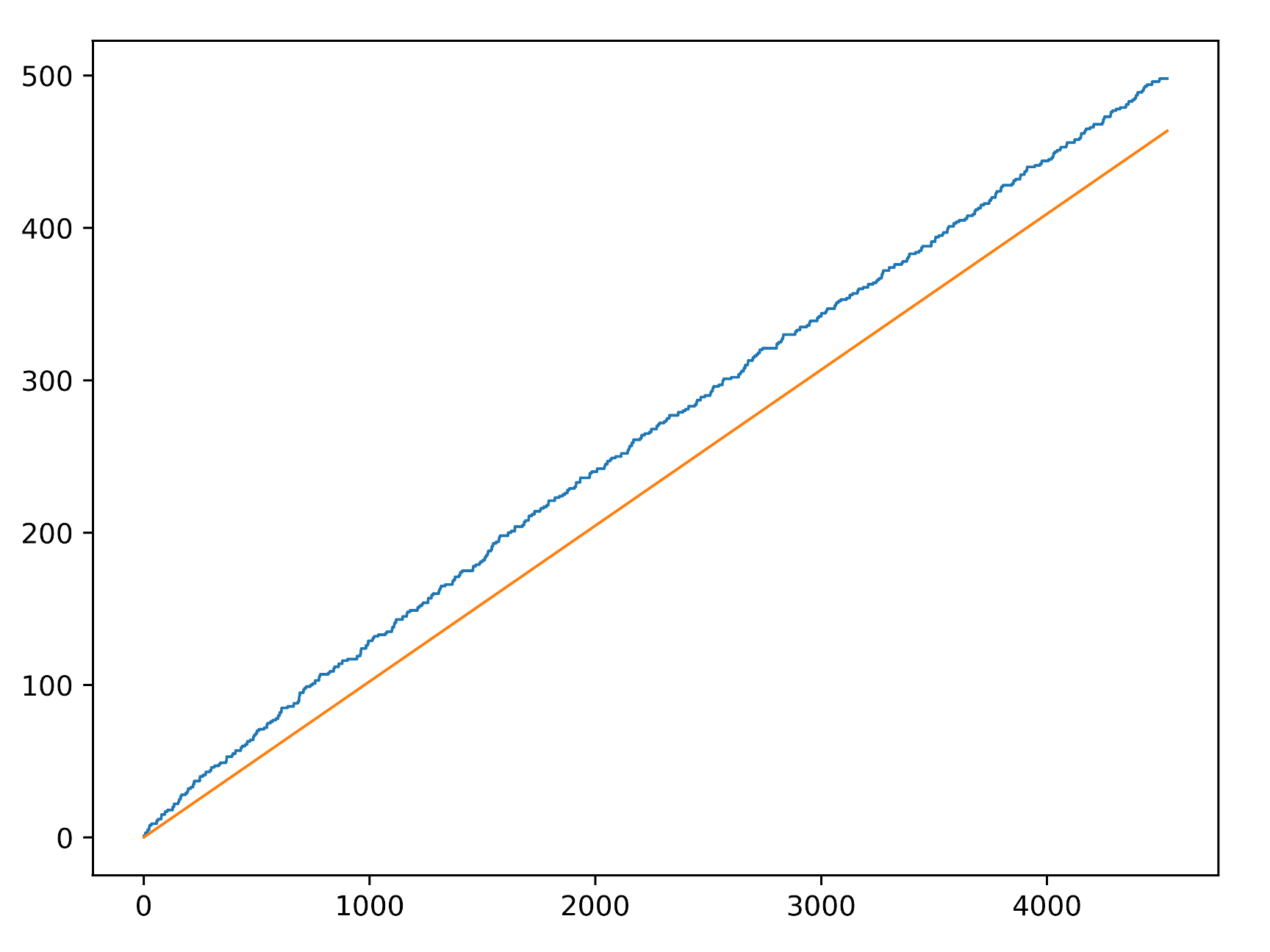}
    }
    \hfill
    \subfigure[$D_1(t)$]
    {
        \includegraphics[width=1.81in]{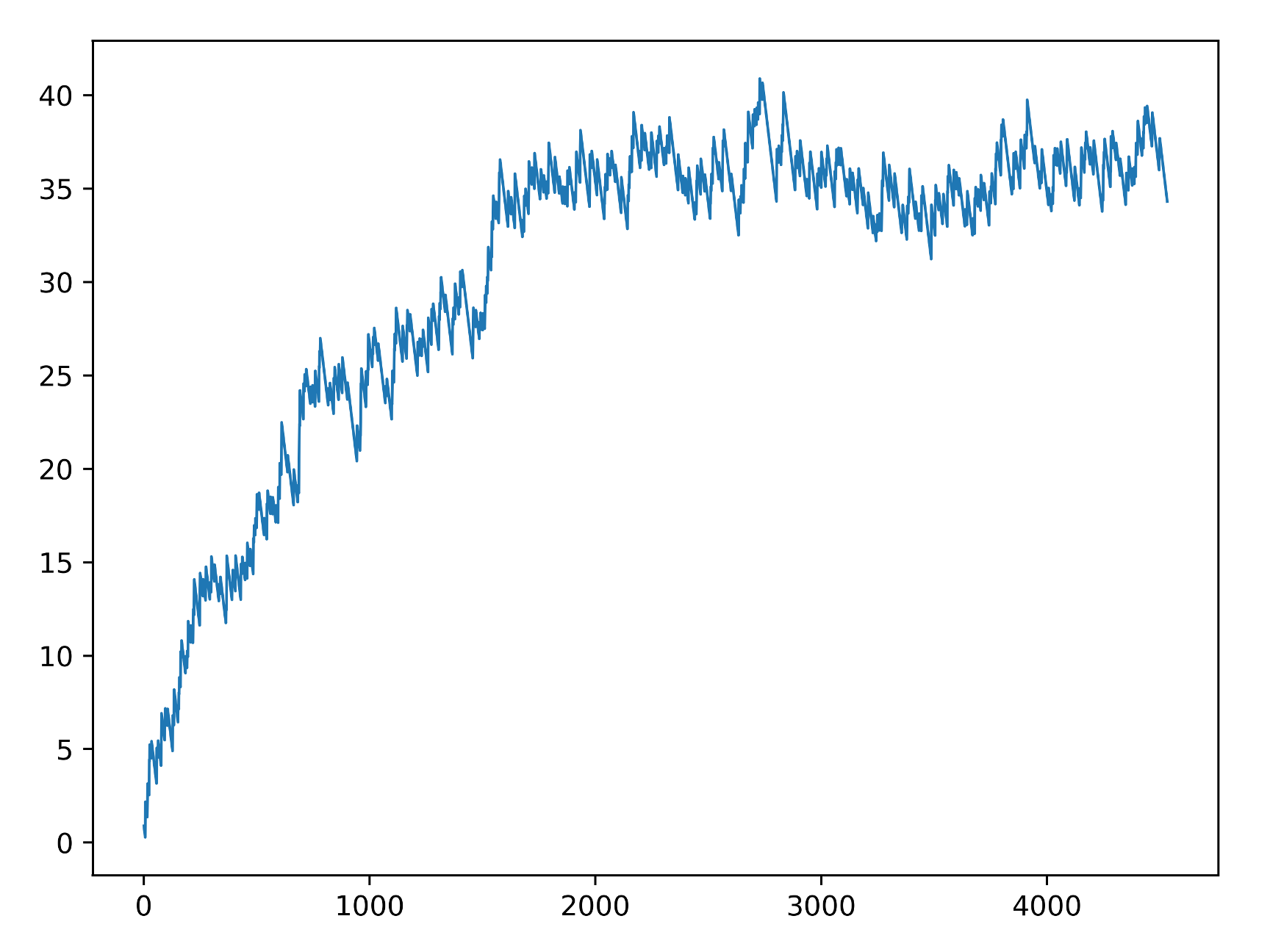}
    }
    \subfigure[$D_2(t)$]
    {
        \includegraphics[width=1.81in]{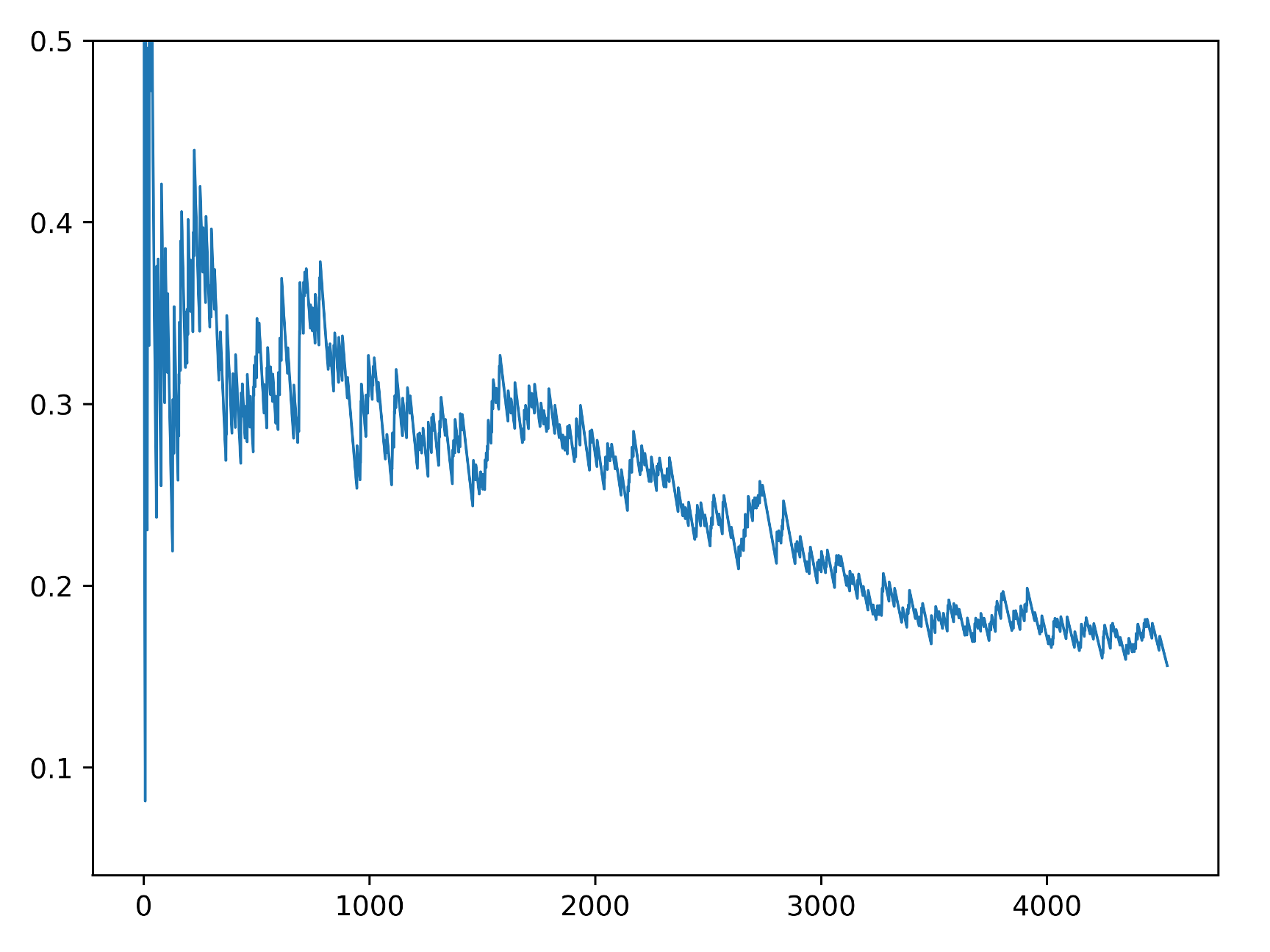}
    }
    \hfill
    \subfigure[$D_2(t)$ (zoom)]
    {
        \includegraphics[width=1.81in]{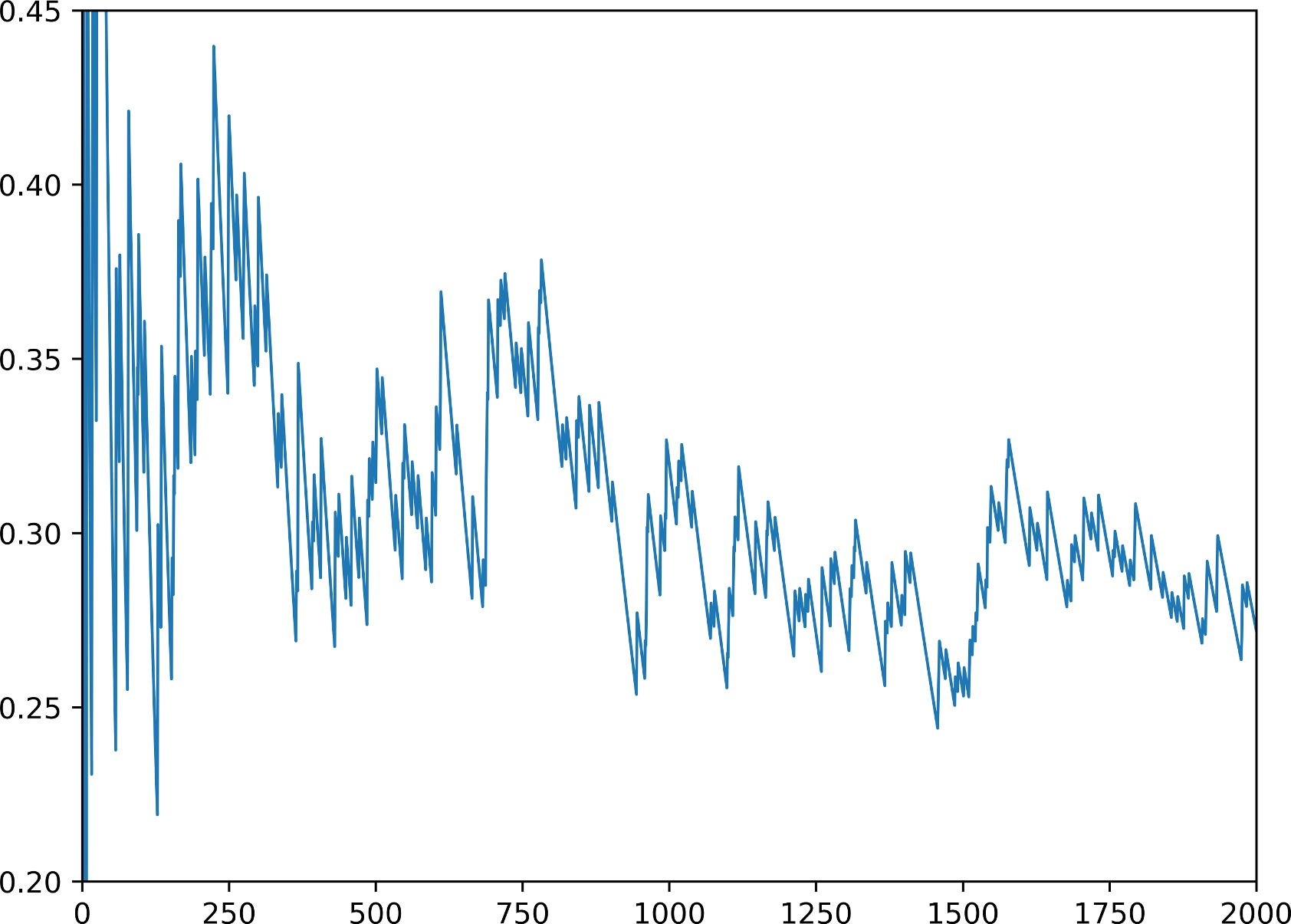}
    }
    \caption{Quadratic snowflake (Level 4, $a=0.4$) with Neumann BC}
\end{figure}

\clearpage

\begin{figure}[h!]
    \centering
    \subfigure[$N(t)$ (blue) and $\frac{A}{4\pi}t$ (orange)]
    {
        \includegraphics[width=1.81in]{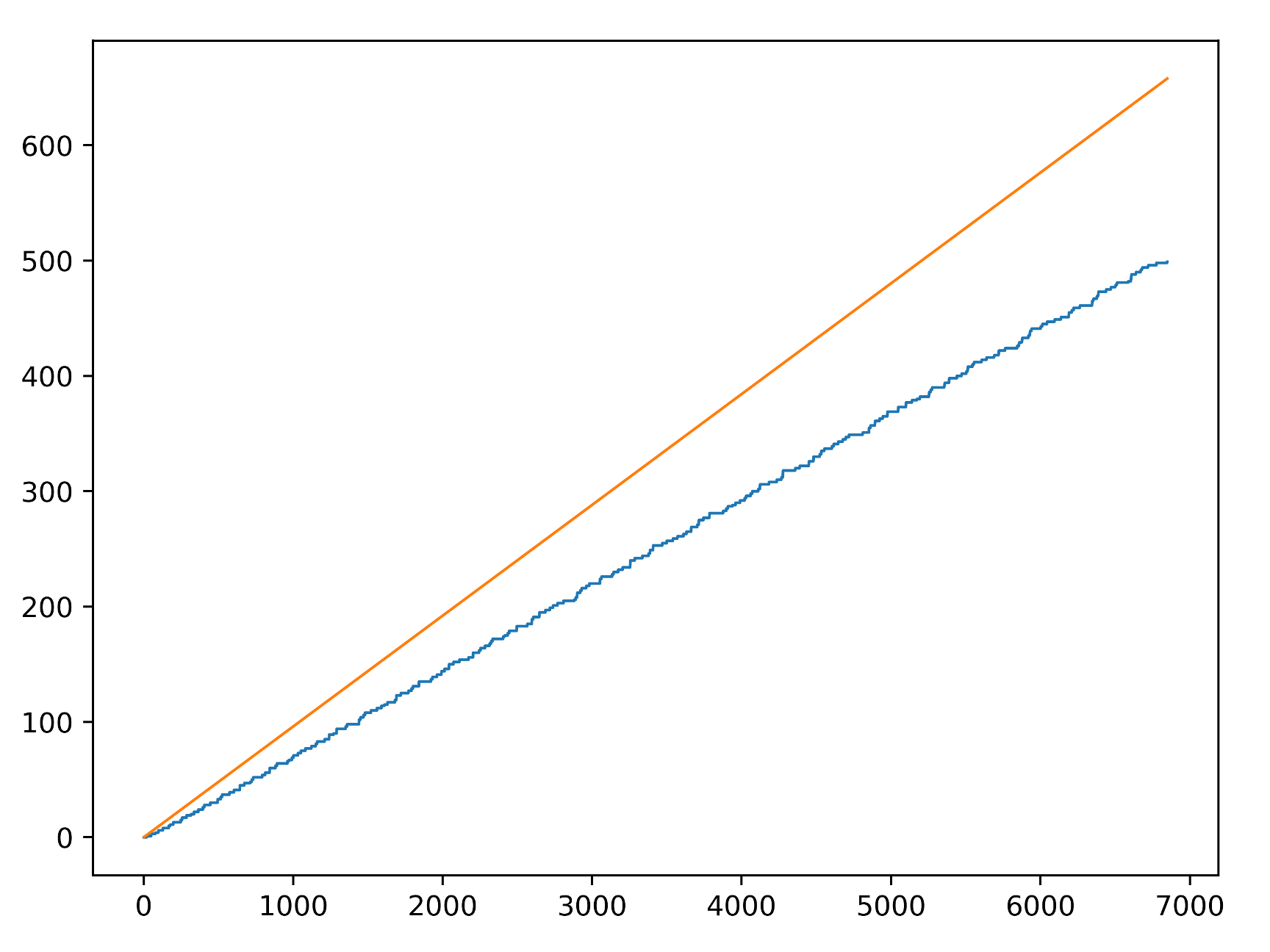}
    }
    \hfill
    \subfigure[$D_1(t)$]
    {
        \includegraphics[width=1.81in]{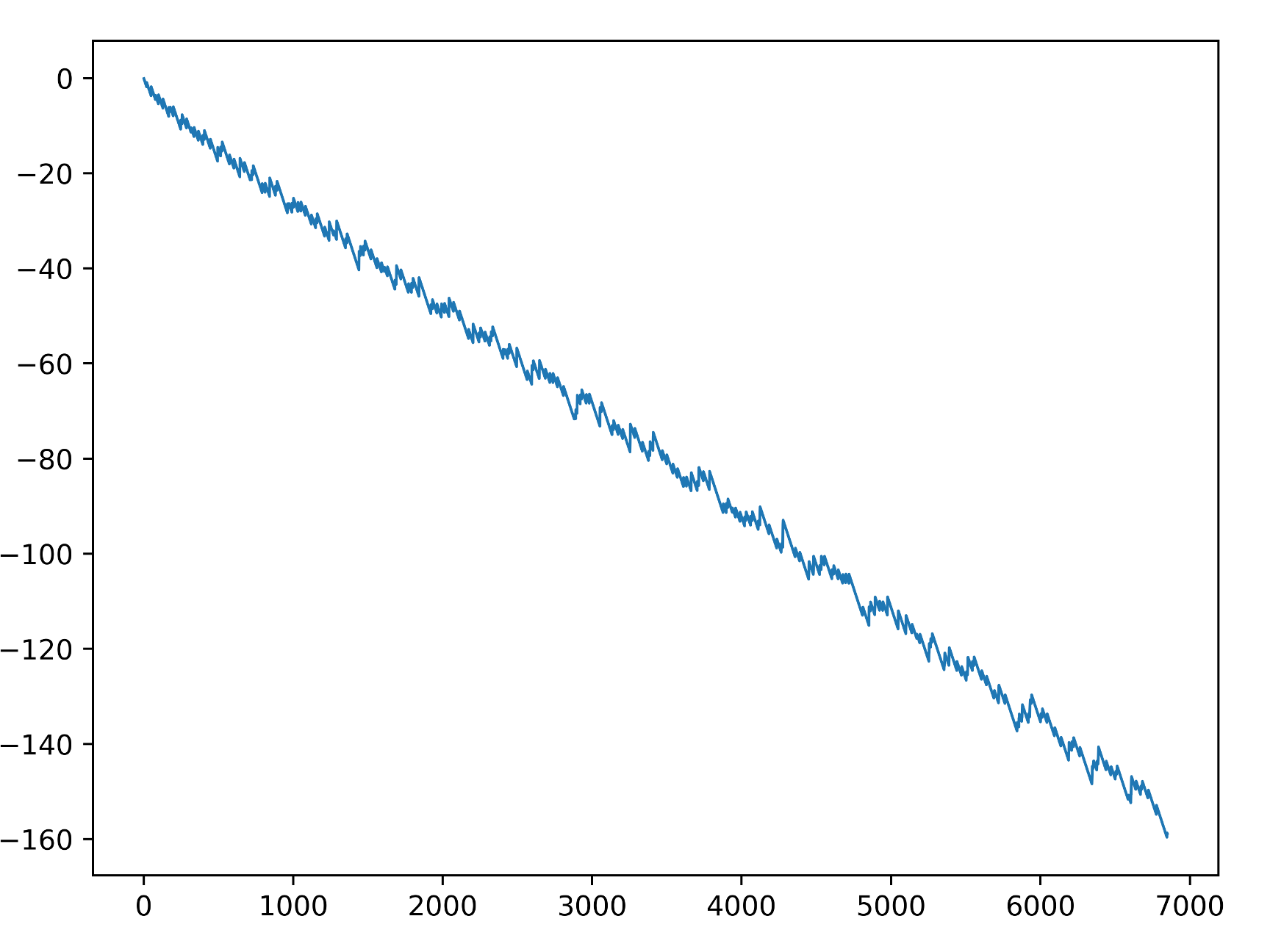}
    }
    \subfigure[$D_2(t)$]
    {
        \includegraphics[width=1.81in]{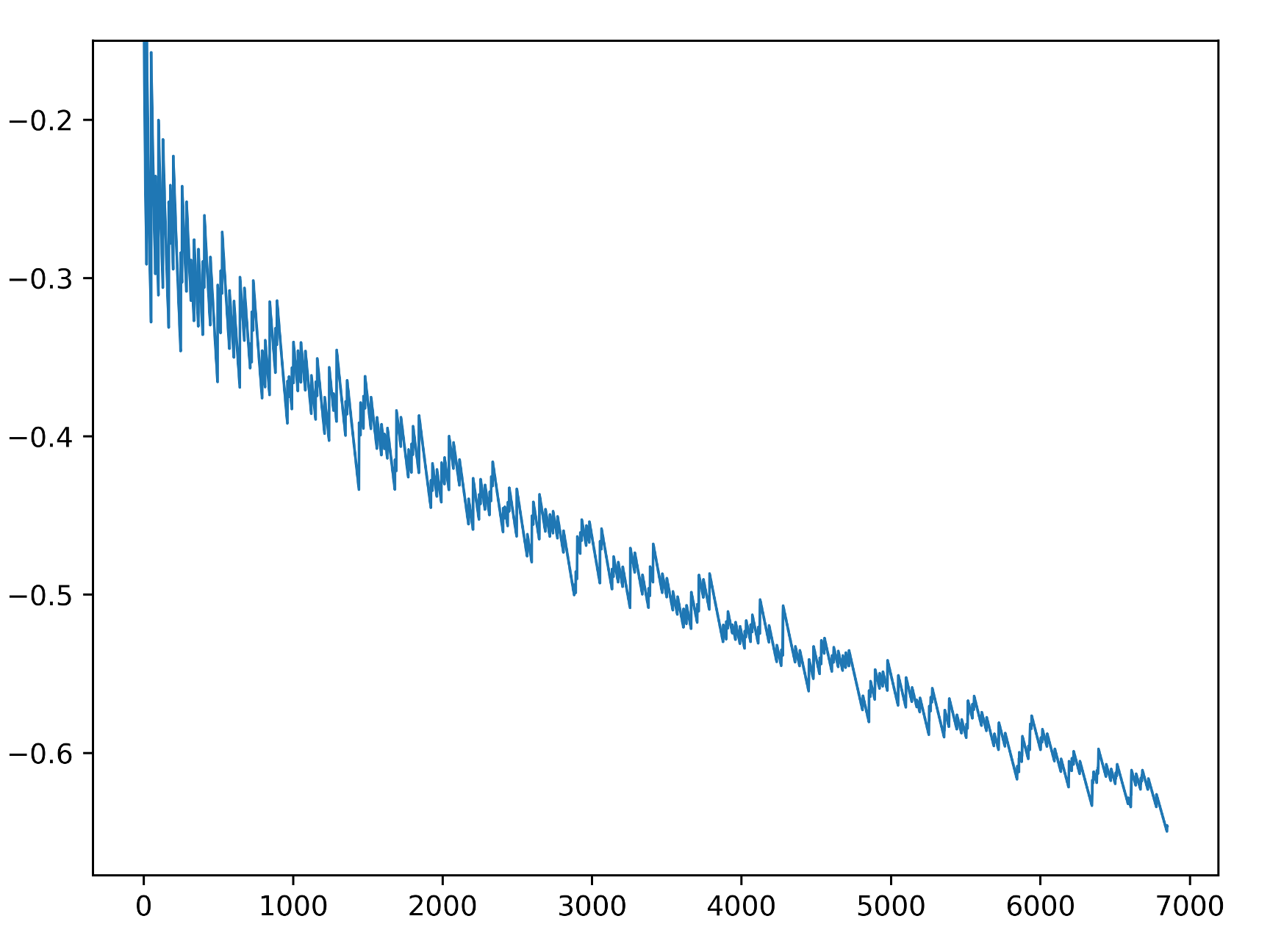}
    }
    \hfill
    \subfigure[$D_2(t)$ (zoom)]
    {
        \includegraphics[width=1.81in]{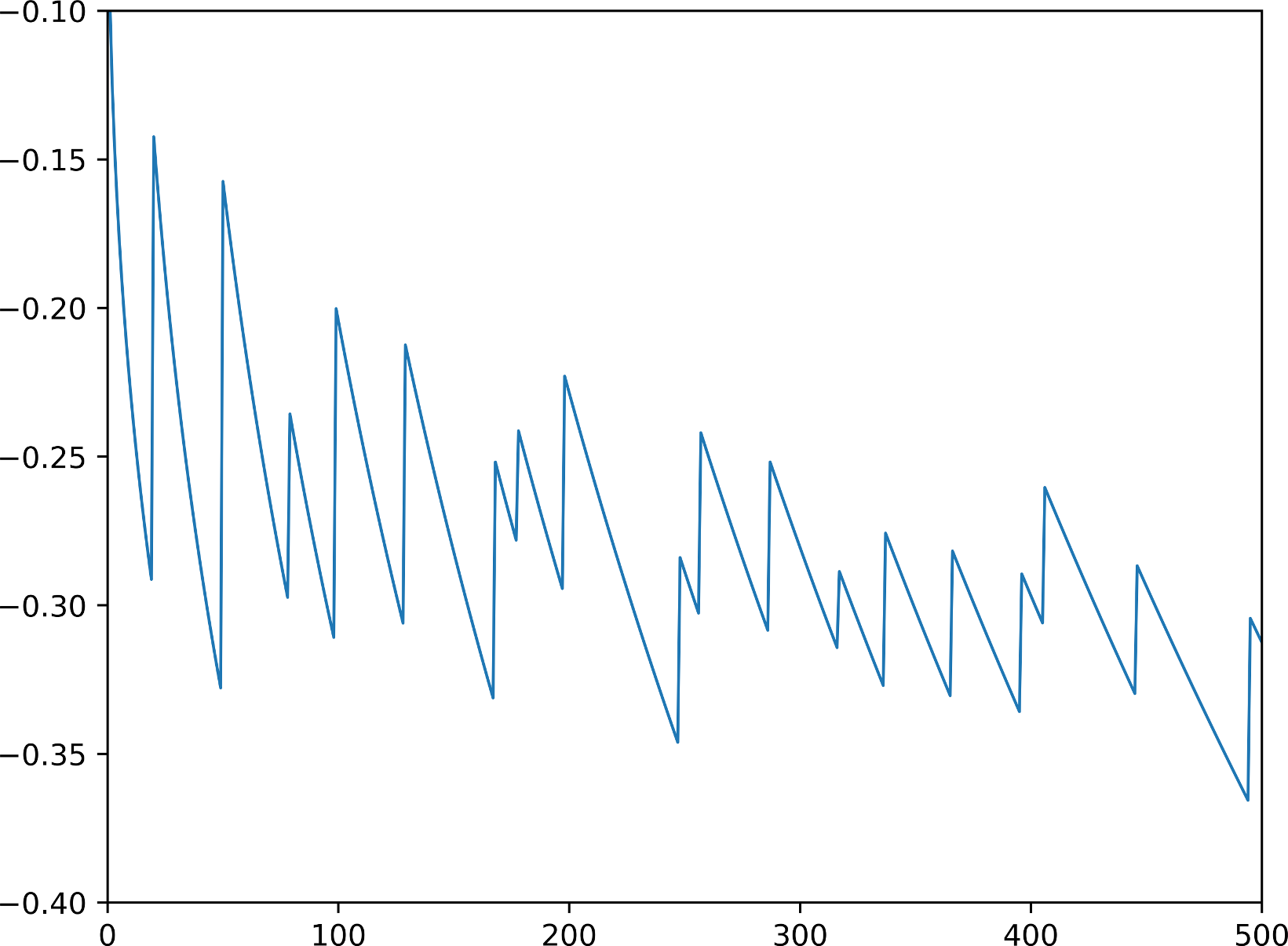}
    }
    \caption{Quad. snowflake (Level 4, $a=\sqrt{2}-1$) with Dirichlet BC}
\end{figure}

\begin{figure}[h!]
    \centering
    \subfigure[$N(t)$ (blue) and $\frac{A}{4\pi}t$ (orange)]
    {
        \includegraphics[width=1.81in]{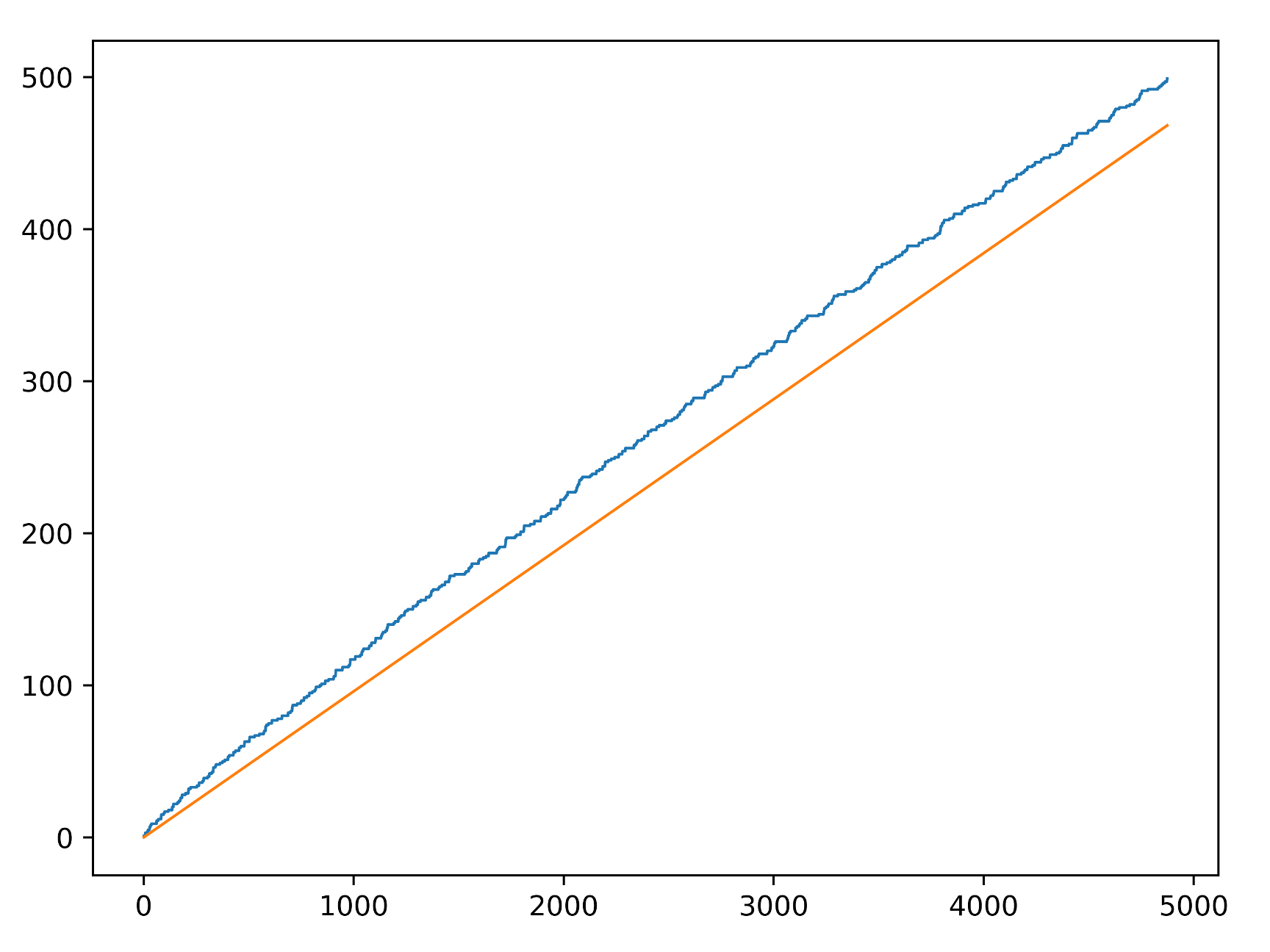}
    }
    \hfill
    \subfigure[$D_1(t)$]
    {
        \includegraphics[width=1.81in]{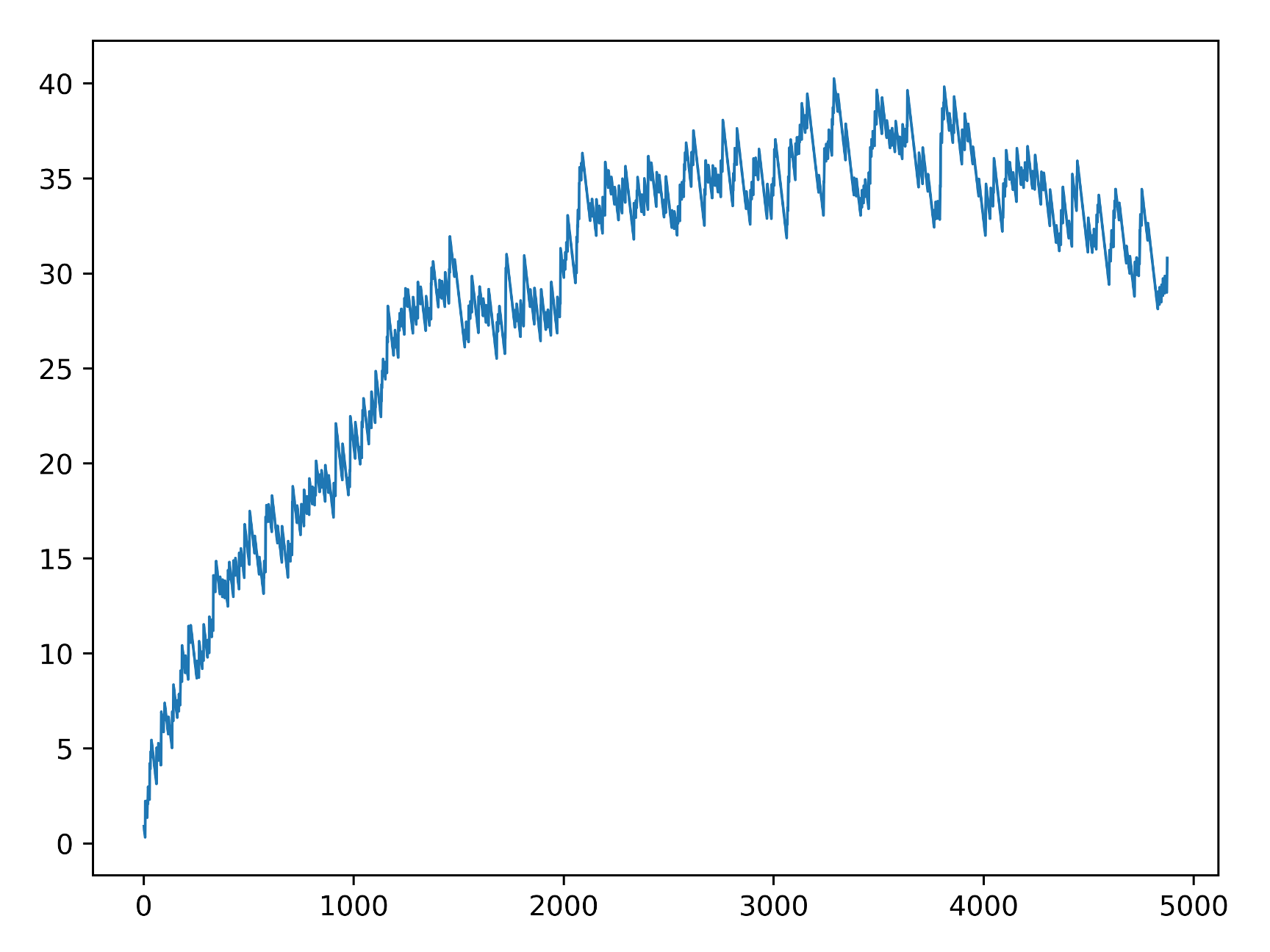}
    }
    \subfigure[$D_2(t)$]
    {
        \includegraphics[width=1.81in]{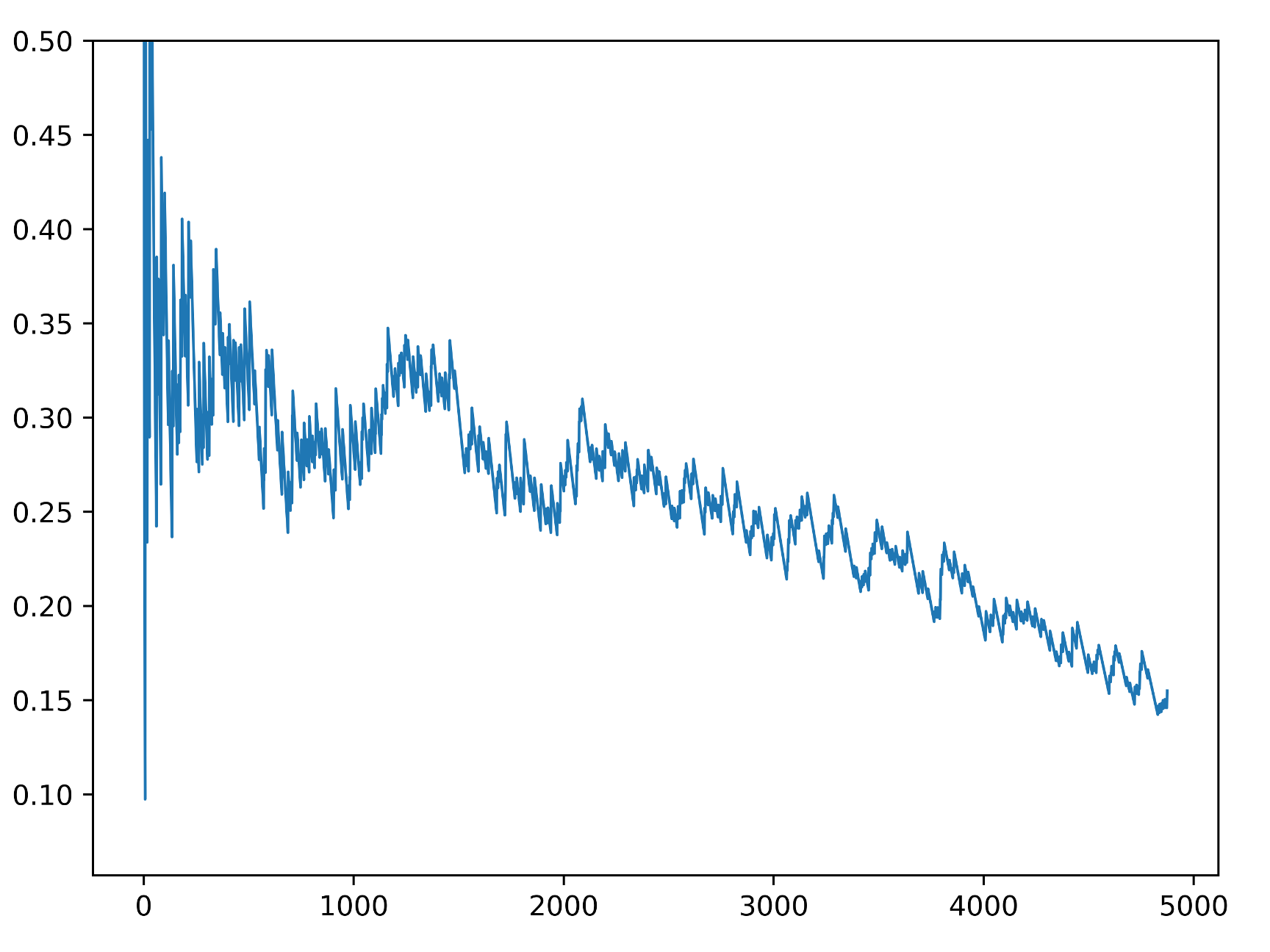}
    }
    \hfill
    \subfigure[$D_2(t)$ (zoom)]
    {
        \includegraphics[width=1.81in]{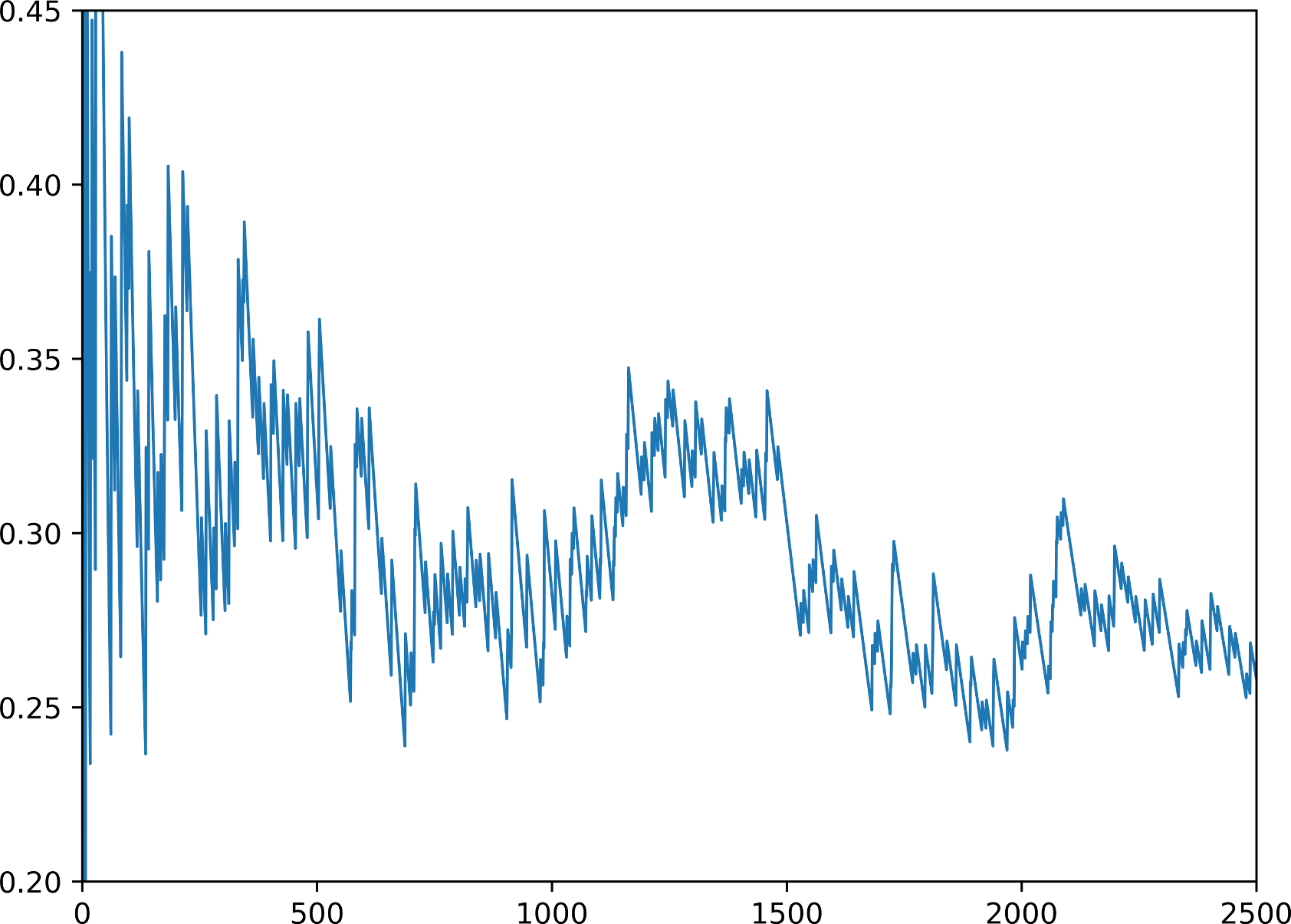}
    }
    \caption{Quad. snowflake (Level 4, $a=\sqrt{2}-1$) with Neumann BC}
    \label{fig:q3n}
\end{figure}

\clearpage

\section{Snowflake eigenfunctions.}

Fig. \ref{fig:sfd} shows three chosen eigenfunctions on the classic snowflake with Dirichlet boundary conditions. In (a) and (b) are eigenfunctions from an eigenspace with multiplicity 2; (a) is symmetric under both reflections; (b) skew-symmetric under both reflections. In (c) is an eigenfunction from the following one-dimensional eigenspace with $D_6$ symmetry. In fact, Neuberger et al. showed in \cite{NSS} that all Dirichlet eigenfunctions on the classical snowflake from an one-dimensional eigenspace have D6 symmetry and all Dirichlet eigenfunctions from a two-dimensional eigenspace are symmetric under both reflections. There are two types of each.

\begin{figure}[h!]
    \centering
    \subfigure[For $\lambda_4=165.96$]
    {
        \includegraphics[width=1.4in]{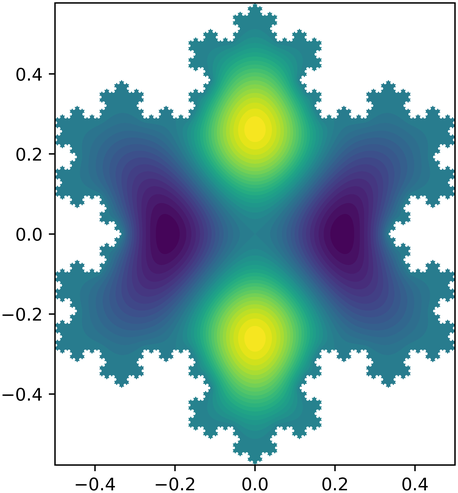}
    }
    \subfigure[For $\lambda_5=165.96$]
    {
        \includegraphics[width=1.4in]{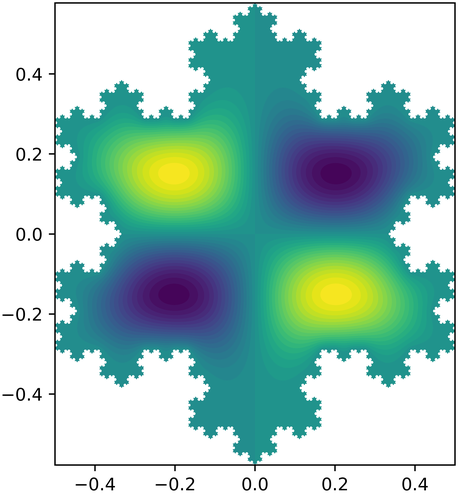}
    }
    \subfigure[For $\lambda_6=191.01$]
    {
        \includegraphics[width=1.4in]{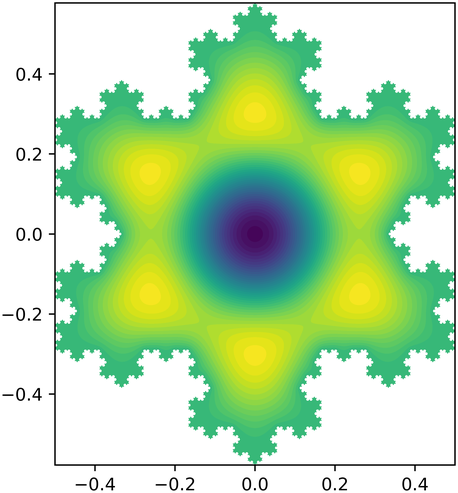}
    }
    \caption{Eigenfunctions on the C. Snowflake (Lv. 6) with DBC}
    \label{fig:sfd}
\end{figure}

Fig. \ref{fig:sfn} shows three chosen eigenfunctions on the classic snowflake with Neumann boundary conditions. In (a) and (b) are eigenfunctions from an eigenspace with multiplicity 2; (a) is symmetric under both reflections; (b) skew-symmetric under both reflections. In (c) is an eigenfunction from the following one-dimensional eigenspace that is skew-symmetric under $60°$ rotation and symmetric under $120°$ rotation. Neuberger et al. showed in \cite{NSS} that all Neumann eigenfunctions on the classical snowflake from an one-dimensional eigenspace have D6 symmetry and all Dirichlet eigenfunctions from a two-dimensional eigenspace are have the symmetric properties as in the picture.

\begin{figure}[h!]
    \centering
    \subfigure[For $\lambda_4=23.32$]
    {
        \includegraphics[width=1.4in]{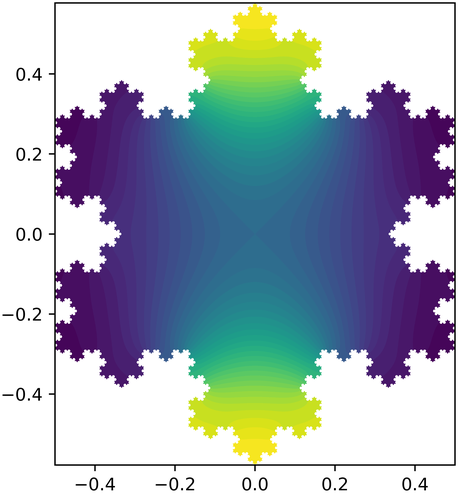}
    }
    \subfigure[For $\lambda_5=23.32$]
    {
        \includegraphics[width=1.4in]{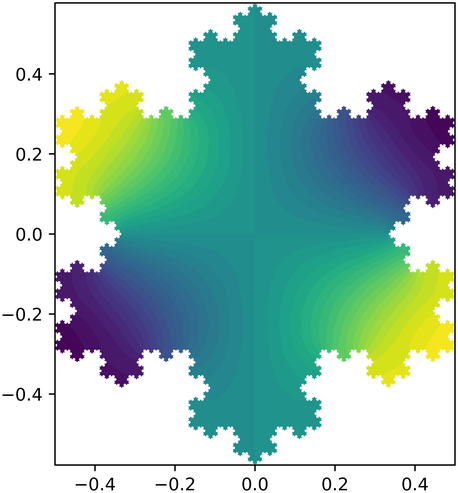}
    }
    \subfigure[For $\lambda_6=27.79$]
    {
        \includegraphics[width=1.4in]{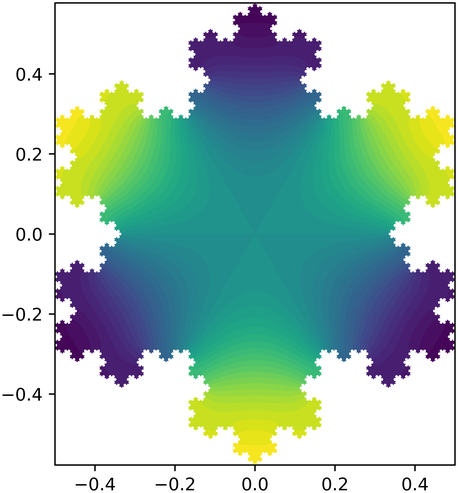}
    }    
    \caption{Eigenfunction on the C. Snowflake (Lv. 6) with NBC}
    \label{fig:sfn}
\end{figure}

\clearpage


We computed the energy-distribution of each eigenfunction by assigning the same weight to each cell - the sum of the squared normal derivative at each edge. It is clear that for both Dirichlet and Neumann conditions, the energy function of an eigenfunction from an one-dimensional eigenspace have $D_6$ symmetry (Figure \ref{fig:energy1}) and all from a two-dimensional eigenspace are symmetric under both reflections (Figure \ref{fig:energy2}).

\vspace{0.5cm}

\begin{figure}[h]
        \includegraphics[width=0.7\linewidth]{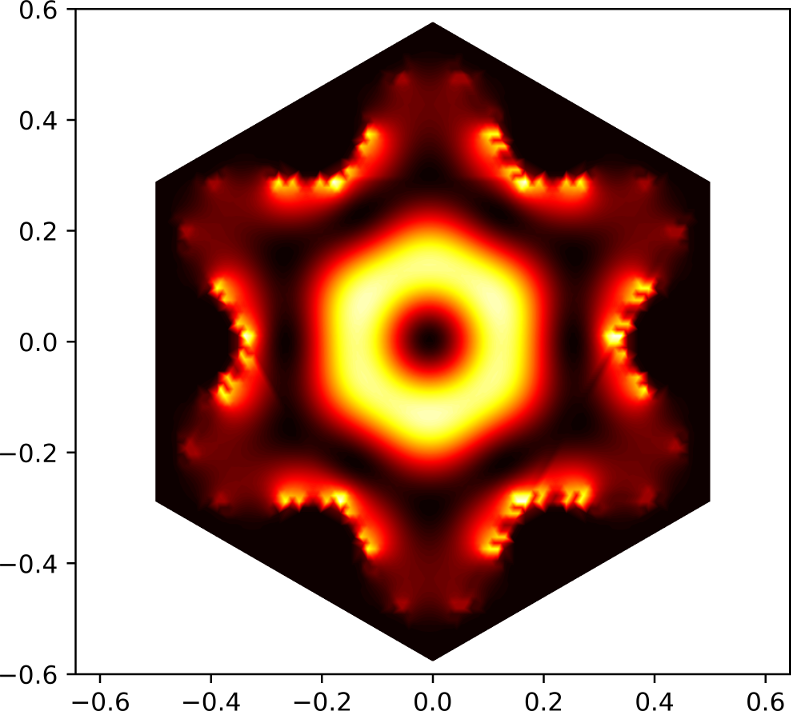}
    \caption{Energy Function on the C. Snowflake (Lv. 6) with DBC for $\lambda_6=192.41$}
    \label{fig:energy1}
\end{figure}

\begin{figure}[h]
    \centering
    \subfigure[For $\lambda_4=23.83$]
    {
        \includegraphics[width=2.2in]{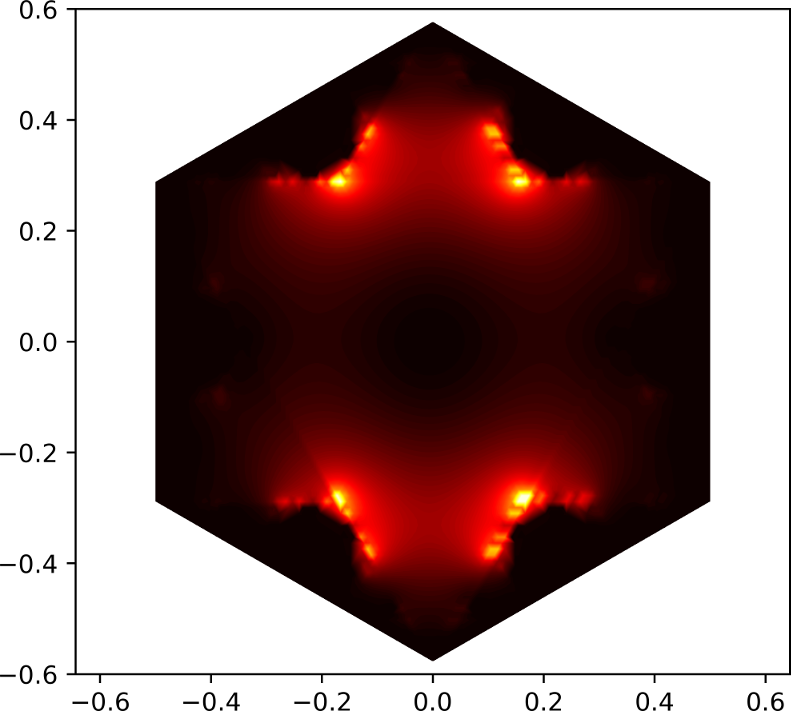}
    }
    \hfill
    \subfigure[For $\lambda_5=23.83$]
    {
        \includegraphics[width=2.2in]{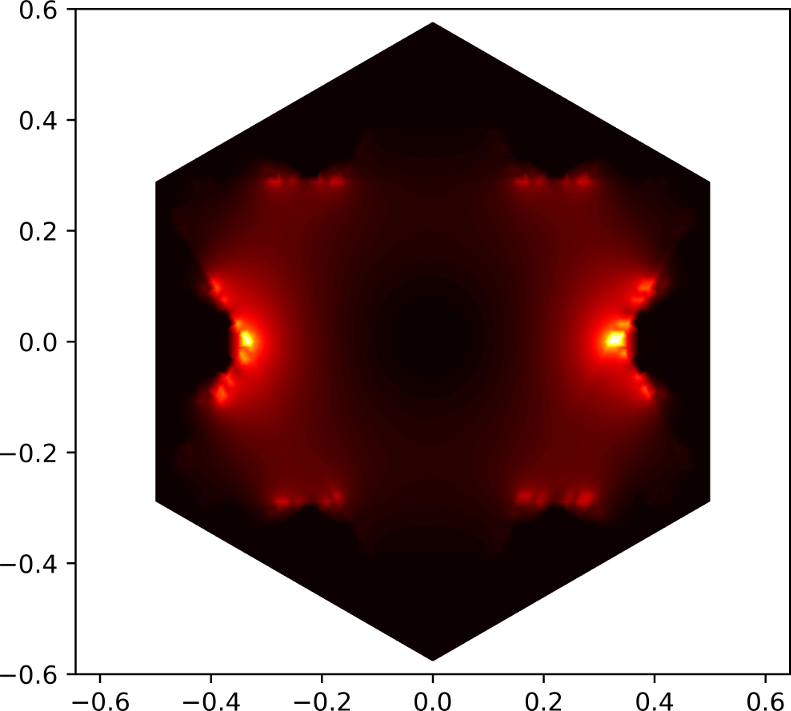}
    }
    \caption{Energy Functions on the C. Snowflake (Lv. 6) with NBC}
    \label{fig:energy2}
\end{figure}

\clearpage

We observe an interesting phenomenom for the Neumann eigenfunctions: it seems like most energy functions can be constructed (at least near the boundary) by choosing a linear combination of the two energy functions of the eigenfunctions from the same eigenspace that came from the previous. In the shown example, a linear combination of eigenfunction $4$ and $5$ looks like a $6$-eigenfunction with Neumann boundary conditions (Figure \ref{fig:energysum1}), and a linear combination of eigenfunctions $23$ and $24$ looks like a $25$-eigenfunction (Figure \ref{fig:energysum2}). By linear combination we mean taking a linear combination of the energies at each individual cell.

\vspace{0.2cm}

\begin{figure}[h]
    \centering
    \subfigure[For $\lambda_6=28.47$]
    {
        \includegraphics[width=2.2in]{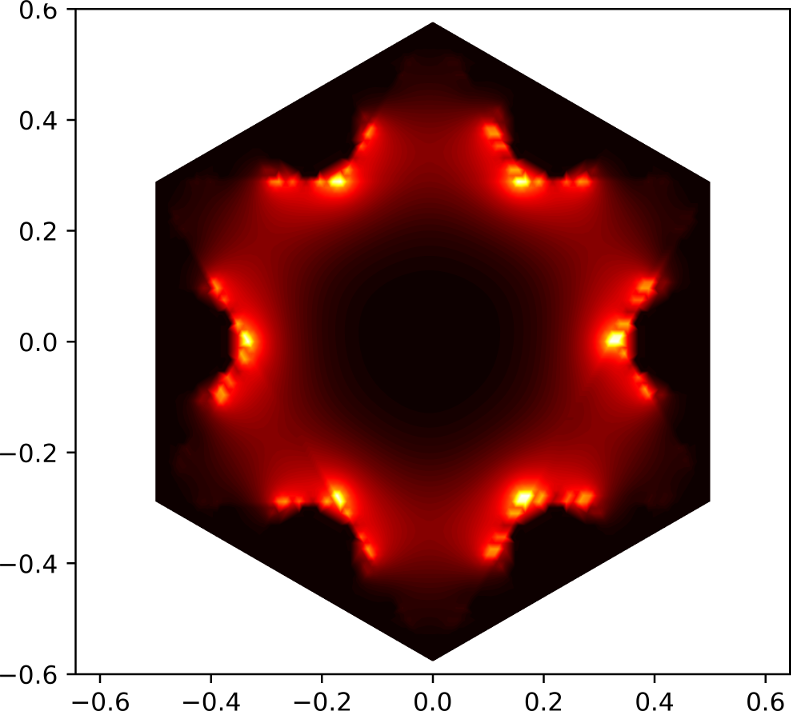}
    }
    \hfill
    \subfigure[Sum of Energy Functions for $\lambda_4$ and $\lambda_5$]
    {
        \includegraphics[width=2.2in]{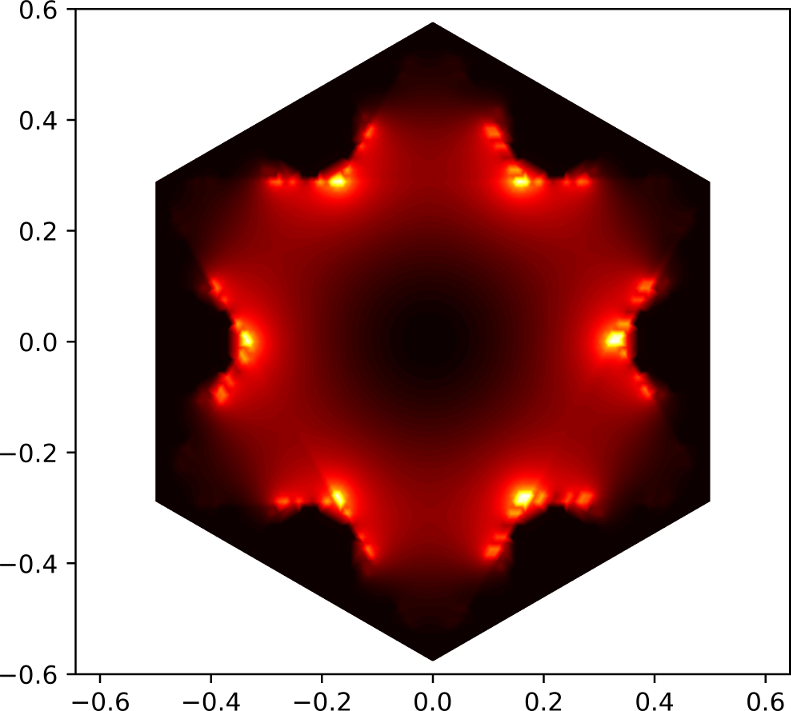}
    }
    \caption{Energy Functions on the C. Snowflake (Lv. 6) with NBC}
    \label{fig:energysum1}
\end{figure}

\begin{figure}[h]
    \centering
    \subfigure[For $\lambda_{25}=244.18$]
    {
        \includegraphics[width=2.2in]{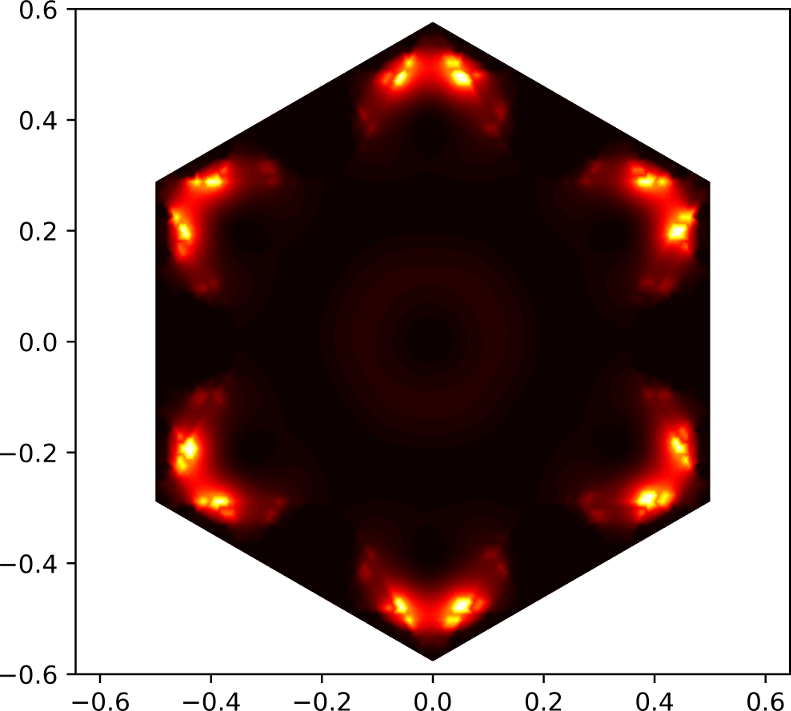}
    }
    \hfill
    \subfigure[Sum of Energy Functions for $\lambda_{23}$ and $\lambda_{24}$]
    {
        \includegraphics[width=2.2in]{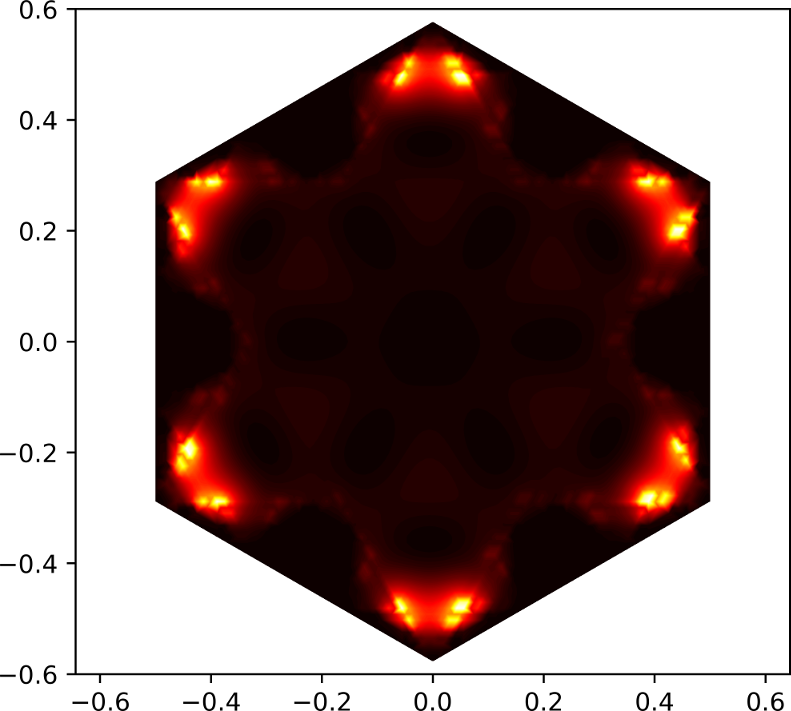}
    }
    \caption{Energy Functions on the C. Snowflake (Lv. 6) with NBC}
    \label{fig:energysum2}
\end{figure}

\newpage


Figure \ref{fig:qsfd} shows three chosen eigenfunctions on a quadratic snowflake with Dirichlet boundary conditions. In (a) and (b) are eigenfunctions from an eigenspace with multiplicity 2; (a) is symmetric under horizontal reflection and skew-symmetric under vertical reflection; (b) skew-symmetric under horizontal reflection and symmetric under vertical reflection. In (c) is an eigenfunction from the following one-dimensional eigenspace that has $D_4$ symmetry.

\begin{figure}[h]
    \centering
    \subfigure[For $\lambda_{14}=244.27$]
    {
        \includegraphics[width=1.5in]{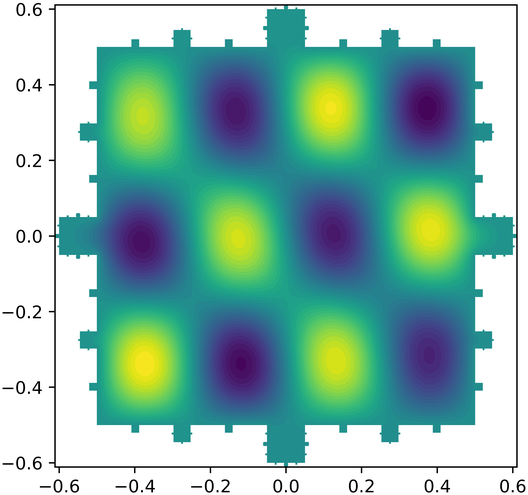}
    }
    \hfill
    \subfigure[For $\lambda_{15}=244.27$]
    {
        \includegraphics[width=1.5in]{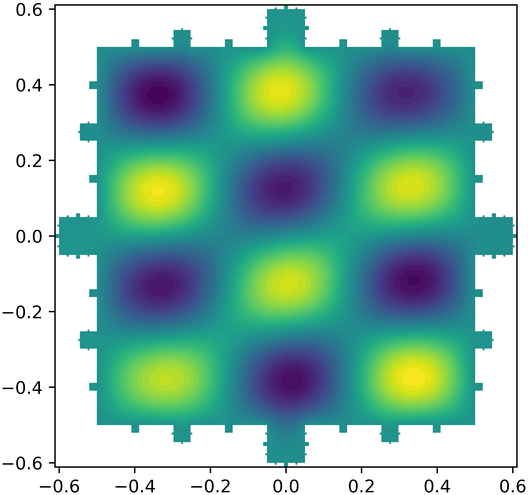}
    }
    \hfill
    \subfigure[For $\lambda_{16}=251.80$]
    {
        \includegraphics[width=1.5in]{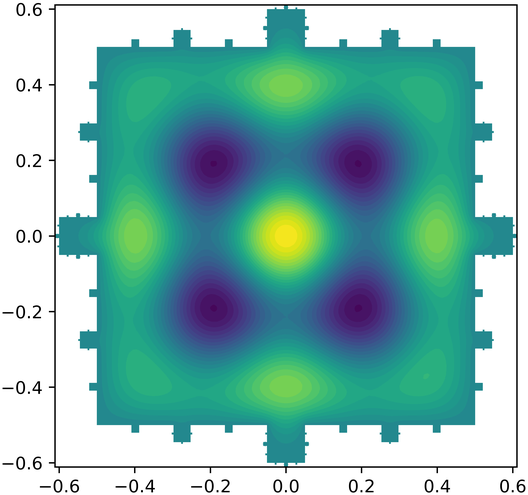}
    }
    \caption{Eigenfunctions on the Quadratic Snowflake (b=0.2, Lv. 4) with Dirichlet BC}
    \label{fig:qsfd}
\end{figure}

Figure \ref{fig:qsfn} shows three chosen eigenfunctions on a quadratic snowflake with Neumann boundary conditions. In (a) and (b) are eigenfunctions from an eigenspace with multiplicity 2; these are symmetric under one diagonal reflection and skew-symmetric under the other. In (c) is an eigenfunction from the following one-dimensional eigenspace that is skew-symmetric under both diagonal reflections and skew-symmetric under horizontal and vertical reflection.

\begin{figure}[h]
    \centering
    \subfigure[For $\lambda_{10}=79.33$]
    {
        \includegraphics[width=1.5in]{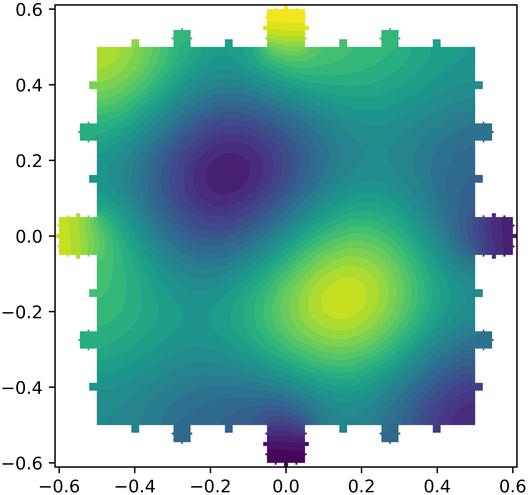}
    }
    \hfill
    \subfigure[For $\lambda_{11}=79.33$]
    {
        \includegraphics[width=1.5in]{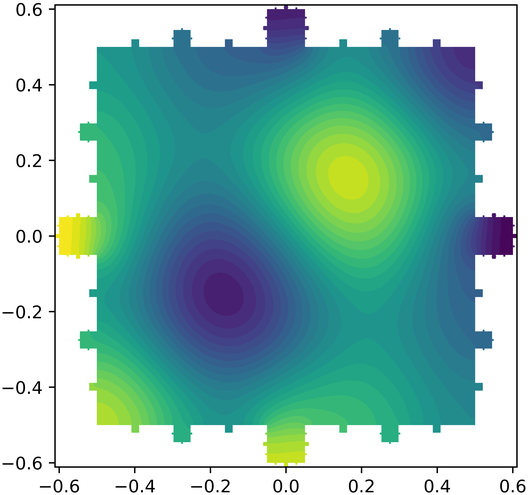}
    }
    \hfill
    \subfigure[For $\lambda_{12}=93.02$]
    {
        \includegraphics[width=1.5in]{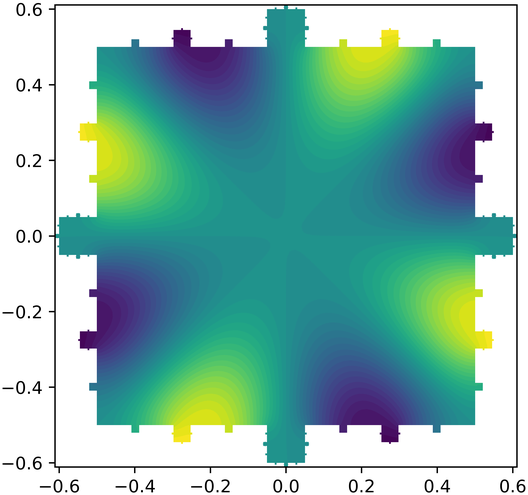}
    }
    \caption{Eigenfunctions on the Quadratic Snowflake (b=0.2, Lv. 4) with Neumann BC}
    \label{fig:qsfn}
\end{figure}

\newpage

\section{Julia Set geometry.}

We will do  a similar analysis on filled Julia sets of quadratic polynomials and compute eigenfunctions with Dirichlet and Neumann boundary conditions.

\begin{definition}
Let $p_c(z)=z^2+c$. The Julia set $J_c$ and the filled Julia set $K_c$ are defined as
\begin{align*}
K_p&=\{z\in\mathbb{C}:\text{ the sequence } z,p_c(z),p^2(z),\dots\text{ is bounded}\}\\
J_p&=\partial K_p
\end{align*}
\end{definition}

\begin{definition}
The Mandelbrot set $M$ is defined as
\begin{equation*}
M=\{c\in\mathbb{C}:0\in K_c\}=\{c\in\mathbb{C}: \text{$K_c$ is connected}\}\text{.}
\end{equation*}
\end{definition}

We compute the area of the Julia sets in the main bulb by just counting the pixels. Figure \ref{fig:area} shows the computed area for a slice at $Im(c)=0$ and for the whole region. It has been shown in \cite{Y} that the area is continuous in each bulb, but not at the junction points. We also show where each point $c$ denotes the area of the corresponding Julia set. The Mandelbrot set is clearly visible.

\vspace{0.5cm}

\begin{figure}[h]
    \centering
    \subfigure[Area for the slice at $Im(c)=0$]
    {
        \includegraphics[width=2.2in]{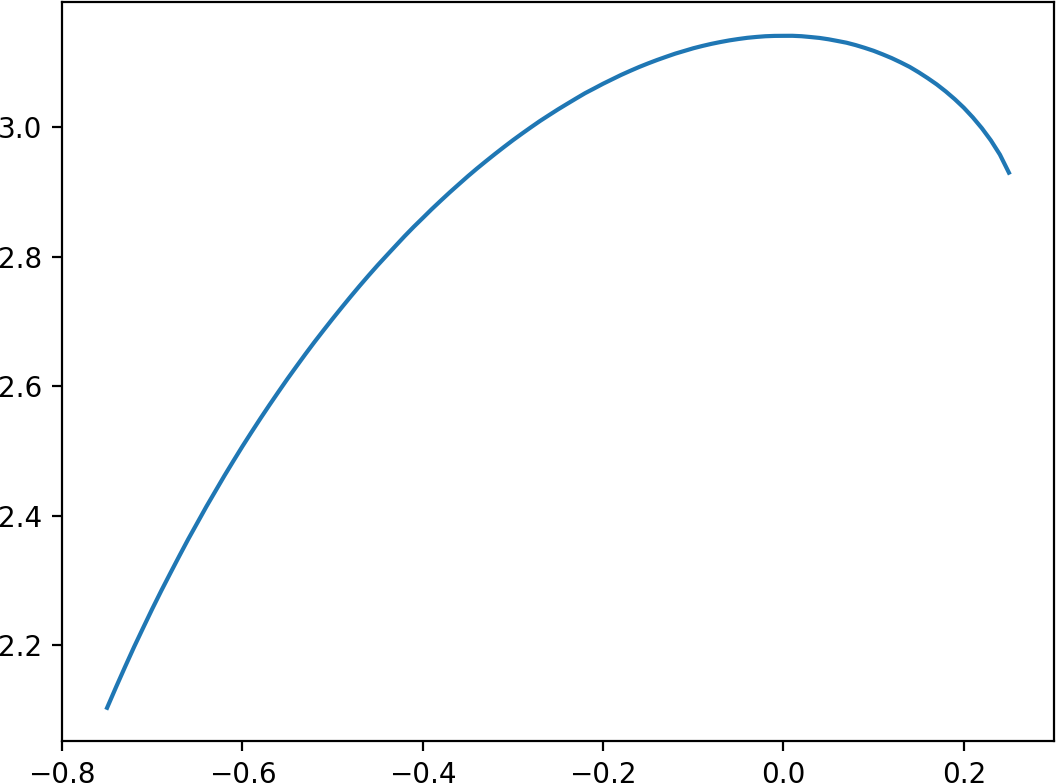}
    }
    \hfill
    \subfigure[Area for the whole region]
    {
        \includegraphics[width=2.2in]{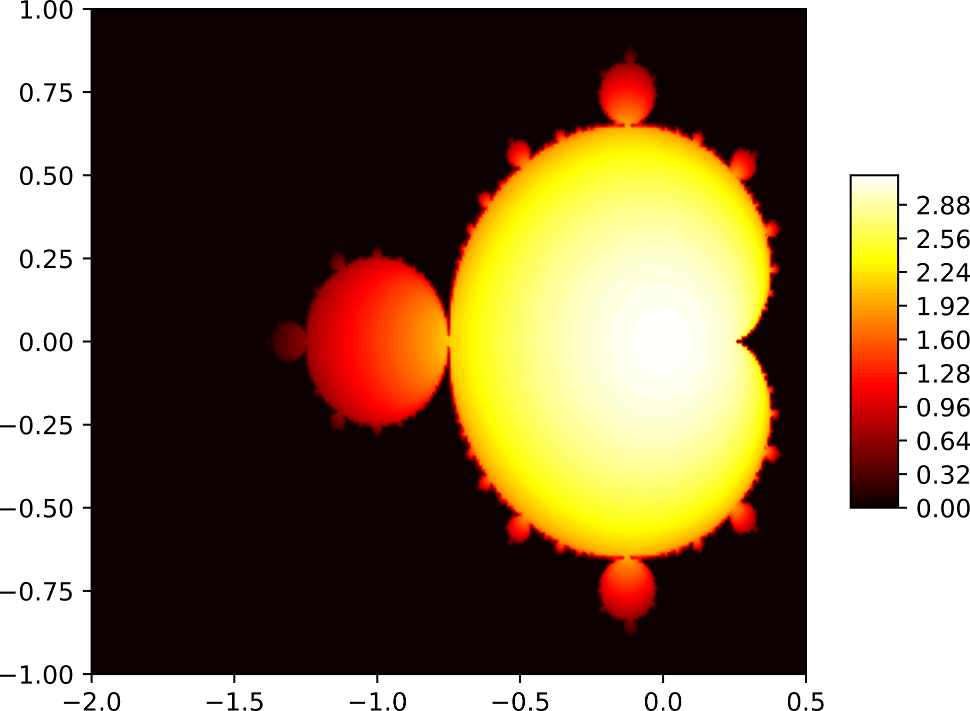}
    }
    \caption{Area of the quadratic Julia sets for $z^2+c$}
    \label{fig:area}
\end{figure}

\clearpage

We also explore the box-counting dimension of each Julia set using 8 different pictures sizes, starting from 640x480 up to 5120x3840, and counting the number of boxes necessary to cover the curve at fixed steps.

Figures \ref{fig:bcd1a} and \ref{fig:bcd1b} show the possible approximation of the box-counting dimension of the quadratic Julia sets for $z^2+c$ and varying $c$. It seems like the larger the used images are, the smoother the curve becomes. The largest images correspond to the thicker red curve.

\begin{figure}[h!]
	\includegraphics[width=5in]{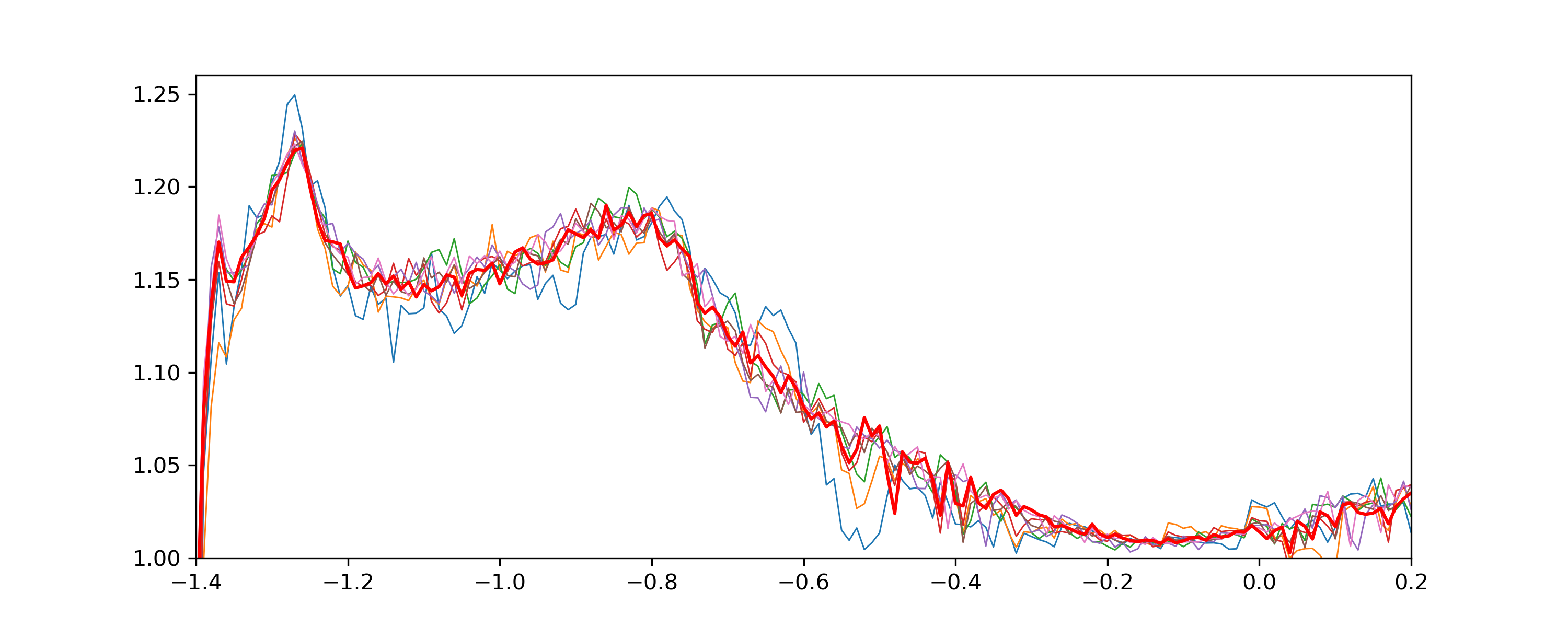}
	    \caption{Dimension for the slice at $Im(c)=0$}
	\label{fig:bcd1a}
\end{figure}

\begin{figure}[h!]
    \includegraphics[width=5in]{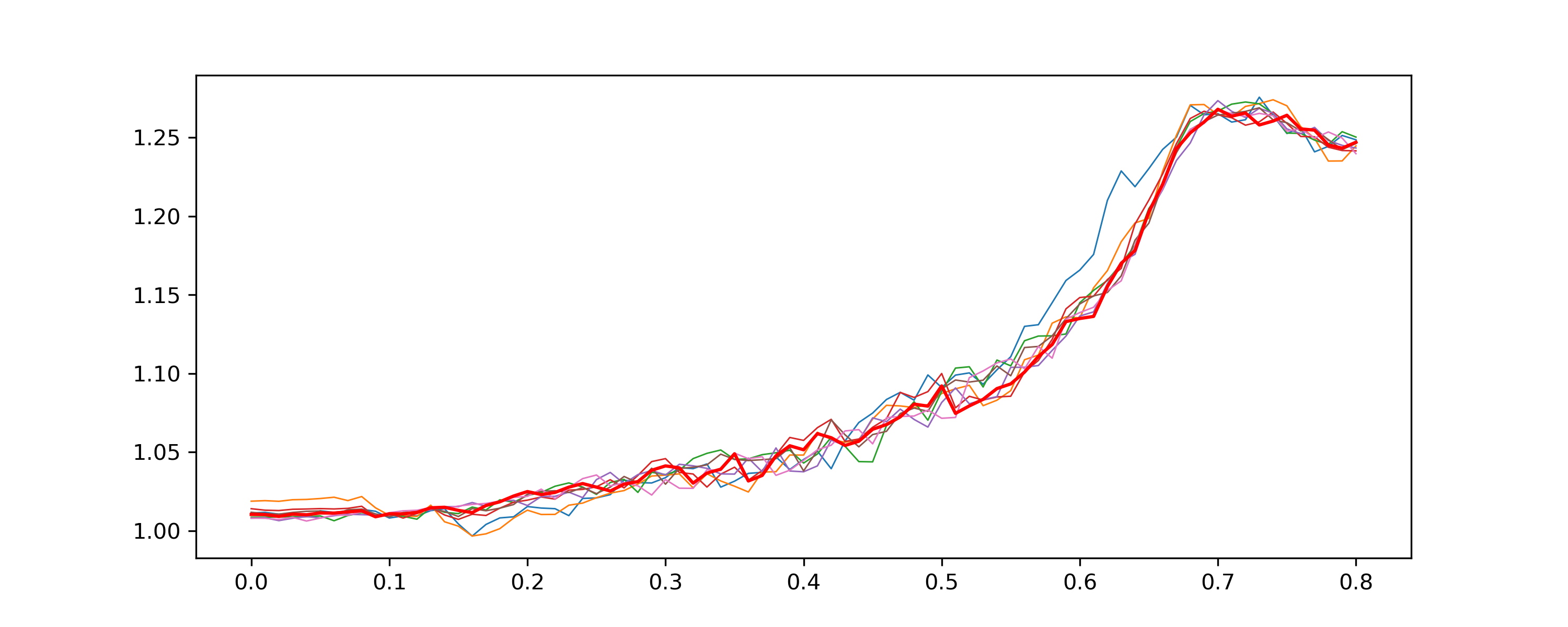}
        \caption{Dimension for the slice at $Re(c)=-0.12$}
    \label{fig:bcd1b}
\end{figure}

\clearpage

Figure \ref{fig:loglog2} shows two chosen log-log plots used to explore two examples of Julia sets with real $c$. The red dots denote the pairs (box size, number of boxed necessary), while the blue line is a linear fit in the log-log plot.

\begin{figure}[h]
    \centering
    \subfigure[For $z^2+0.2$, d = 1.035, error: 0.00088]
    {
        \includegraphics[width=2.3in]{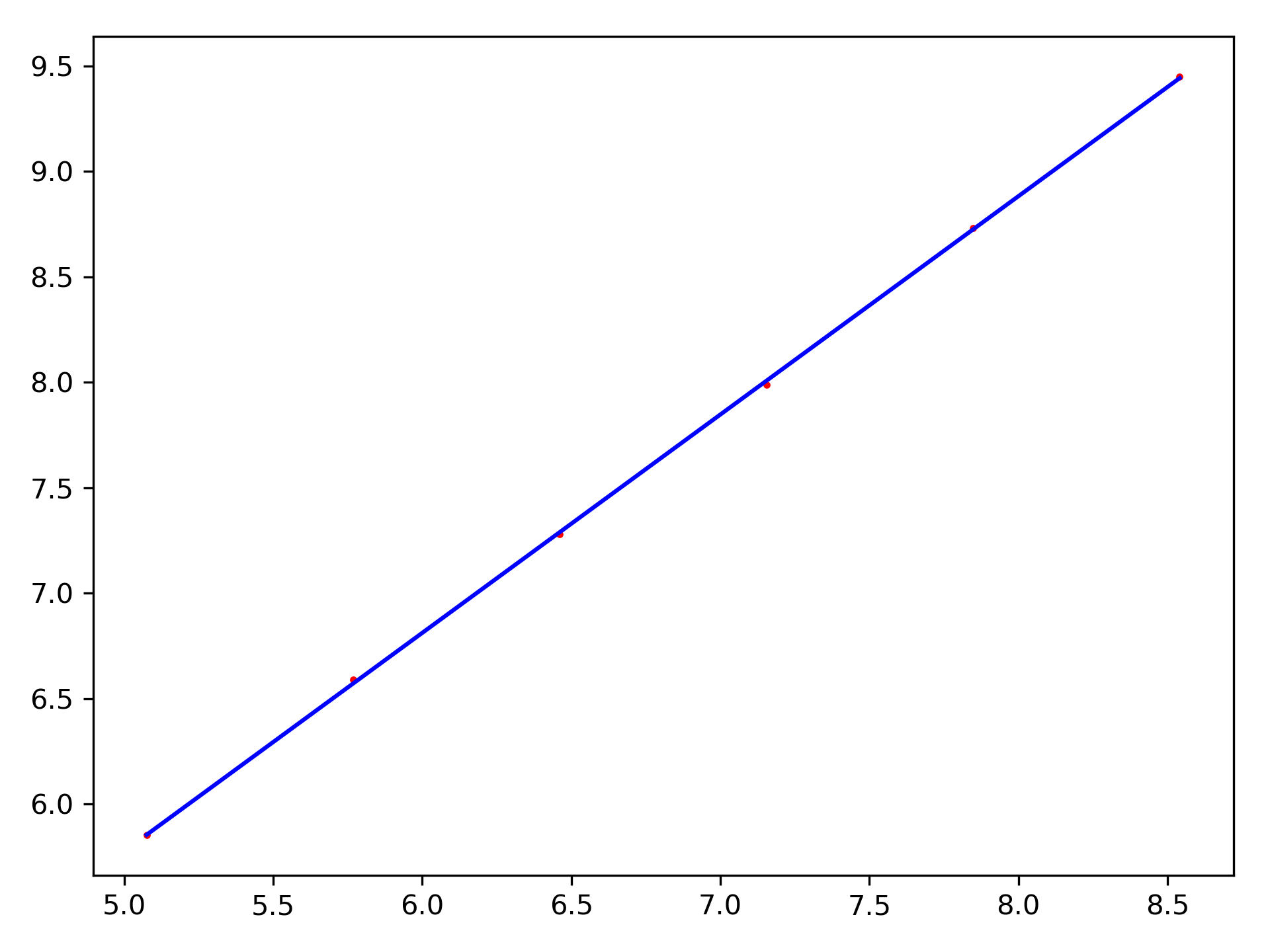}
    }
    \hfill
    \subfigure[For $z^2-0.5$, d = 1.071, error: 0.0029]
    {
        \includegraphics[width=2.3in]{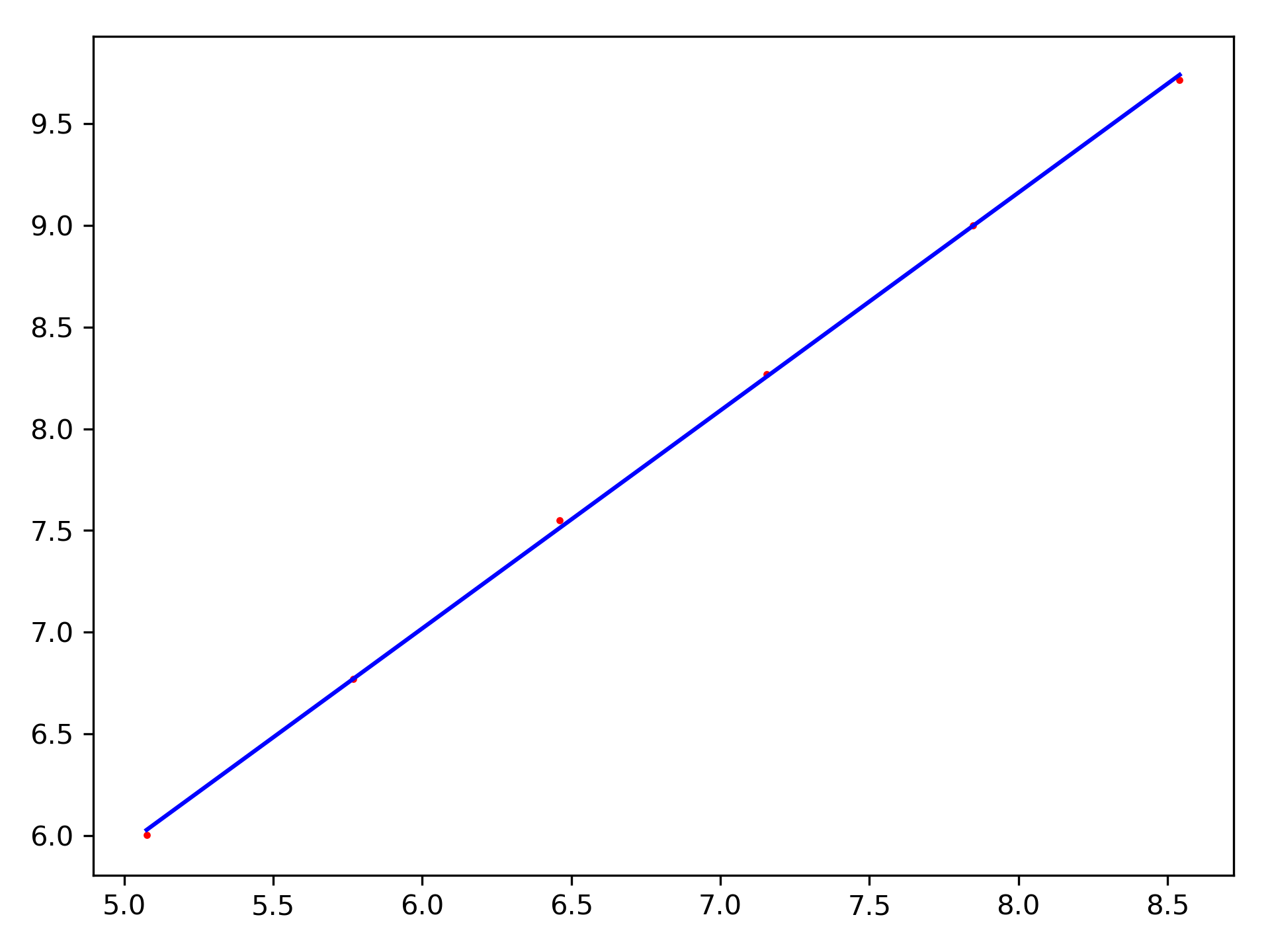}
    }
    \caption{Log-Log plot for the box-counting dimension}
    \label{fig:loglog2}
\end{figure}

In some cases it seems like the limit does not converge, i.e. when we can't find a good linear fit. Figure \ref{fig:errors} shows the error when finding the linear fit. It seems like the larger the used images are, the smoother the error curve becomes. The largest images correspond to the thicker red curve.

\begin{figure}[h]
    \includegraphics[width=5in]{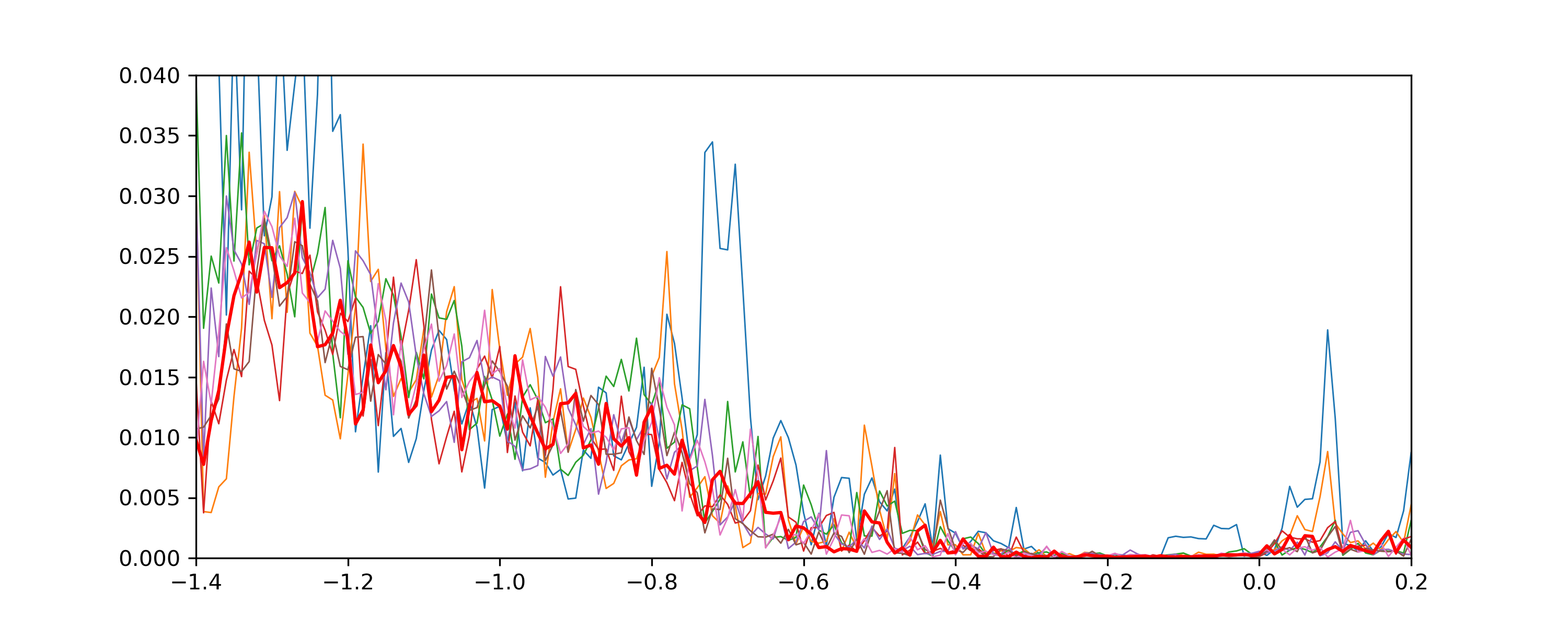}
    \caption{Error when computing the box-counting dimension}
    \label{fig:errors}
\end{figure}

\clearpage


\section{Filled Julia set spectrum.}

We compute the spectrum for the Basilica, the Rabbit and the junctions points from the main bulb to the Basilica bulb and the Rabbit bulb. One way of doing this is by computing the spectrum of these Julia sets for different iterations (and then extrapolating). Figure \ref{fig:iters} show meshes of the basilica Julia set after 10 and 20 iterations.

\begin{figure}[h]
    \centering
    \subfigure[Mesh of the Basilica after 10 iterations]
    {
        \includegraphics[width=2.2in]{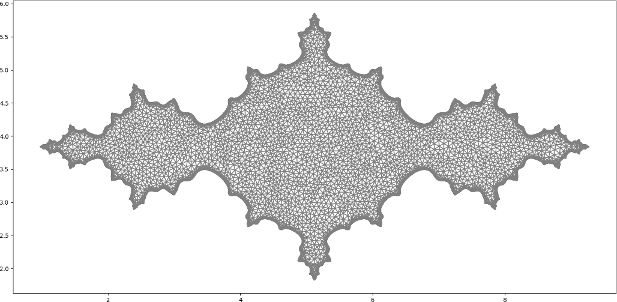}
    }
    \hfill
    \subfigure[Mesh of the Basilica after 20 iterations]
    {
        \includegraphics[width=2.2in]{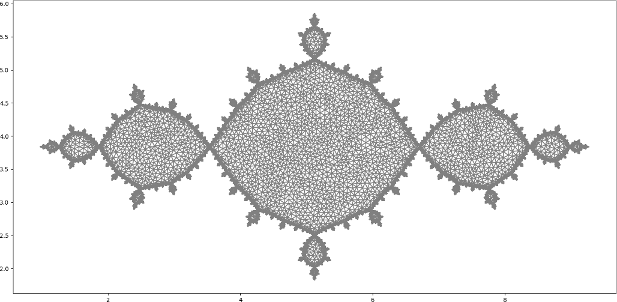}
    }
    \caption{Examples of meshes}
    \label{fig:iters}
\end{figure}

In the Dirichlet case, we will compare the spectrum found using this method with the sorted union of the eigenvalues of the quasicircles (after 170 iterations). For the basilica, we use quasicircle 1, 2 (twice), 3 (twice) and 4 (twice) and for the rabbit quasicircle 1, 2 (twice) and 3 (twice) to find the lower part of the spectrum, since smaller quasicirles don't contribute to the lower part of the spectrum (Figure \ref{fig:qcs}).

\begin{figure}[h]
    \centering
    \subfigure[On the Basilica]
    {
        \includegraphics[width=2.6in]{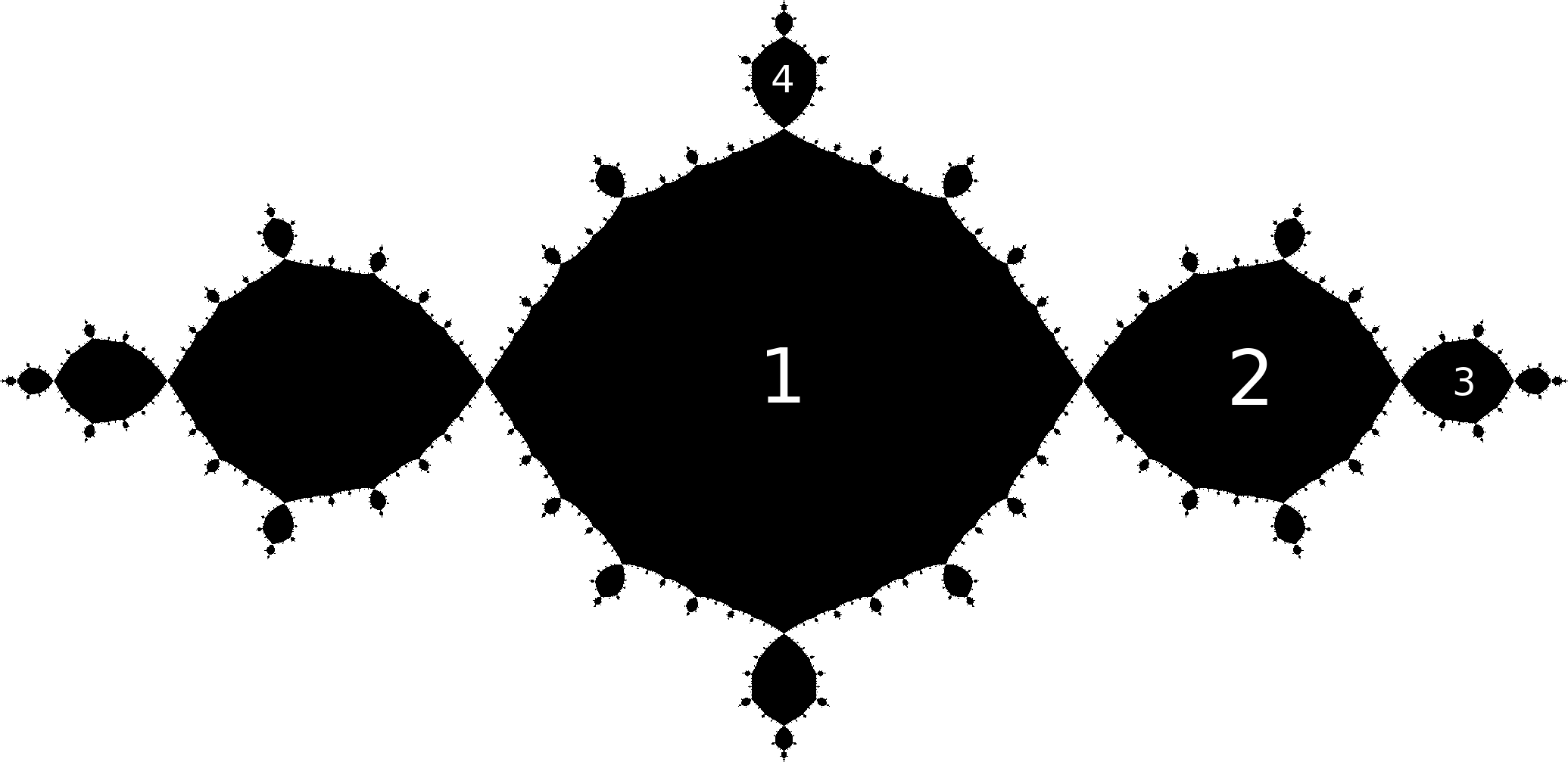}
    }
    \hfill
    \subfigure[On the Rabbit]
    {
        \includegraphics[width=1.7in]{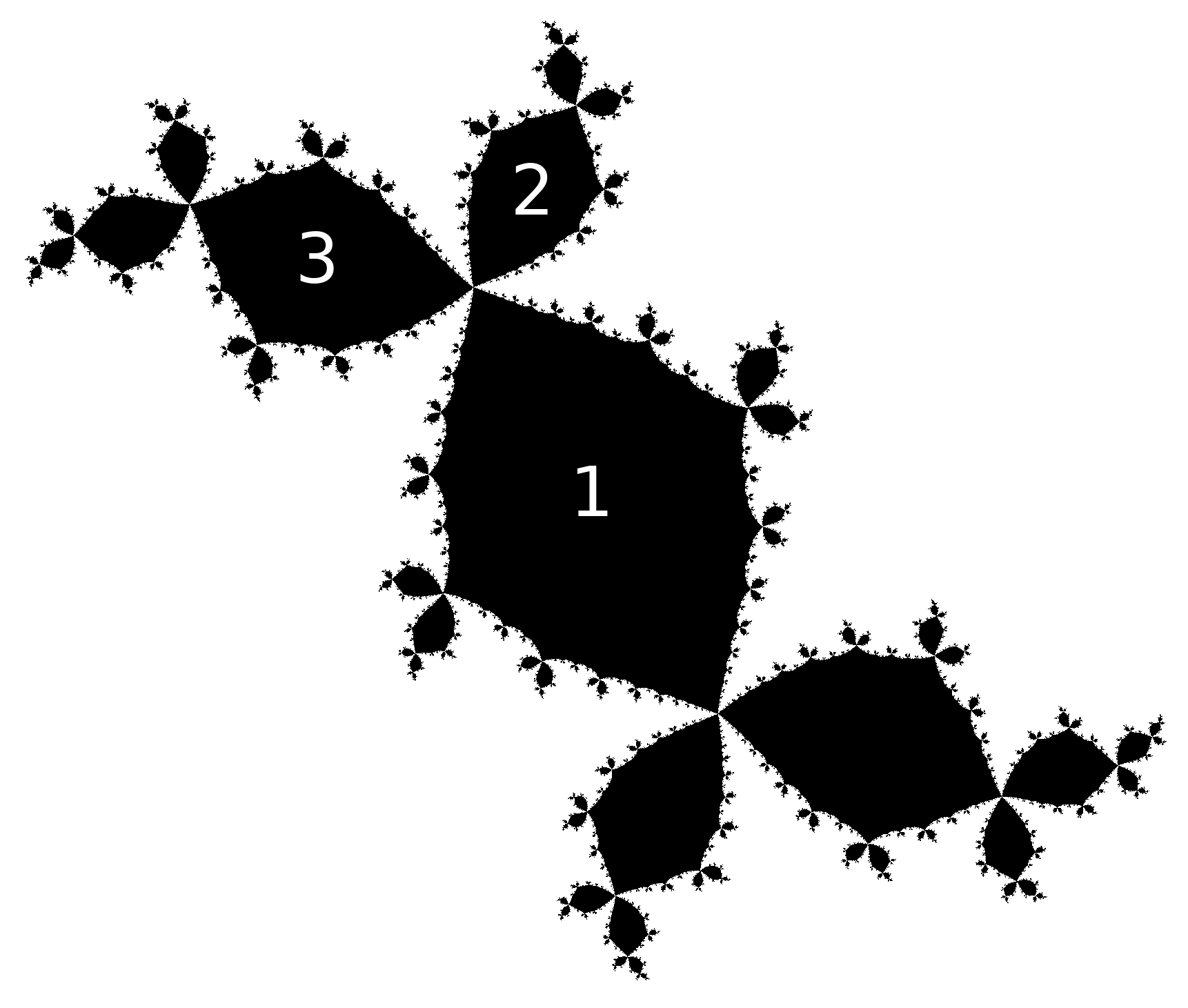}
    }
    \caption{Selected quasicircles}
    \label{fig:qcs}
\end{figure}

\clearpage

Figures \ref{fig:qcbas1d} -- \ref{fig:qcrab3n} shows three chosen eigenfunctions with Dirichlet or Neumann boundary conditions each. We will write "H" or "V" if a quasicircle or an eigenfunction is symmetric under horizontal or vertical reflection and "SH" or "SV" if an eigenfunction is skew-symmetric under horizontal or vertical reflection.

\begin{figure}[h!]
    \centering
    \subfigure[For $\lambda_6=304.88$ (H, V)]
    {
        \includegraphics[width=1.5in]{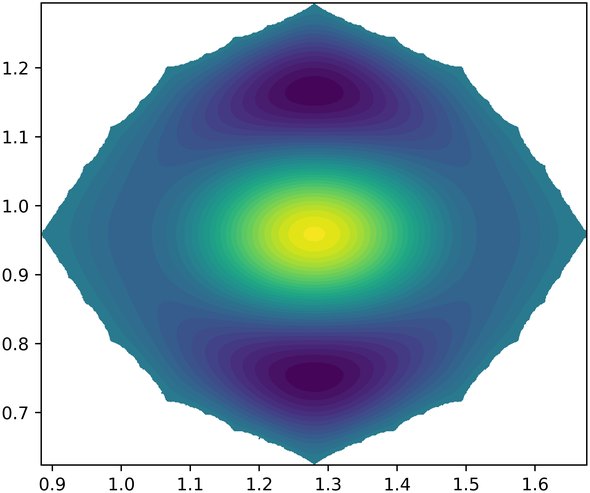}
    }
    \hfill
    \subfigure[For $\lambda_7=363.17$ (H, SV)]
    {
        \includegraphics[width=1.5in]{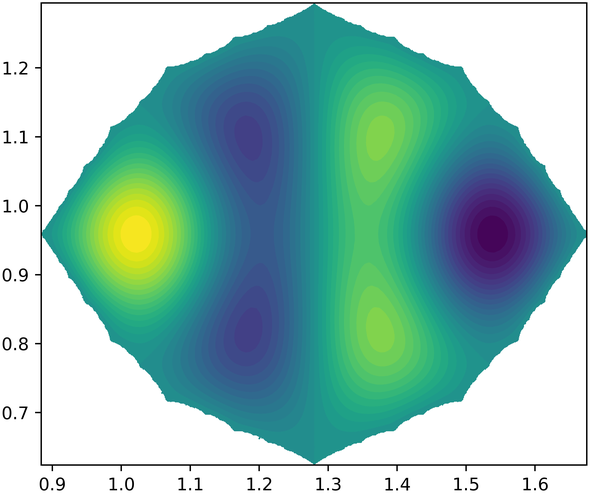}
    }
    \hfill
    \subfigure[For $\lambda_8=392.97$ (SH, V)]
    {
    \includegraphics[width=1.5in]{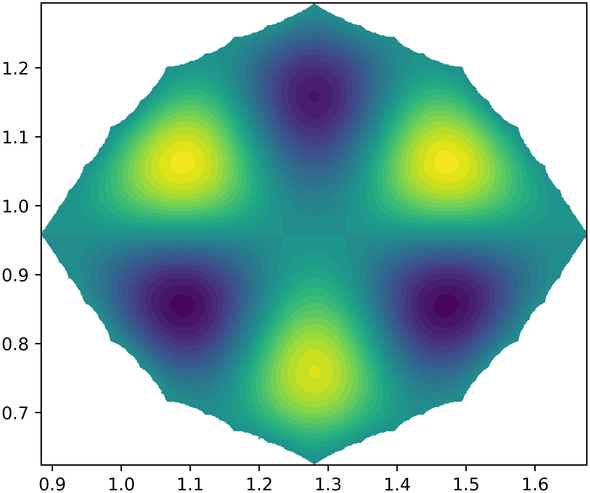}
    }
    \caption{Eigenfunctions on the 1st quasicircle of the Basilica (H, V) with Dirichlet BC}
    \label{fig:qcbas1d}
\end{figure}

\begin{figure}[h!]
    \centering
    \subfigure[For $\lambda_6=104.25$ (H, V)]
    {
        \includegraphics[width=1.5in]{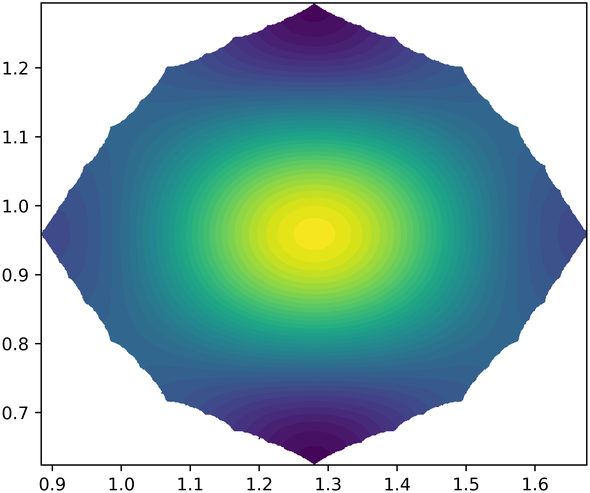}
    }
    \hfill
    \subfigure[For $\lambda_7=149.02$ (H, SV)]
    {
        \includegraphics[width=1.5in]{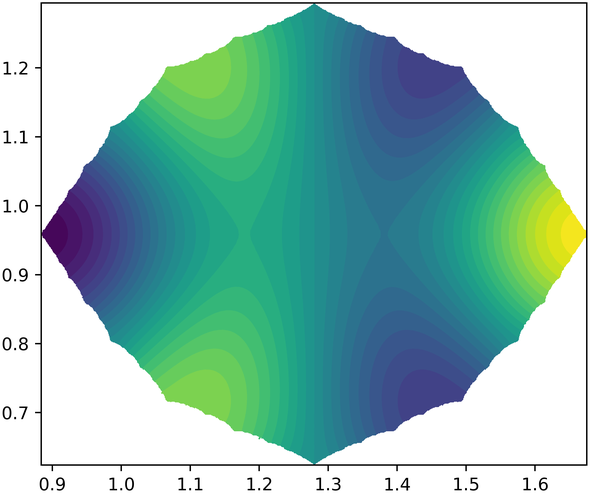}
    }
    \hfill
    \subfigure[For $\lambda_8=168.01$ (SH, V)]
    {
        \includegraphics[width=1.5in]{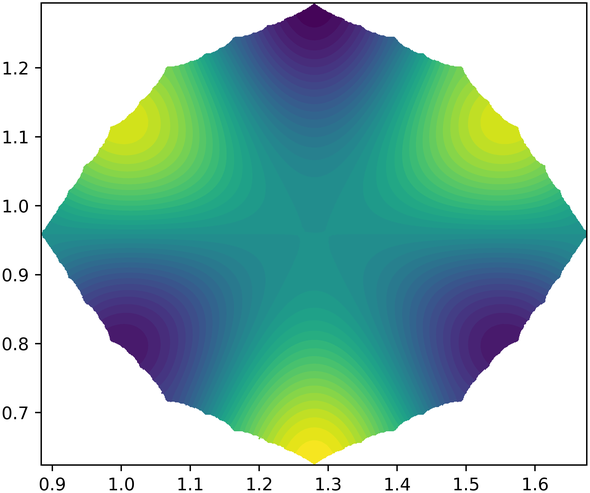}
    }
    \caption{Eigenfunctions on the 1st quasicircle of the Basilica (H, V) with Neumann BC}
    \label{fig:qcbas1n}
\end{figure}

\begin{figure}[h!]
    \centering
    \subfigure[For $\lambda_6=1189.22$ (H)]
    {
        \includegraphics[width=1.5in]{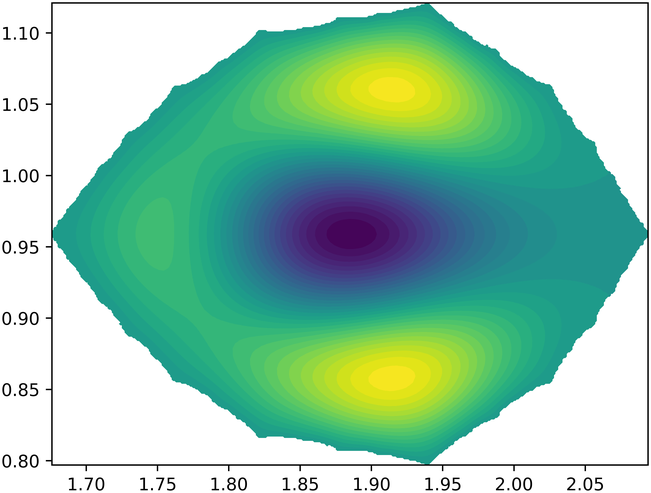}
    }
    \hfill
    \subfigure[For $\lambda_7=1331.34$ (H)]
    {
        \includegraphics[width=1.5in]{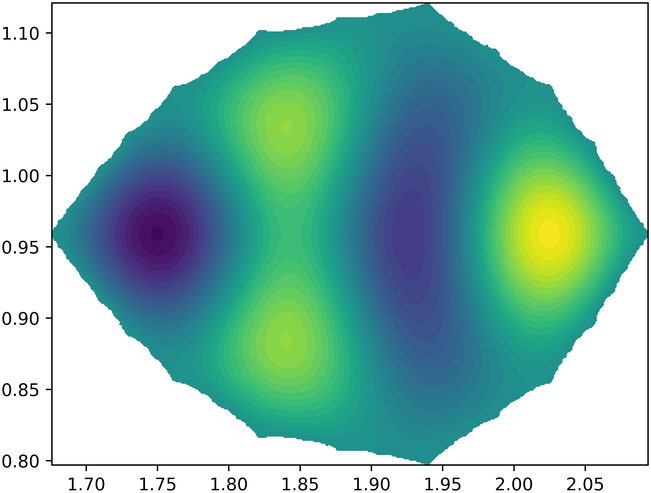}
    }
    \hfill
    \subfigure[For $\lambda_8=1495.42$ (SH)]
    {
        \includegraphics[width=1.5in]{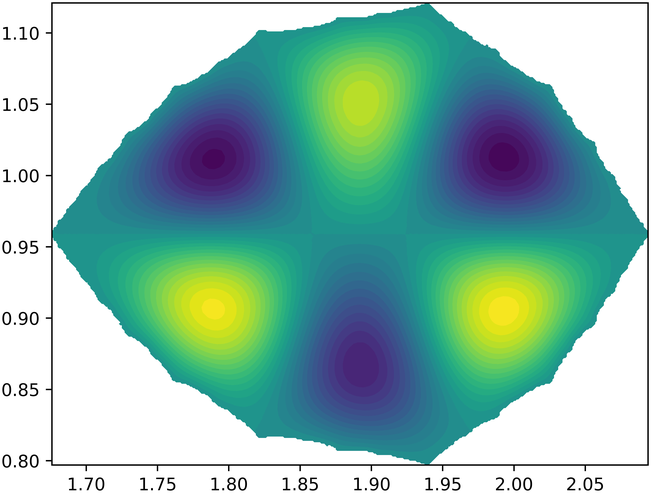}
    }
    \caption{Eigenfunctions on the 2nd quasicircle of the Basilica (H) with Dirichlet BC}
    \label{fig:qcbas2d}
\end{figure}

\clearpage

\begin{figure}[h!]
    \centering
    \subfigure[For $\lambda_6=511.08$ (H)]
    {
        \includegraphics[width=1.5in]{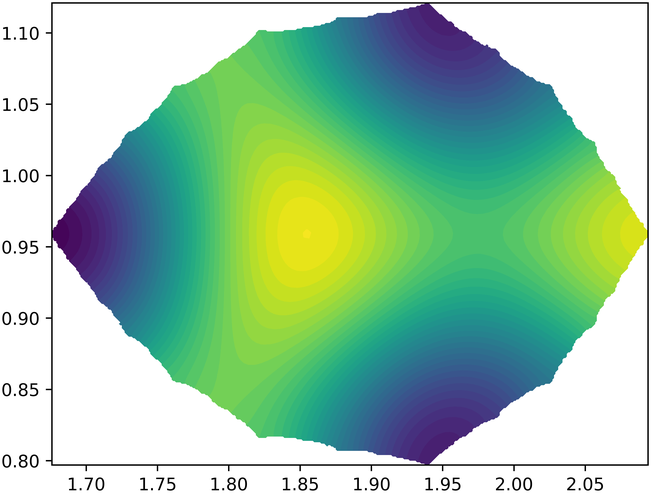}
    }
    \hfill
    \subfigure[For $\lambda_7=571.49$ (H)]
    {
        \includegraphics[width=1.5in]{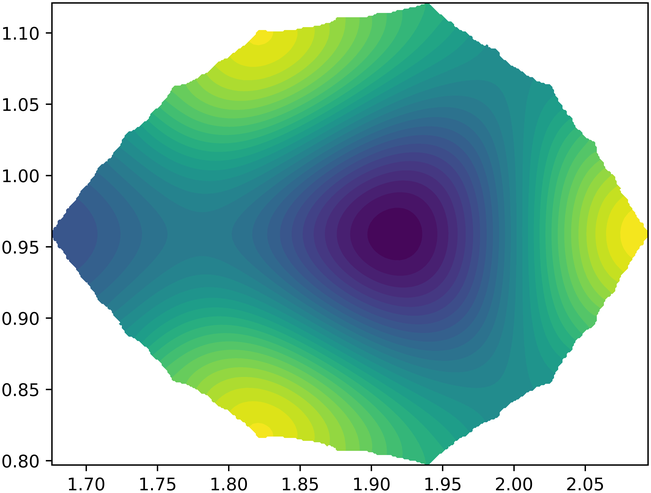}
    }
    \hfill
    \subfigure[For $\lambda_8=655.85$ (SH)]
    {
        \includegraphics[width=1.5in]{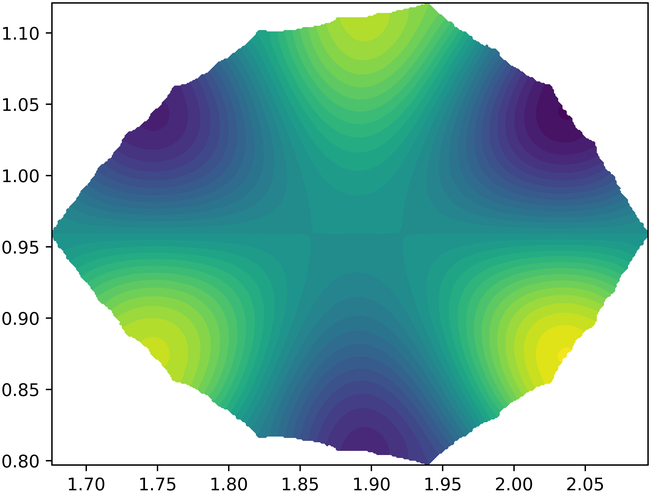}
    }
    \caption{Eigenfunctions on the 2nd quasicircle of the Basilica (H) with Neumann BC}
    \label{fig:qcbas2n}
\end{figure}

\begin{figure}[h!]
    \centering
    \subfigure[For $\lambda_6=559.22$ (H, V)]
    {
        \includegraphics[width=1.3in]{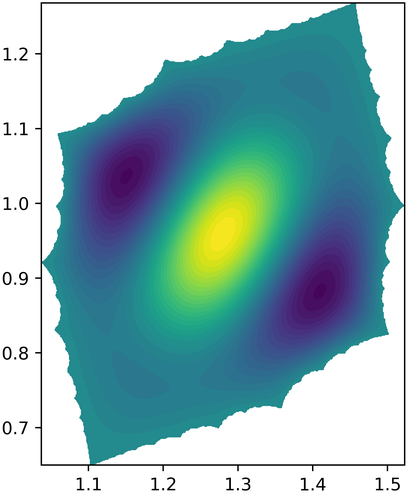}
    }
    \hfill
    \subfigure[For $\lambda_7=564.00$ (H, SV)]
    {
        \includegraphics[width=1.3in]{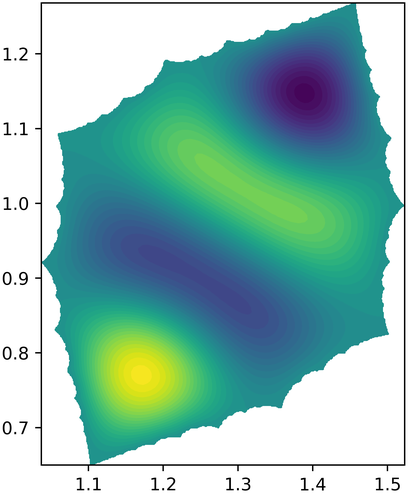}
    }
    \hfill
    \subfigure[For $\lambda_8=679.60$ (SH, V)]
    {
        \includegraphics[width=1.3in]{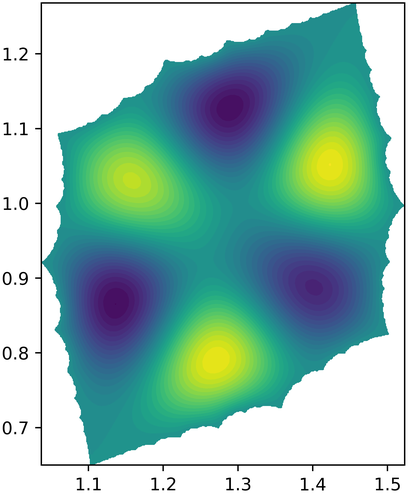}
    }
    \caption{Eigenfunctions on the 1st quasicircle of the Rabbit (H, V) with Dirichlet BC}
    \label{fig:qcrab1d}
\end{figure}

\begin{figure}[h!]
    \centering
    \subfigure[For $\lambda_6=205.07$ (H, SV)]
    {
        \includegraphics[width=1.3in]{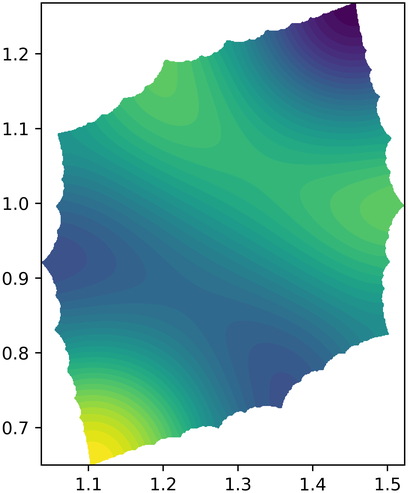}
    }
    \hfill
    \subfigure[For $\lambda_7=238.94$ (H, V)]
    {
        \includegraphics[width=1.3in]{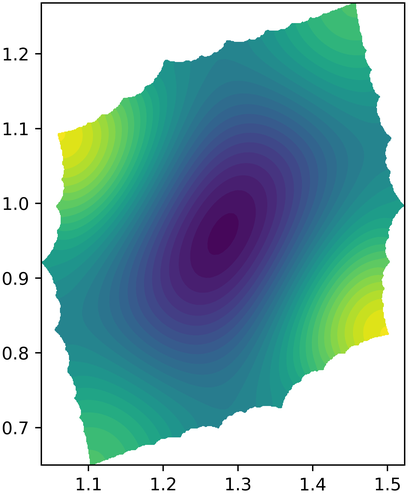}
    }
    \hfill
    \subfigure[For $\lambda_8=280.59$ (SH, V)]
    {
        \includegraphics[width=1.3in]{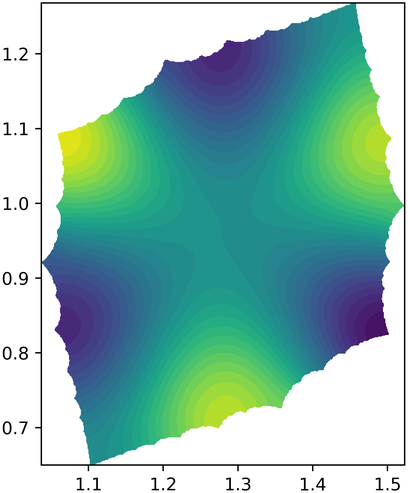}
    }
    \caption{Eigenfunctions on the 1st quasicircle of the Rabbit (H, V) with Neumann BC}
    \label{fig:qcrab1n}
\end{figure}

\clearpage

\begin{figure}[h!]
    \centering
    \subfigure[For $\lambda_6=1542.19$]
    {
        \includegraphics[width=1.5in]{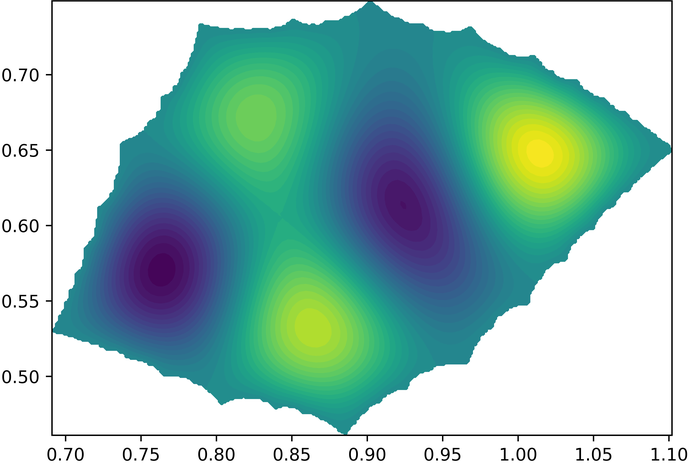}
    }
    \hfill
    \subfigure[For $\lambda_7=1674.14$]
    {
        \includegraphics[width=1.5in]{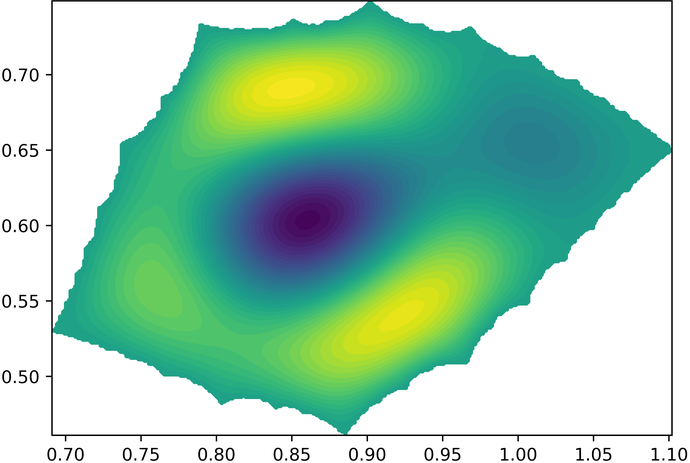}
    }
    \hfill
    \subfigure[For $\lambda_8=1939.44$]
    {
        \includegraphics[width=1.5in]{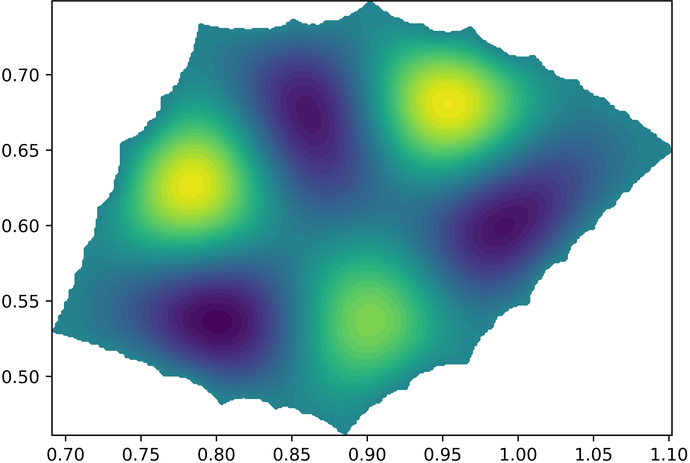}
    }
    \caption{Eigenfunctions on the 3rd quasicircle of the Rabbit with Dirichlet BC}
    \label{fig:qcrab3d}
\end{figure}

\vspace{-0.5cm}

\begin{figure}[h!]
    \centering
    \subfigure[For $\lambda_6=555.02$]
    {
        \includegraphics[width=1.5in]{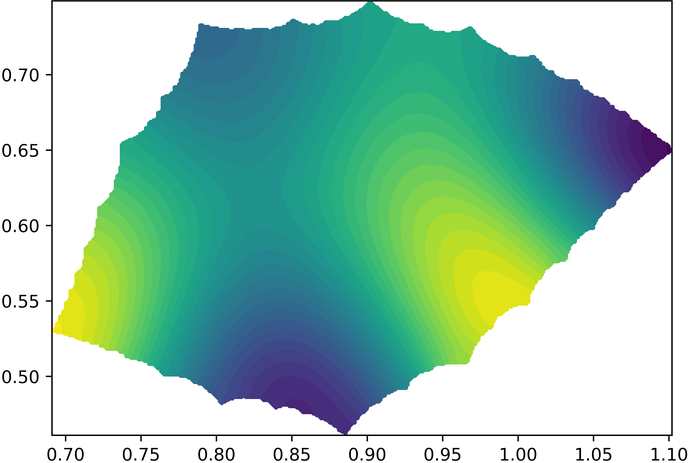}
    }
    \hfill
    \subfigure[For $\lambda_7=732.63$]
    {
        \includegraphics[width=1.5in]{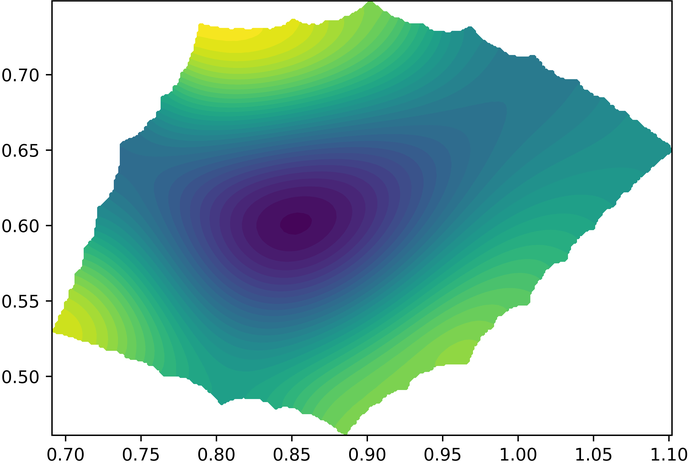}
    }
    \hfill
    \subfigure[For $\lambda_8=825.79$]
    {
        \includegraphics[width=1.5in]{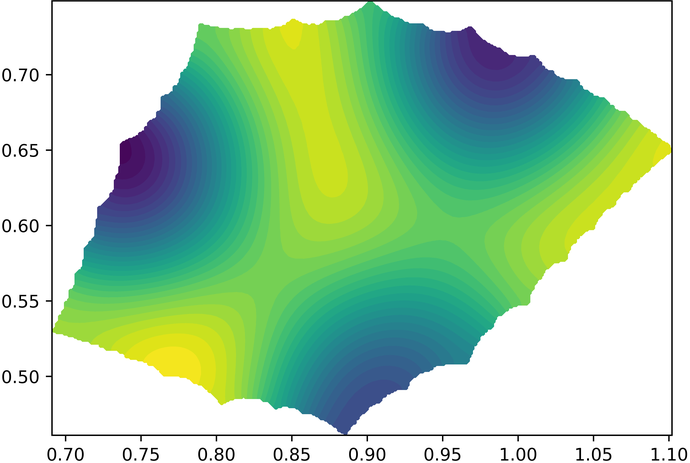}
    }
    \caption{Eigenfunctions on the 3rd quasicircle of the Rabbit with Neumann BC}
    \label{fig:qcrab3n}
\end{figure}

Tables 1 and 2 show the beginning of the Dirichlet spectrum computed using the quasicircles and after 20 iterations for the Basilica and the Rabbit. The quasicircle that the eigenvalue belongs to (from Fig. \ref{fig:qcs}) is shown in brackets.

\begin{table}[h!]
\centering
\begin{minipage}{0.48\textwidth}
\centering
\begin{tabular}{|c|c|}
\hline
Union of QCs & After 10 it. \\ \hline
56.01 (1)             & 55.93               \\
131.96 (1)            & 132.68              \\
151.74 (1)            & 152.99              \\
213.40 (2)            & 214.33              \\
213.40 (2)            & 214.35              \\
265.00 (1)            & 235.90              \\
304.88 (1)            & 269.83              \\
363.17 (1)            & 312.11              \\
392.97 (1)            & 371.67              \\
467.87 (1)            & 403.65              \\
487.42 (2)            & 487.55              \\
487.42 (2)            & 499.71              \\
507.85 (1)            & 501.36              \\
514.55 (1)            & 529.97              \\
556.89 (1)            & 531.92              \\
592.41 (2)            & 581.48              \\
92.41 (2)            & 617.26              \\ \hline
\end{tabular}
\vspace{0.2cm}
\caption{Basilica}
\end{minipage}
\begin{minipage}{0.48\textwidth}
\centering
\begin{tabular}{|c|c|}
\hline
Union of QCs & After 10 it. \\ \hline
98.59 (1)             & 94.61               \\
218.01 (1)            & 208.98              \\
279.41 (1)            & 270.26              \\
283.77 (3)            & 270.54              \\
283.77 (3)            & 273.48              \\
370.02 (1)            & 352.79              \\
472.78 (1)            & 469.07              \\
559.22 (1)            & 528.39              \\
564.00 (1)            & 551.48              \\
600.25 (2)            & 556.50              \\
600.25 (2)            & 558.21              \\
607.97 (3)            & 574.74              \\
607.97 (3)            & 575.97              \\
679.60 (1)            & 678.63              \\
790.59 (1)            & 728.05              \\
820.88 (3)            & 820.29              \\
820.88 (3)            & 825.30              \\ \hline
\end{tabular}
\vspace{0.2cm}
\caption{Rabbit}
\end{minipage}
\end{table}

\clearpage

Figures \ref{fig:basqc_d} and \ref{fig:rabqc_d} show the beginning of the spectrum using different iterations and using the spectra of the few largest quasicircles for the basilica and the rabbit with Dirichlet boundary conditions. The approximation works really well, especially at the beginning of the spectrum. It shows that after only 20 iterations the Dirichlet spectrum is already approximated well. The spectrum computed using quasicircles is sloping downwards, because we are missing eigenvalues from other smaller quasicircles that are not taken into account.


\vspace{0.2cm}

\begin{figure}[h]
	\includegraphics[width=3.5in]{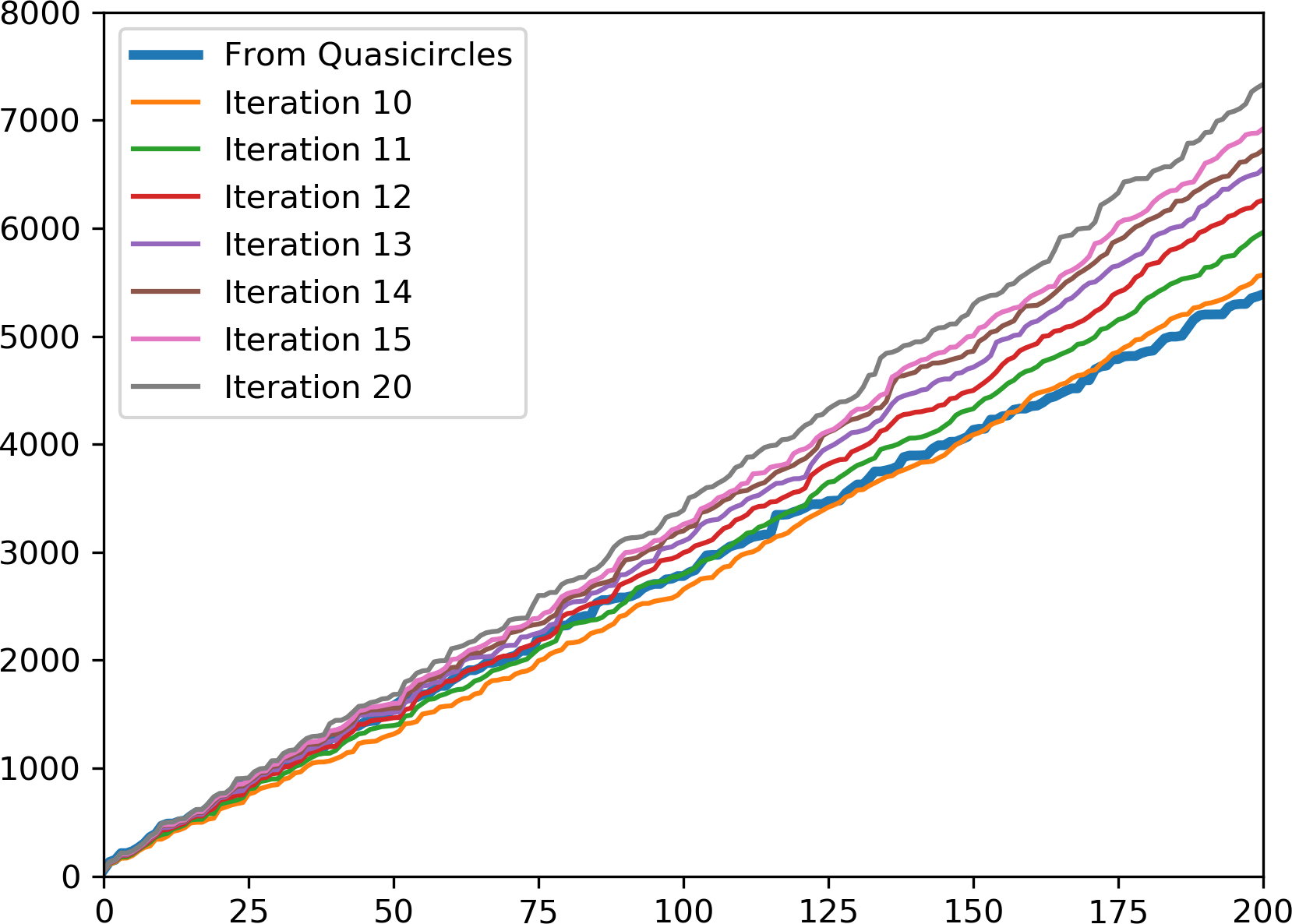}
	    \caption{Spectrum of the Basilica Julia set with Dirichlet BC}
	\label{fig:basqc_d}
\end{figure}

\begin{figure}[h]
	\includegraphics[width=3.5in]{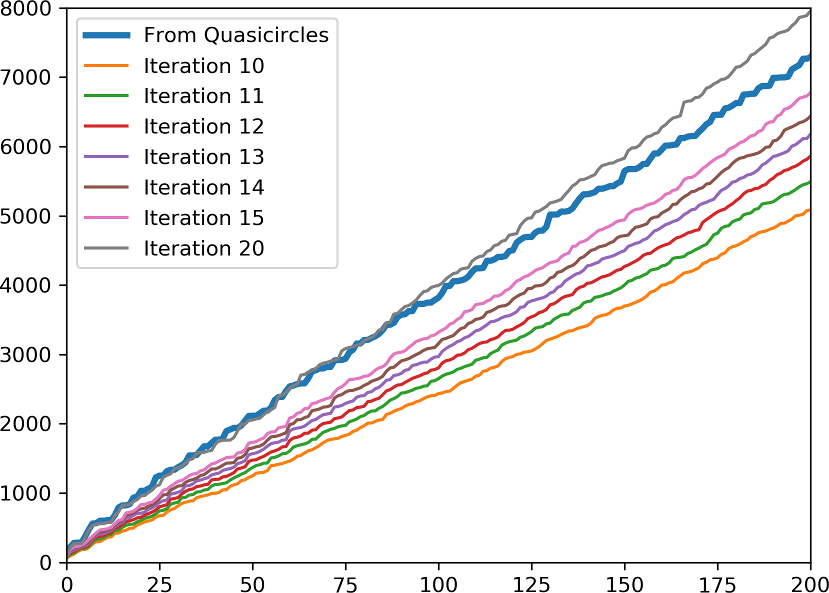}
	    \caption{Spectrum of the Rabbit Julia set with Dirichlet BC}
	\label{fig:rabqc_d}
\end{figure}

\clearpage

In the Neumann case, we can't use the quasicircle spectra, but we can still show the spectrum after different iterations (Fig. \ref{fig:basqc_n}, \ref{fig:rabqc_n}).


\begin{figure}[h]
	\includegraphics[width=3.5in]{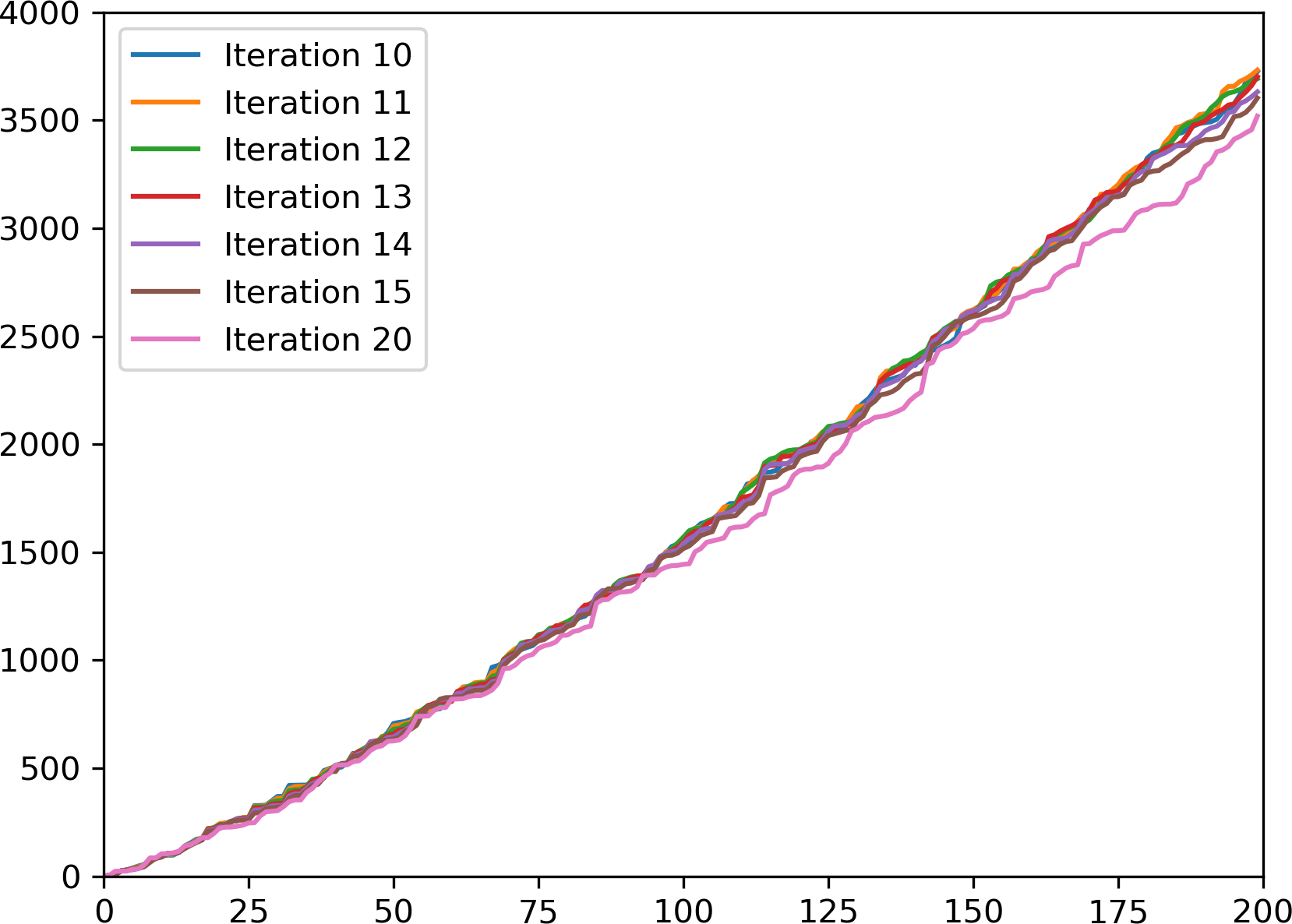}
	    \caption{Spectrum of the Basilica Julia set with Neumann BC}
	\label{fig:basqc_n}
\end{figure}

\begin{figure}[h]
	\includegraphics[width=3.5in]{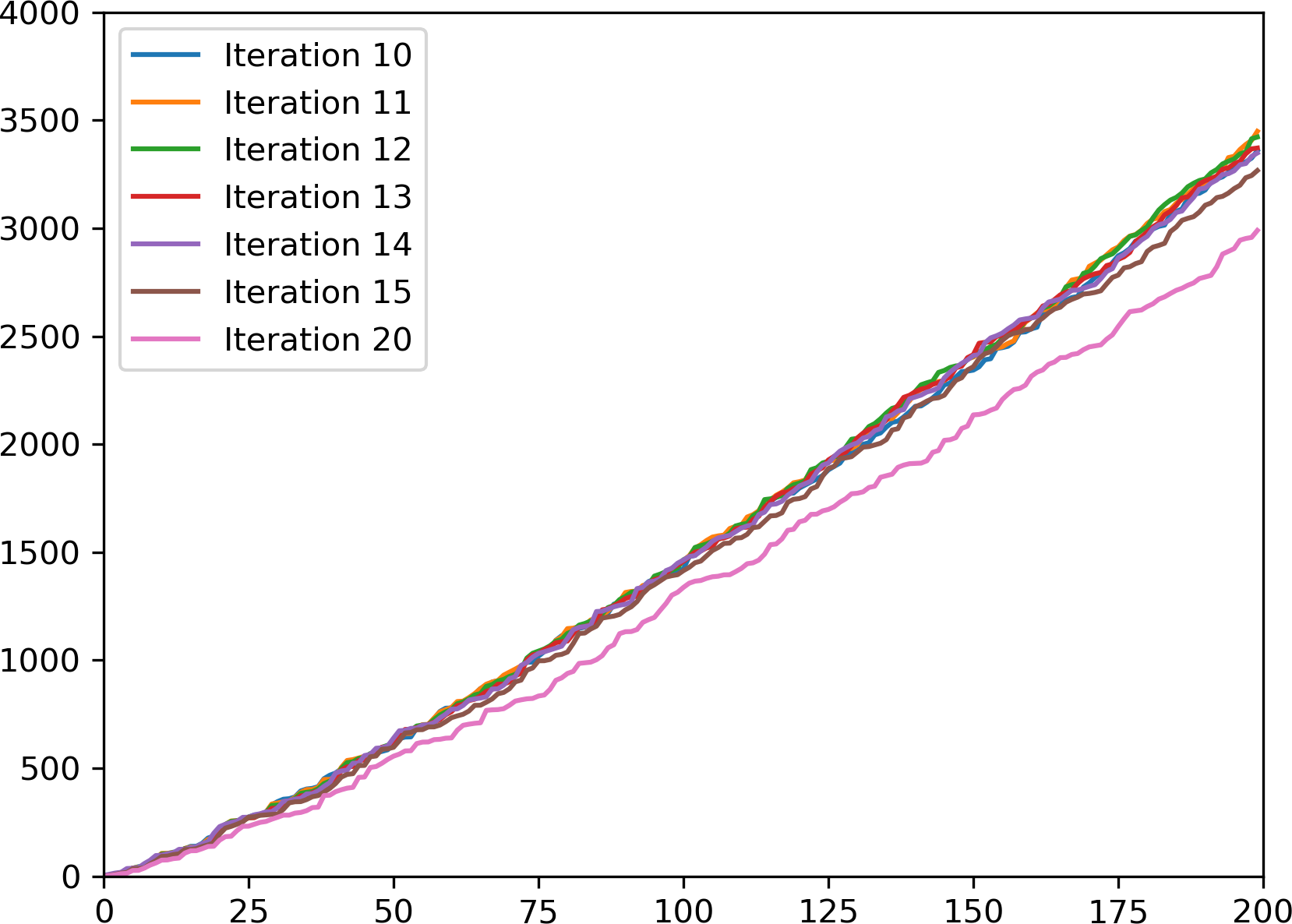}
	    \caption{Spectrum of the Rabbit Julia set with Neumann BC}
	\label{fig:rabqc_n}
\end{figure}

\clearpage

We are also interested in the behaviour of the spectrum of a family of Julia sets to $z^2+c$, $Im(c)=0$. Figure \ref{fig:evslice1} shows the first 100 eigenvalues, and it seems like there won't overlap. Note that in the chosen interval for $c$, the area of the corresponding Julia set increases, as seen in Figure \ref{fig:area} (a), and so does the box-counting dimension, if our conjecture is correct.

\begin{figure}[h]
	\includegraphics[width=\linewidth]{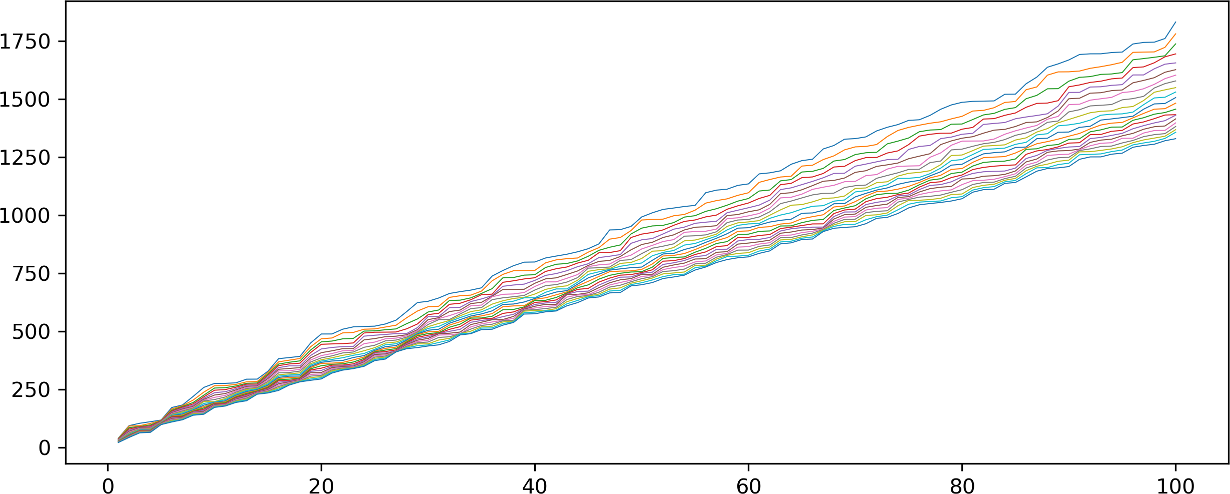}
	    \caption{Spectrum of 21 Julia sets to $z^2+c$,\hspace{\textwidth}$c\in \{-0.74, -0.73, ..., -0.54\}$, Dirichlet BC}
	\label{fig:evslice1}
\end{figure}

However, if we widen the iterval of the $c$-values, the spectra will intersect, as seen in Figure \ref{fig:evslice2}.

\begin{figure}[h]
	\includegraphics[width=\linewidth]{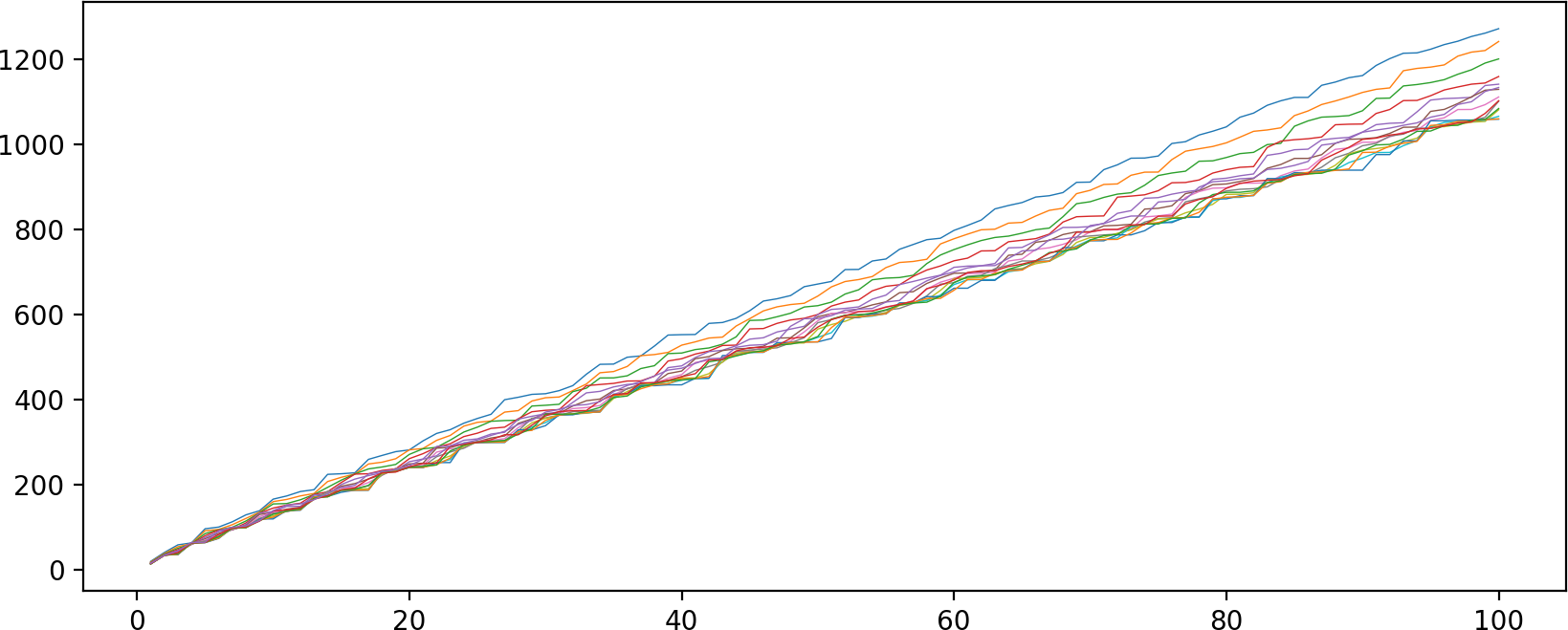}
	    \caption{Spectrum of 15 Julia sets to $z^2+c$,\hspace{\textwidth}$c\in \{-0.5, -0.45, ..., 0.2\}$, Dirichlet BC}
	\label{fig:evslice2}
\end{figure}


\clearpage

In Figures \ref{fig:cfjd1} and \ref{fig:cfjn1} we show similar graphs of the counting function $N(t)$ and the differences $D_1(t)$, $D_2(t)$ for a selected Julia set with $A=3.0305$, $d=1.0351$.

\begin{figure}[h!]
    \centering
    \subfigure[$N(t)$ (blue) and $\frac{A}{4\pi}t$ (orange)]
    {
        \includegraphics[width=1.75in]{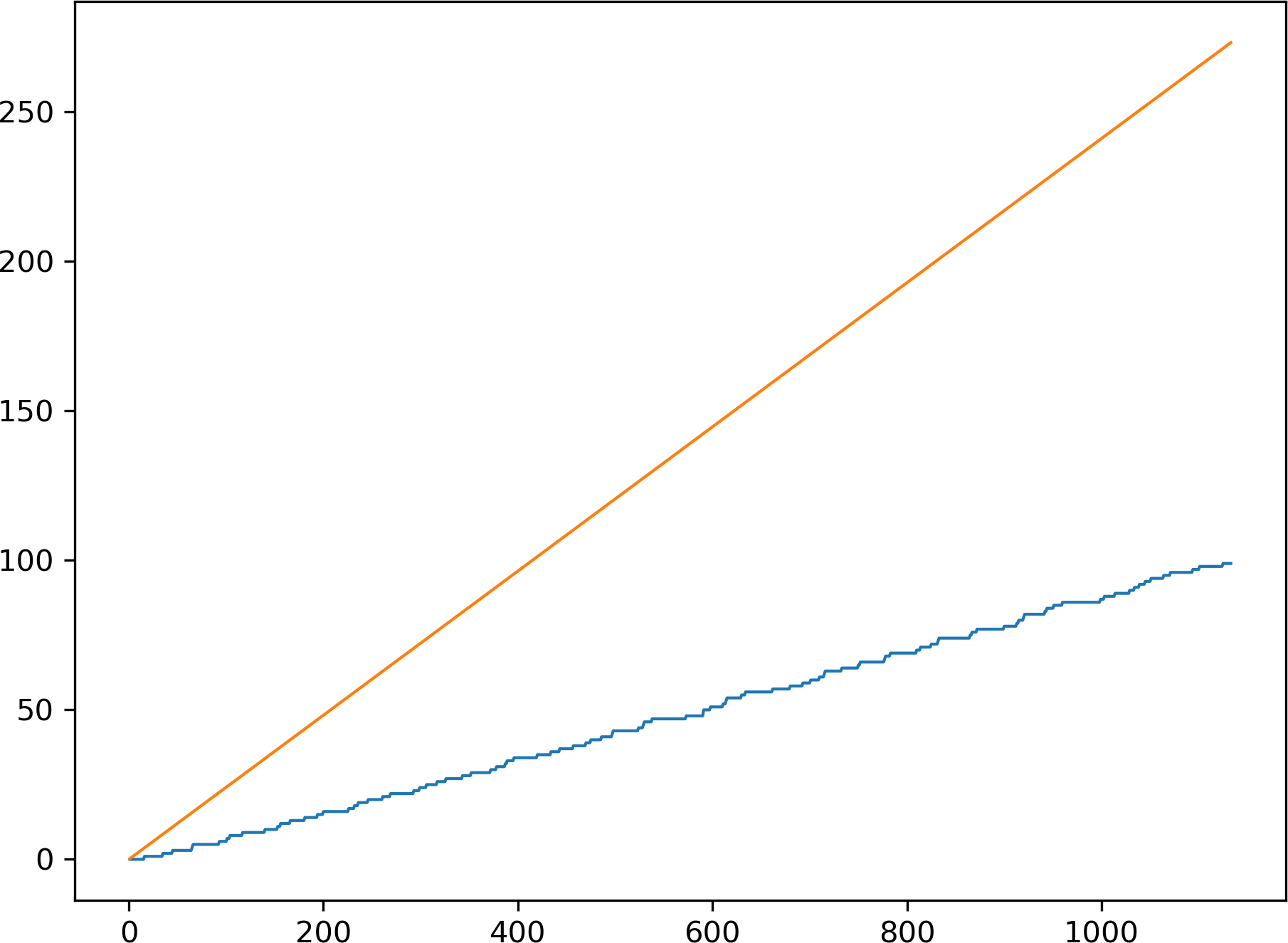}
    }
    \hfill
    \subfigure[$D_1(t)$]
    {
        \includegraphics[width=1.75in]{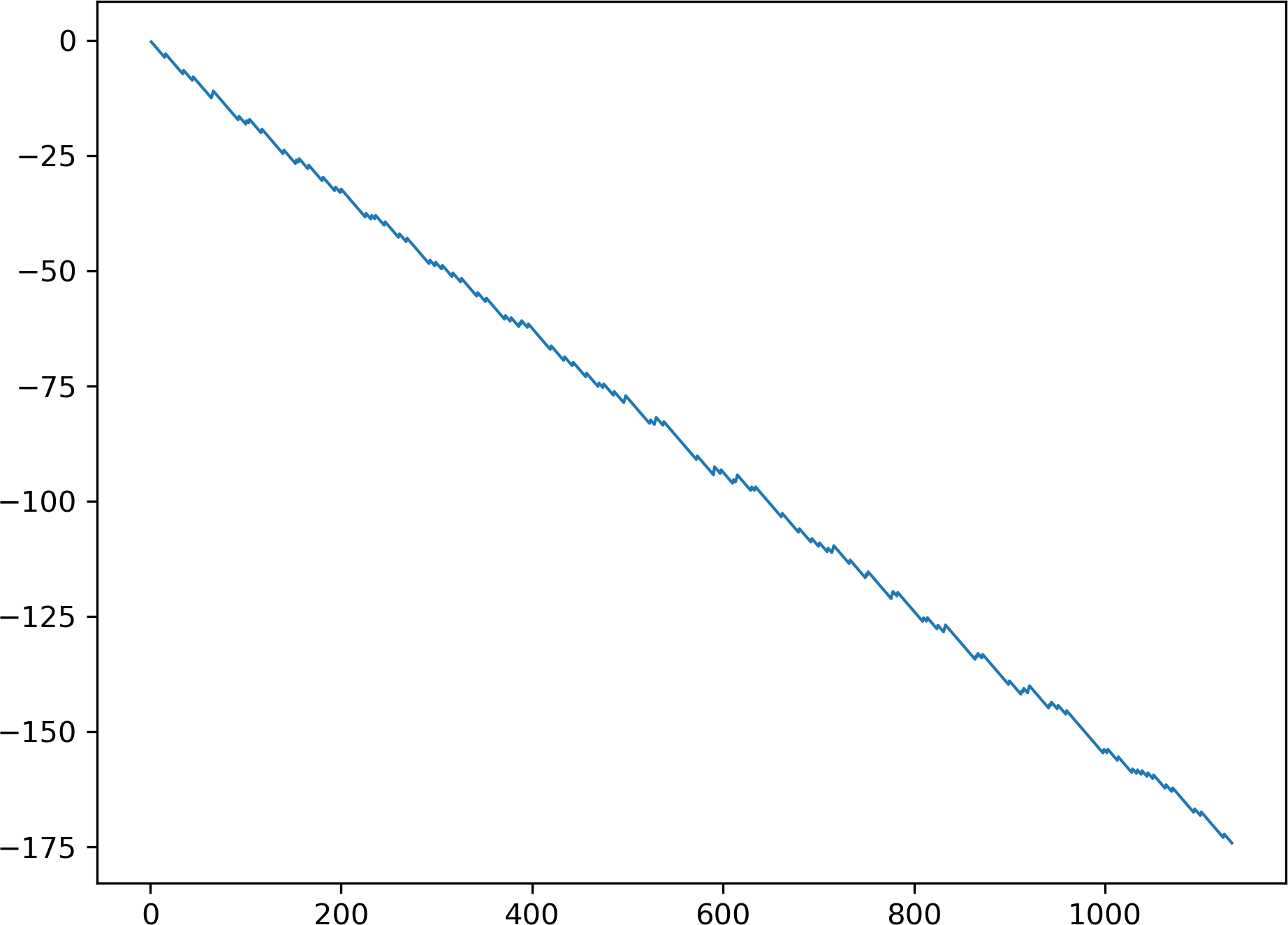}
    }
    \subfigure[$D_2(t)$]
    {
        \includegraphics[width=1.75in]{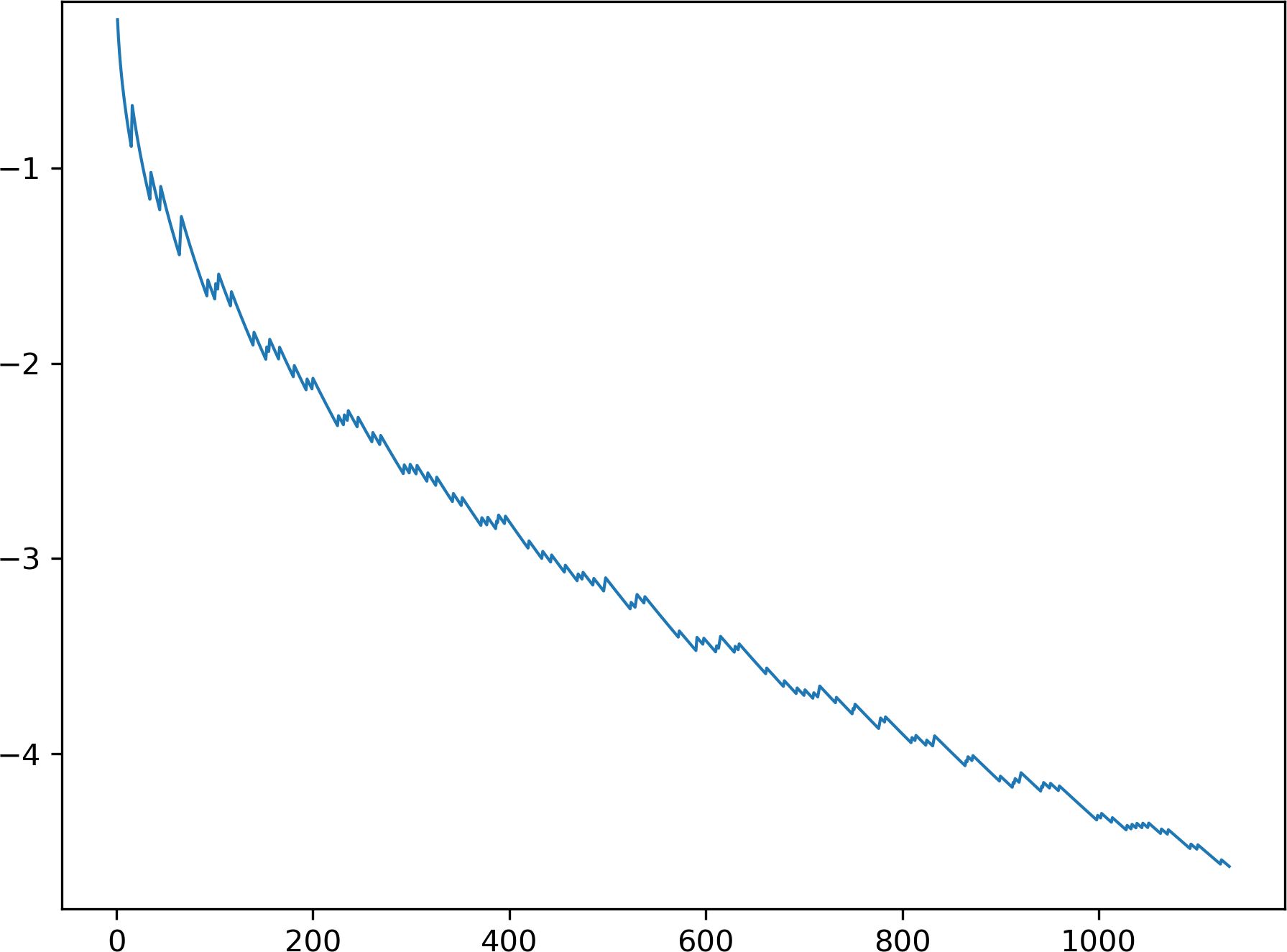}
    }
    \hfill
    \subfigure[$D_2(t)$]
    {
        \includegraphics[width=1.75in]{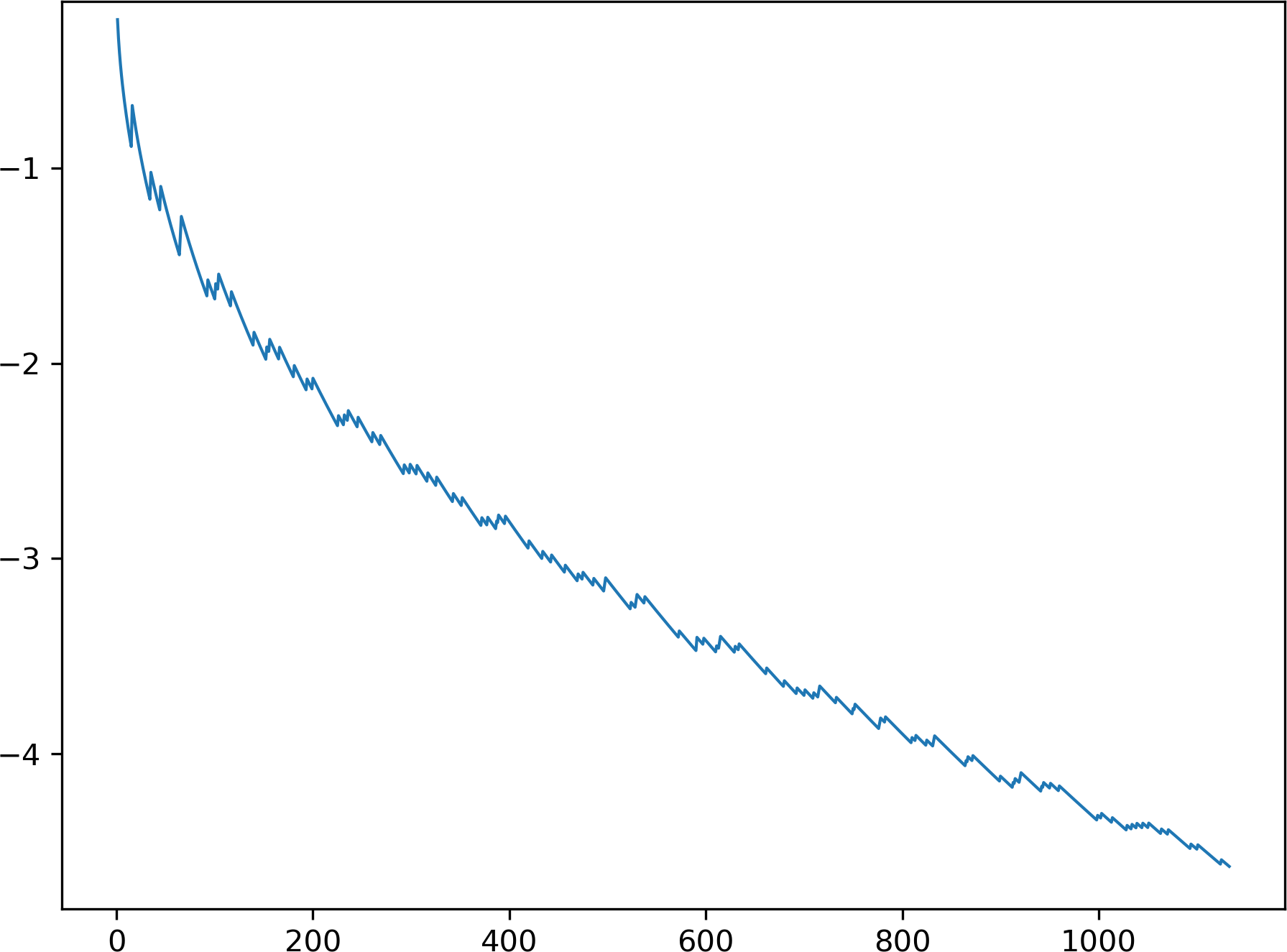}
    }
    \caption{For the Julia set to $z^2+0.2$ and Dirichlet BC}
    \label{fig:cfjd1}
\end{figure}

\begin{figure}[h!]
    \centering
    \subfigure[$N(t)$ (blue) and $\frac{A}{4\pi}t$ (orange)]
    {
        \includegraphics[width=1.75in]{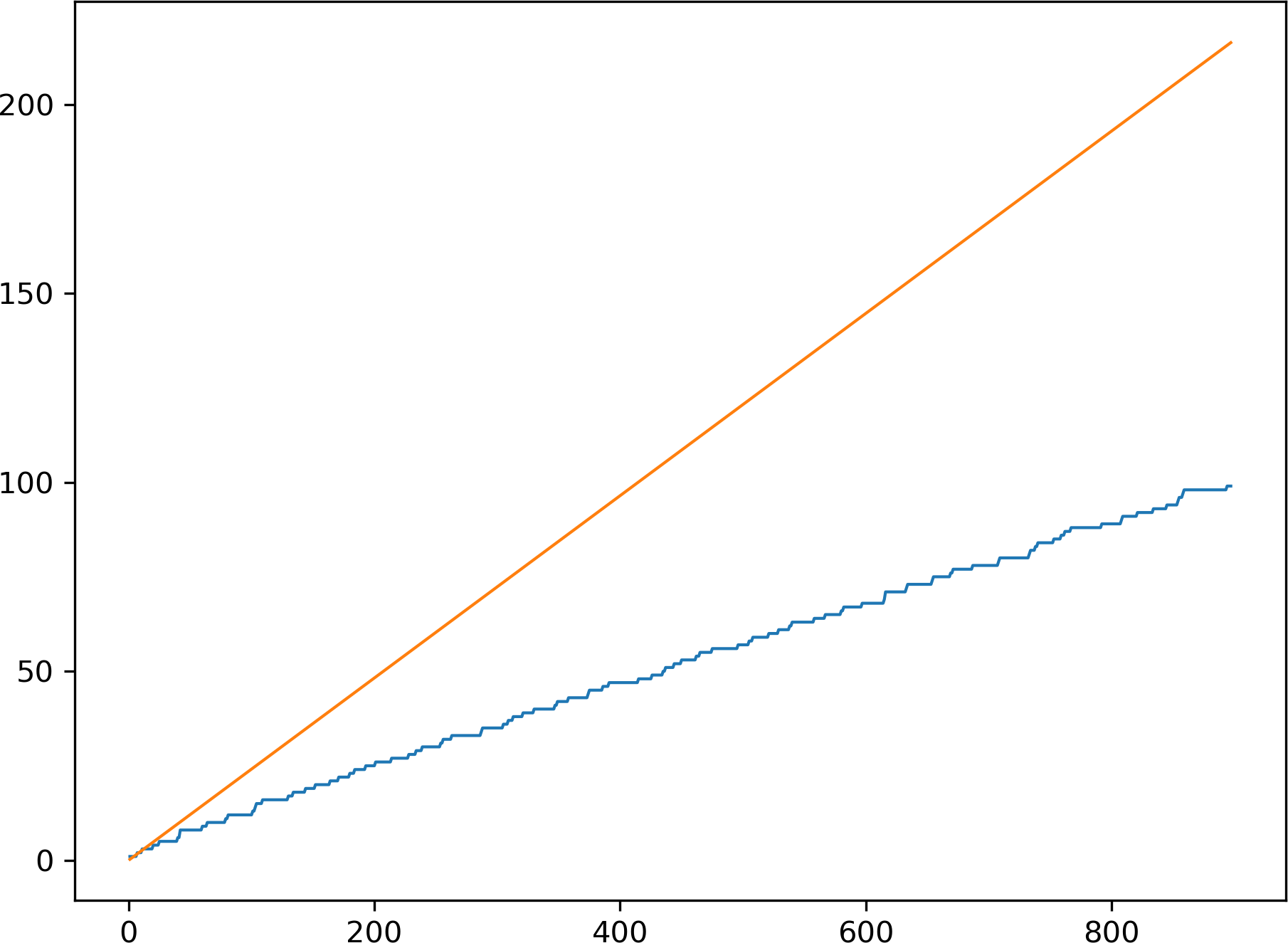}
    }
    \hfill
    \subfigure[$D_1(t)$]
    {
        \includegraphics[width=1.75in]{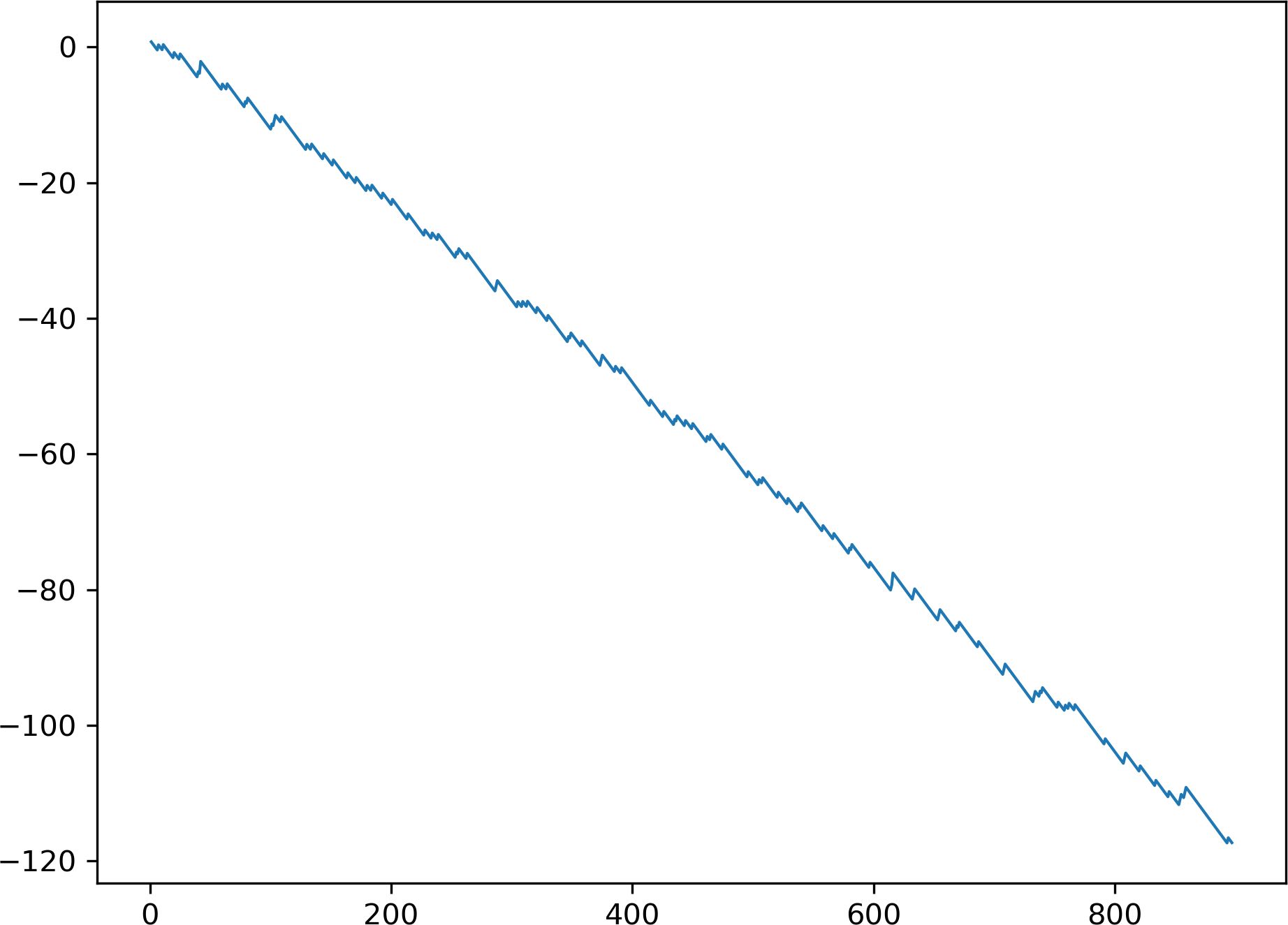}
    }
    \subfigure[$D_2(t)$]
    {
        \includegraphics[width=1.75in]{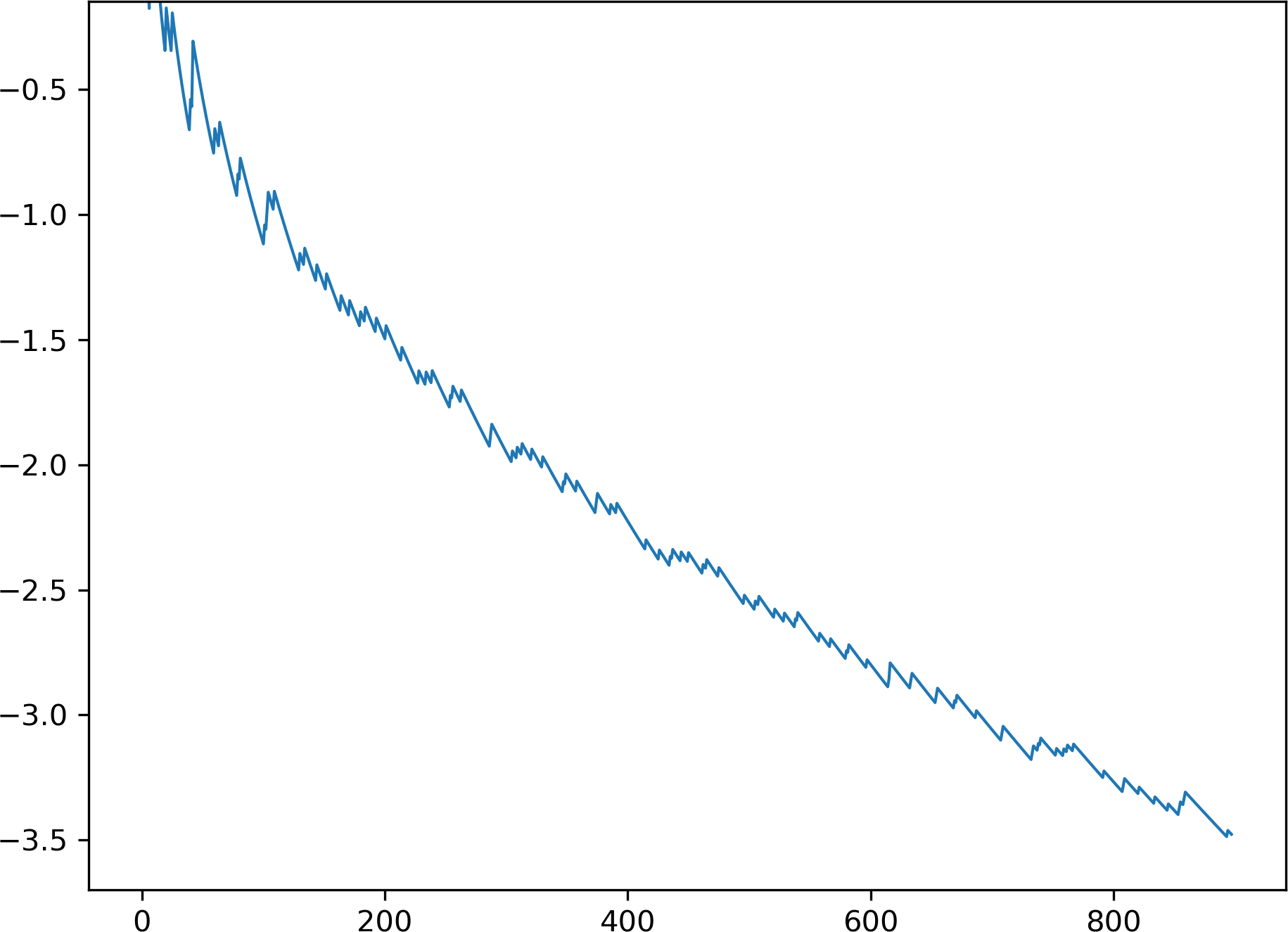}
    }
    \hfill
    \subfigure[$D_2(t)$]
    {
        \includegraphics[width=1.75in]{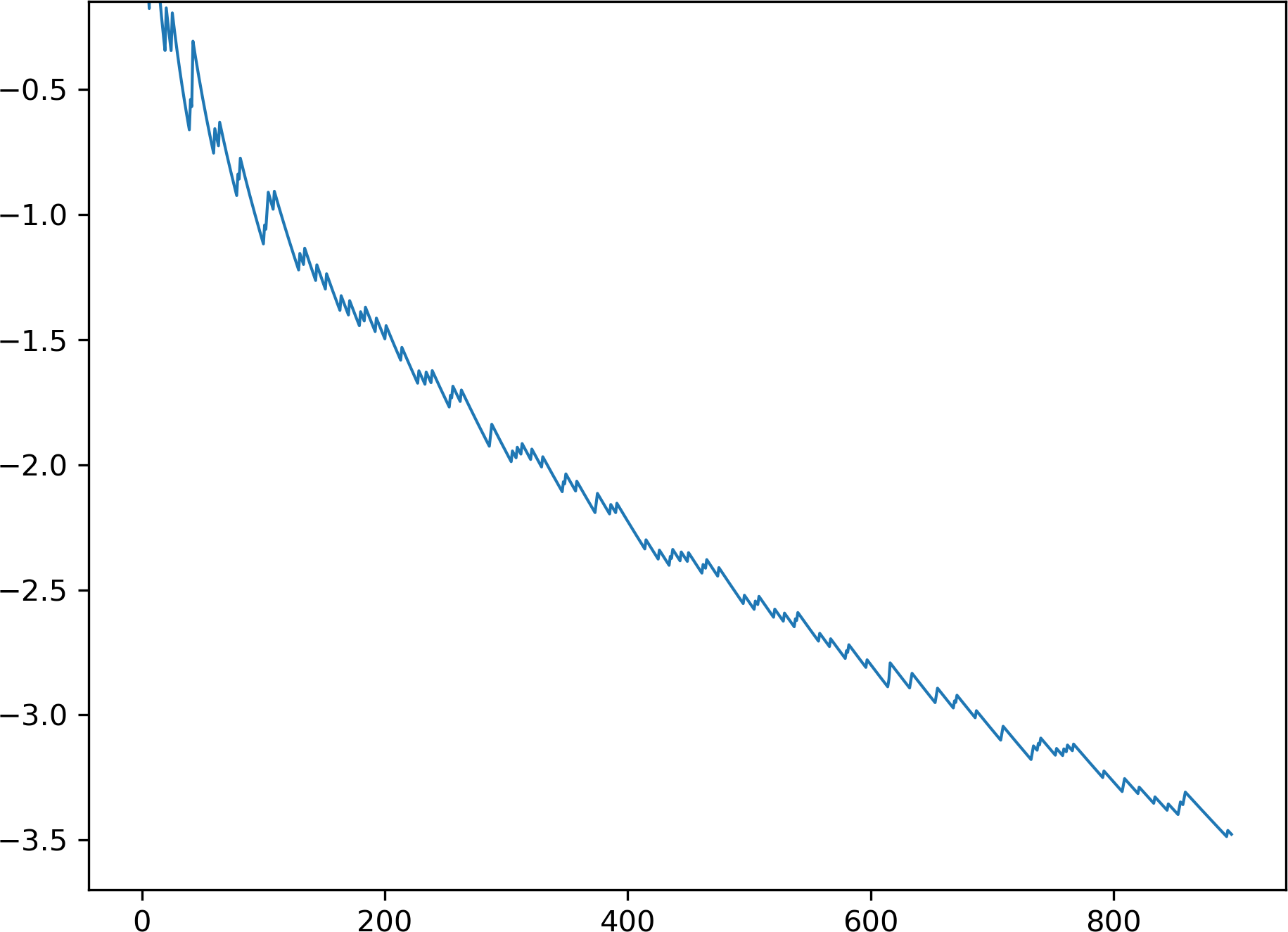}
    }
    \caption{For the Julia set to $z^2+0.2$ and Neumann BC}
    \label{fig:cfjn1}
\end{figure}

\clearpage

\section{Filled Julia set eigenfunctions.}

We show chosen eigenfunctions on four different Julia sets from the main bulb with Dirichlet and Neumann boundary conditions. Note that if $c$ is real, then the corresponding quadratic Julia set has the symmetry of an ellipse, for examples in Figures \ref{fig:js1} and \ref{fig:js3}. The eigenfunctions are symmetric with regards to at least one reflection and symmetric or skew-symmetric with respect to the other.

For non-real $c$, the Julia set has rotational symmetry and so do the eigenfunctions with Dirichlet or Neumann boundary conditions, for example in Figures \ref{fig:js2} and \ref{fig:js4}.

\vspace{0.5cm}

\begin{figure}[h]
    \centering
    \subfigure[For EV $\lambda_{11}=152.65$ with Dirichlet BC]
    {
        \includegraphics[width=1.8in]{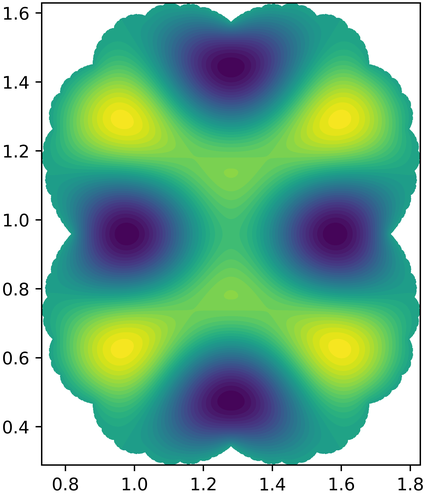}
    }
    \hfill
    \subfigure[For EV $\lambda_{18}=133.30$ with Neumann BC]
    {
        \includegraphics[width=1.8in]{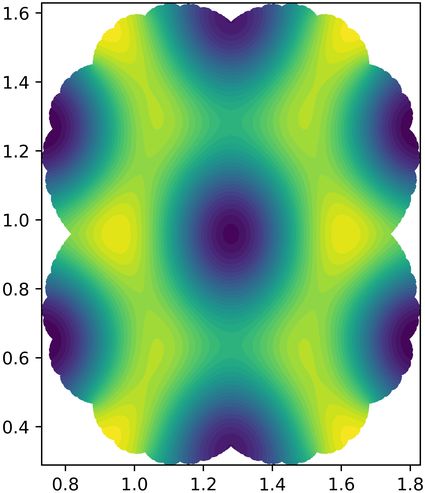}
    }
    \caption{Eigenfunctions on the filled JS to $z^2+0.2$}
    \label{fig:js1}
\end{figure}

\begin{figure}[h]
    \centering
    \subfigure[For EV $\lambda_{10}=199.32$ with Dirichlet BC]
    {
        \includegraphics[width=2in]{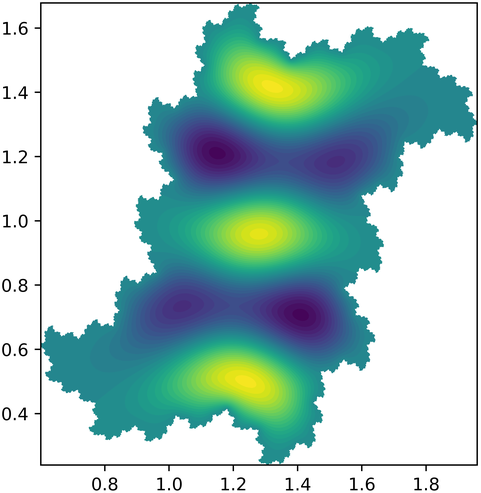}
    }
    \hfill
    \subfigure[For EV $\lambda_{10}=61.69$ with Neumann BC]
    {
        \includegraphics[width=2in]{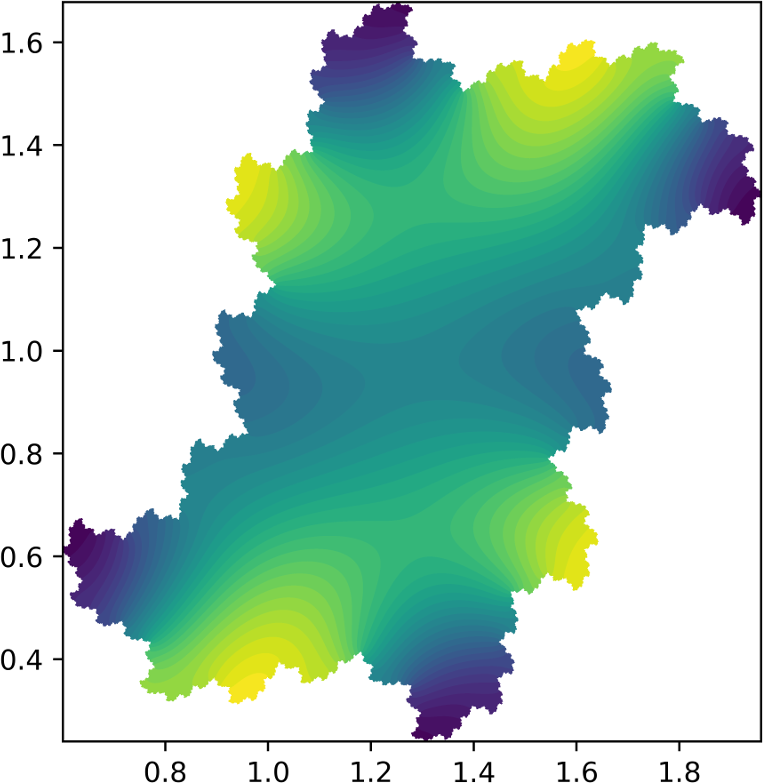}
    }
    \caption{Eigenfunctions on the filled JS to $z^2+0.2-0.45i$}
    \label{fig:js2}
\end{figure}

\begin{figure}[h]
    \centering
    \subfigure[For EV $\lambda_{9}=139.33$ with Dirichlet BC]
    {
        \includegraphics[width=2in]{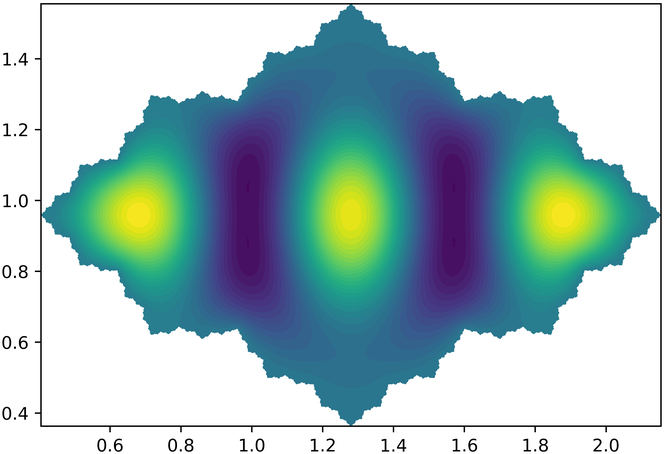}
    }
    \hfill
    \subfigure[For EV $\lambda_{13}=91.60$ with Neumann BC]
    {
        \includegraphics[width=2in]{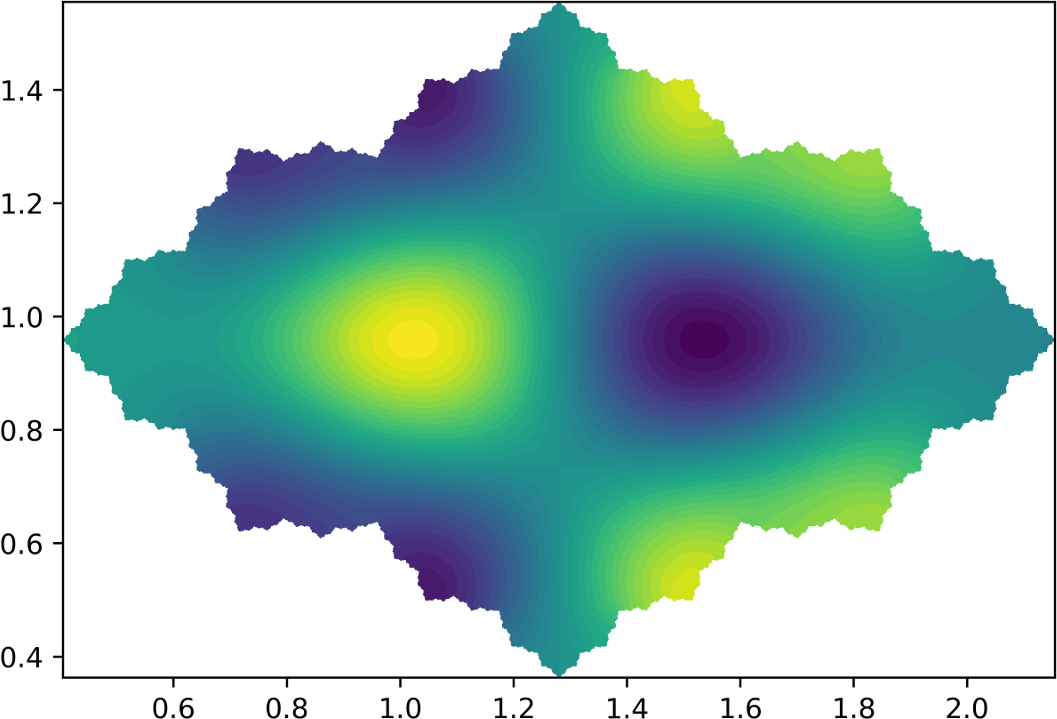}
    }
    \caption{Eigenfunctions on the filled JS to $z^2-0.5$}
    \label{fig:js3}
\end{figure}

\begin{figure}[h]
    \centering
    \subfigure[For EV $\lambda_{5}=142.09$ with Dirichlet BC]
    {
        \includegraphics[width=2in]{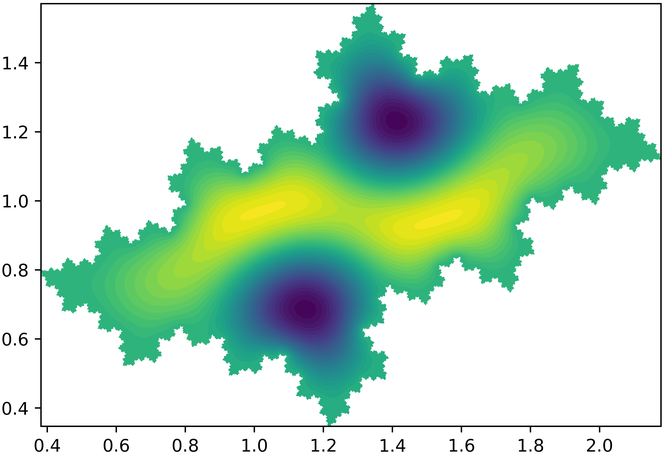}
    }
    \hfill
    \subfigure[For EV $\lambda_{7}=46.32$ with Neumann BC]
    {
        \includegraphics[width=2in]{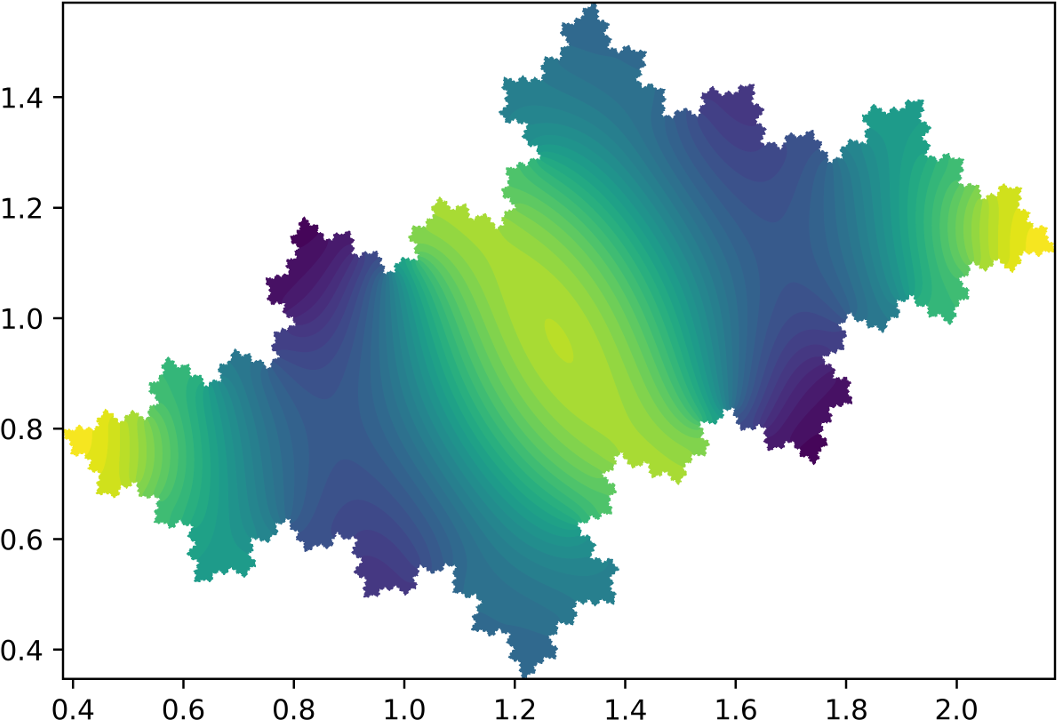}
    }
    \caption{Eigenfunctions on the filled JS to $z^2-0.5-0.45i$}
    \label{fig:js4}
\end{figure}

\clearpage

\section{Discussion.}

We have gathered a lot of experimental evidence converning the eigenvalues and eigenfunctions for snowflake domains and filled Julia sets. The next challenge is to describe properties of the spectral data and to extend the results to more general open sets with fractal boundary. Here are two interesting conjectures that arise from our data.

\begin{conjecture}
The maximal area of the filled Julia set associated with a parameter in the Mandelbrot set is $\pi$, attained by the unit disk. See Fig. \ref{fig:area}.
\end{conjecture}

\begin{conjecture}
The box-counting dimension does not exist for a typical connected Julia set. See Fig. \ref{fig:bcd1a} and Fig. \ref{fig:bcd1b}. Of course the unit circle is an exception, as may also be the case for certain named examples, such as the Basilica and Rabbit.
\end{conjecture}

\clearpage

\bibliographystyle{amsplain}

\end{document}